\documentclass[10pt]{article}
\usepackage{amssymb,amsmath,amsfonts,amsthm,mathtools,color}

\usepackage{mathrsfs}

\usepackage[title,titletoc,header]{appendix}
\usepackage{graphicx}

\usepackage{paralist}

\usepackage{tikz}
\usetikzlibrary{arrows,positioning,shapes.geometric}

\usepackage{algorithm,algorithmic}

\usepackage[labelformat=simple]{subcaption}

\graphicspath{ {./figures/} }

\usepackage{indentfirst}
\usepackage{multicol}
\usepackage{booktabs}
\usepackage{url}
\usepackage[outdir=./]{epstopdf} %to enable the use of eps

\usepackage[shortlabels]{enumitem}
\usepackage{hyperref}
\hypersetup{
    colorlinks=true, %set true if you want colored links
    linktoc=all,     %set to all if you want both sections and subsections linked
    linkcolor=blue,  %choose some color if you want links to stand out
}
\numberwithin{equation}{section}

\newtheorem{Theorem}{Theorem}[section]
\newtheorem{Lemma}[Theorem]{Lemma}
\newtheorem{Proposition}[Theorem]{Proposition}
\newtheorem{Assumption}{H.\!\!}

\theoremstyle{definition}
\newtheorem{Definition}{Definition}[section]

\newtheorem{Example}{Example}[section]

\theoremstyle{remark}
\newtheorem{Remark}{Remark}[section]

 \def\p{\partial} \def\nb{\nonumber}
\def\to{\rightarrow}
 \def\ol{\overline}    
\def\Om{\Omega}  \def\om{\omega} %\def\I{ {\rm (I) } }
 
\newcommand{\q}{\quad}

\def\l{\label}  \def\f{\frac}  \def\fa{\forall}
\def\b{\beta}  \def\a{\alpha} 
 
\def\eps{\varepsilon}

 \def\t{\times}  
\def\ms{\medskip}

\def \la{\langle} \def\ra{\rangle}

\def\cA{\mathcal{A}}
\def\cB{\mathcal{B}}

\def\cE{\mathcal{E}}
\def\cF{\mathcal{F}}

\def\cH{\mathcal{H}}
\def\cI{\mathcal{I}}

\def\cK{\mathcal{K}}

\def\cO{\mathcal{O}}

\def\cR{\mathcal{R}}
\def\cS{\mathcal{S}}

\def\cV{\mathcal{V}}

\def\d{{\mathrm{d}}}

\def\bA{{\textbf{A}}}

\def\sD{\mathbb{D}}
\def\sE{{\mathbb{E}}}
\def\sF{{\mathbb{F}}}
\def\sI{{\mathbb{I}}}
\def\sN{{\mathbb{N}}}
\def\sP{\mathbb{P}}

\def\sR{{\mathbb R}}
\def\sS{{\mathbb{S}}}

\def\trans{\mathsf{T}}

\newcommand{\subW}{\mathrm{subW}}

\newcommand{\tr}{\textnormal{tr}}

\DeclareMathOperator*{\argmax}{arg\,max}
\DeclareMathOperator*{\argmin}{arg\,min}
\DeclareMathOperator*{\dom}{dom}
\DeclareMathOperator*{\esssup}{ess\,sup}

\newcommand{\lc}
{\mathrel{\raise2pt\hbox{${\mathop<\limits_{\raise1pt\hbox
{\mbox{$\sim$}}}}$}}}

\newcommand{\gc}
{\mathrel{\raise2pt\hbox{${\mathop>\limits_{\raise1pt\hbox{\mbox{$\sim$}}}}$}}}

\newcommand{\ec}
{\mathrel{\raise2pt\hbox{${\mathop=\limits_{\raise1pt\hbox{\mbox{$\sim$}}}}$}}}

\def\bb{\begin{equation}} \def\ee{\end{equation}}
\def\bbn{\begin{equation*}} \def\een{\end{equation*}}

\def\beqn{\begin{eqnarray}}  \def\eqn{\end{eqnarray}}

\def\beqnx{\begin{eqnarray*}} \def\eqnx{\end{eqnarray*}}

\def\bn{\begin{enumerate}} \def\en{\end{enumerate}}

\def\bd{\begin{description}} \def\ed{\end{description}}



\begin{document}

\title{
Reinforcement learning
for linear-convex models with jumps
via stability analysis of  feedback controls
}
\author{
\and Xin Guo\thanks{Department of Industrial Engineering and Operations Research, University of California, Berkeley, USA. \textbf{Email}: \texttt{xinguo@berkeley.edu}}
\and Anran Hu\thanks{Department of Industrial Engineering and Operations Research, University of California, Berkeley, USA. \textbf{Email}: \texttt{anran\_hu@berkeley.edu}}
\and Yufei Zhang\thanks{Mathematical Institute, University of Oxford, UK. \textbf{Email}: \texttt{yufei.zhang@maths.ox.ac.uk}}
}
\date{}

\maketitle

%\tableofcontents

\noindent\textbf{Abstract.} 
We study 
 finite-time horizon
 continuous-time linear-convex  reinforcement learning
problems in an episodic setting. In this problem,
the unknown
 linear jump-diffusion process
 is controlled subject to
 nonsmooth
 convex  costs.
We show that
the associated linear-convex control problems
admit  Lipchitz continuous optimal feedback controls
and further prove the Lipschitz stability of the feedback controls,
i.e.,  the performance gap
between applying feedback controls 
   for an incorrect model
 and for the true model
  depends Lipschitz-continuously on the magnitude of perturbations in the model coefficients;
%   even with lower semicontinuous   costs;
 the proof relies on a stability analysis of the associated forward-backward stochastic differential equation.
 We then 
 propose a %novel 
  least-squares algorithm 
which 
 achieves  a regret of the order $\cO(\sqrt{N\ln N})$
on  linear-convex learning problems
with jumps,
where $N$ is the number of learning episodes; 
the analysis leverages 
the Lipschitz stability of  feedback controls
and concentration properties  of sub-Weibull random variables. 
{%\color{blue} 
Numerical experiment confirms  the convergence and the robustness of the proposed algorithm.}


\ms

%	93E35  	Stochastic learning and adaptive control
% 	62G35  	Nonparametric robustness
%	93E24  	Least squares and related methods for stochastic control systems
%	68Q32  	Computational learning theory 

\medskip

\section{Introduction}
Reinforcement learning (RL) seeks optimal strategies to control an unknown dynamical system 
by interacting with the random environment  through exploration and exploitation \cite{sutton2018reinforcement}.
This paper studies a reinforcement learning problem for controlled linear-convex models with  unknown drift parameters.
The controlled dynamics are  with possible jumps,  the objectives are  extended real-valued 
nonsmooth convex  functions, and the learning is in an episodic setting for a finite-time horizon.

\paragraph {Regret analysis of RL algorithm and stability of controls.} 
  RL algorithms are in general characterized by iterations of 
exploitation  and exploration 
(see e.g.~\cite{abbasi2011regret,mania2019certainty,basei2020logarithmic}). In the model-based approach, for instance,  the 
agent interacts with the environment via  policies 
based on the present estimation of 
the unknown model parameters,
and then
 incorporates  the responses of these interactions
to improve   their knowledge of the system. 
One of the main performance criteria for RL algorithm, called {\it regret}, is to 
measure its deviation from the optimality
 over the learning process. 
 
One key component in regret analysis is   the Lipschitz stability of feedback controls 
which quantifies the mismatch between the assumed and actual models,  or the stability of controls  with respect to model perturbations.  It is to analyze
 the precise derivation 
of a pre-computed feedback  control 
from   the optimal one, and is also  known as the robustness of control policies in the learning community
\cite{mania2019certainty, basei2020logarithmic,
 kara2020robustness}).

Despite the long history of stability of  controls in the control literature, its main focus in classical control theory 
has been restricted to the continuity of  value functions and optimal open-loop controls 
(see e.g.~\cite{ aldous1981weak,yong1999stochastic, backhoff2017sensitivity, bayraktar2020continuity,kara2020robustness}).
Studies of  high-order stability of controls such as the Lipschitz stability, has only attracted attention very recently, largely due to  its crucial importance in characterizing 
the precise %\textit{non-asymptotic} behaviour 
{%\color{blue}
\textit{regret order}} of learning algorithms
(see \cite{mania2019certainty,  
 basei2020logarithmic, reisinger2021regularity}).  
Analyzing Lipschitz stability of feedback control is technically more challenging. It requires analyzing  
the derivatives of the value function
in a suitable function space, as  optimal feedback controls are usually  characterized via  
 the derivatives of the 
 value function. 
 
  %whose stability  requires
% regularity estimates of the associated nonlinear 
%Hamilton-Jacobi-Bellman  partial differential equation .
Due to this  technical difficulty,  
most existing works on
%non-asymptotic performance guarantees 
{%\color{blue}
regret analysis} of RL algorithms 
concentrate on 
the linear-quadratic (LQ) control framework.
In this special setting, 
the  optimal feedback control
is  an affine function of state variables,
whose  coefficients satisfy  an associated 
algebraic or ordinary
Riccati equation.
Consequently, 
the Lipschitz stability of feedback controls is simplified by analyzing the robustness of the
Riccati equation  (see e.g.~\cite{abbasi2011regret,mania2019certainty,%basei2020linear
 basei2020logarithmic}). 
Unfortunately, 
these techniques  developed specifically for Riccati  equations in LQ-RL problems are clearly not 
applicable  for 
general RL problems (see 
e.g.~\cite{bensoussan1984stochastic, clason2016convex,exarchos2018stochastic,li2019sparse}).
In particular, optimal policies  are typically nonlinear 
in the state variable,  especially with the  inclusion of  entropy regularization 
for  the exploration strategy in the optimization objective
(see e.g.~\cite{wang2019exploration,guo2020entropy,vsivska2020gradient,reisinger2021regularity}).

\paragraph{Our work.}
This paper consists of three parts. 
 \begin{itemize}[leftmargin=*]
\item It first establishes the Lipschitz stability for   finite-time horizon linear-convex control problems,
%with respect to parameter perturbations in  the underlying model,
whose dynamics are linear jump-diffusion processes 
with controlled drifts and 
possibly degenerate additive noises, 
and  objectives are 
 extended real-valued 
 lower semicontinuous convex functions.
Such control problems include as special cases
LQ  problems with convex control constraints,
sparse and switching control of linear systems,
and entropy-regularized relaxed control problems
(see Examples \ref{example:sparse} and \ref{example:relax}).
It  shows that
these control problems
admit  Lipchitz continuous optimal feedback controls
with linear growth in the spatial variables
(Theorem \ref{thm:lc_fb}).
It further proves that the performance gap
between applying feedback controls 
   for an incorrect model
 and for the true model
  depends Lipschitz-continuously on the magnitude of perturbations in the model coefficients,
   even with lower semicontinuous   cost functions (Theorem \ref{thm:dynamics_stable}).
   The Lipschitz stability of feedback controls is 
  extended to entropy-regularized control problems with controlled diffusion in Proposition \ref{prop:stability_entropy}. 
%To the best of our knowledge, 
%this is the first paper 
%on the Lipschitz stability
%of feedback controls 
%for  
% general continuous-time stochastic control problems with
% linear state dynamics and 
%nonsmooth and unbounded cost functions, and with possible jumps.

\item
It then proposes a greedy least-squares (GLS) algorithm
for a class of continuous-time 
linear-convex RL problems 
in an episodic setting. 
At each iteration, 
the GLS algorithm estimates 
the unknown drift parameters 
by  a regularized least-squares estimator
based on  observed trajectories, 
and then designs a feedback control for  
 the estimated model.
% Based on the Lipschitz stability of LC control problems,
 It establishes that the  regret of  this GLS algorithm
 is sublinear, i.e., of the magnitude $\cO(\sqrt{N\ln N})$ 
 with $N$ being 
 the number of learning episodes,
provided that the    least-squares estimator 
satisfies a general concentration inequality
(Theorem \ref{thm:regret}). 
It  further characterizes the explicit concentration behaviour of 
the   least-squares estimator 
(and hence the precise regret bound of the GLS algorithm),
depending on  tail behaviours of the random jumps in the state dynamics
(Theorem \ref{cor:regret_jump}). 
In  the pure diffusion case, 
 a sharper 
 regret bound has been obtained 
% than   that in 
% \cite[Theorem 2]{%basei2020linear
% basei2021logarithmic}
  (Theorem \ref{cor:regret_diffusion}). 
%To the best of our knowledge,
%this is the first 
% sublinear regret bound
%for continuous-time RL problems 
%beyond the LQ setting
%or 
%for RL problems with random jumps.
\item {%\color{blue} 
It finally verifies the theoretical properties of the proposed GLS algorithm through numerical experiment on a  three-dimensional LQ RL problem. It shows  
the convergence of the
least-squares estimations to the true parameters as the number of episodes increases, as well as a sublinear 
regret as indicated in 
% the regret grows sub-linearly, which are consistent with 
 theoretical results. 
It also demonstrates the GLS algorithm
is robust 
with respect to  initializations.}

\end{itemize}

\paragraph{Our approaches and related works.}
%\ \hspace{-3.5mm}
{%\color{blue}
Optimal control of stochastic systems with parametric uncertainty has been  studied in the classical adaptive control literature (see \cite{duncan1999adaptive, sastry2011adaptive,
ioannou2012robust,aastrom2013adaptive}), where   stationary policy is constructed to minimize the long term average cost and where the \textit{asymptotic} stability and convergence of an adaptive control law
is analyzed when the time horizon goes to infinity.
However, research on  rate of convergence is virtually non-existent. 
The problem studied here is different. The main objective is 
 to construct
optimal (and time-dependent) policies for {\it finite-horizon}  problems, with the {%\color{blue}
{\it finite-sample regret}} analysis for the learning algorithm.
Compared with  the  classical adaptive control literature,
the {%\color{blue} 
regret analysis in this work, also known as the non-asymptotic
performance
analysis,} 
requires  novel  techniques, consisting of
  a precise performance estimate of  a greedy policy 
 (namely the Lipschitz stability of feedback controls)
and a finite-sample analysis of the  parameter estimation scheme.
}

Analyzing the Lipschitz stability of feedback controls 
in a continuous-time setting
requires quantifying 
the impact of parameter uncertainty on the derivatives of the value functions.
  \cite{reisinger2021regularity}
studies the so-called exit time problem and  the Lipschitz stability of regularized  relaxed controls of diffusion processes
via 
a  partial differential equation (PDE) approach,
which assumes that 
the diffusion coefficients are non-degenerate and 
 the state process takes values in a compact set. 
 In contrast, we consider (see
 Section \ref{sec:lipschitz_stable})
 unconstrained jump-diffusion process
 with unbounded drift 
 and (uncontrolled) degenerate noise, 
 and the cost functions are 
 nonsmooth and unbounded. 
 Consequently, the PDE approach requires to 
deal with 
a  {\it degenerate nonlocal}     PDE
with  %concave and 
non-Lipschitz nonlinearity, 
whose solution (i.e., the value function) 
is unbounded and may be nonsmooth  due to the lack of regularization from the
Laplacian operator.
Here the Lipschitz stability of feedback controls 
is established by analyzing the stability  of the associated coupled forward-backward stochastic differential equations (FBSDEs). This is possible by a) first exploiting  the linear-convex structure of the control problem, which enables
constructing 
%{an optimal}
a Lipschitz continuous 
feedback control 
via solutions of
 coupled   FBSDEs, and then b) by extending the stochastic maximum principle in \cite{tang1994necessary} 
to feedback controls with nonsmooth costs. 
To the best of our knowledge, this is the  
first time FBSDE has been used to study   stability of \textit{feedback} controls.

Analyzing the (finite-sample) accuracy
of the least-squares estimator for jump-diffusion models
involves integrations of 
the state and control processes with respect to Brownian motions and Poisson random measures.
Now, the nonlinearity of feedback controls renders it impossible to analyze the tail behaviour of these stochastic integrals as \cite{guo2020entropy}
does for LQ problems with analytical solutions;  
Additionally, the presence of random jumps implies that the state process is no longer sub-Gaussian,
and hence the   stochastic integrals in the least-squares estimator  no longer sub-exponential.
To overcome these difficulties,  a 
 convex concentration inequality is employed for SDEs with jumps
  \cite{ma2010transportation}, along with 
% which enables us to show that the state and %control processes are 
% sub-Gaussian for the pure diffusion case,
% and 
%  sub-exponential if
% the random jumps are  sub-exponential.
  Burkholder's inequality   
 and the Girsanov theorem
 to characterize  precisely the sub-Weibull behaviour of the required stochastic integrals 
 in terms of their Orlicz norms
 (Lemmas \ref{lemma:stochastic_integral_jump} and \ref{lemma:stochatic_integral_diffusion}).
 Leveraging recent developments in the
theory of sub-Weibull random variables,
 the precise parameter estimation error of the least-squares estimator is quantified 
 in terms of the sample size.
 
 {%\color{blue}
 It is worth pointing out that the  stability analysis of feedback controls can be extended (see Section \ref{sec:ctrl_diffusion})   
to entropy-regularized control problems with
controlled diffusion and without the linear-convex structure. Instead of the maximum principle for the linear-convex setting, regularity analysis of the associated fully-nonlinear parabolic PDEs may be  needed for nondegenerate noise
 with  regular (such as bounded and high-order differentiable) coefficients. (See the discussion after Proposition 
  \ref{prop:stability_entropy} for more details). 
}

\paragraph{Notation.}
%\textbf{Notation.}
%We end this section by introducing  some  spaces  used throughout  this paper.
For each 
$T>0$,
 filtered probability space
$(\Om,\cF,\sF=\{\cF_t\}_{t\in[0,T]},\sP)$  satisfying the usual condition
 and
  Euclidean space $(E,|\cdot|)$,
we introduce the following spaces: 
 % $L^p(0,T)$ 
 \begin{itemize}[leftmargin=*,noitemsep,topsep=0pt]
\item 
  $L^p(0,T;E)$, $p\in[ 2,\infty]$,
 is the space of  
  (Borel) 
  measurable functions $\phi:[0,T]\to E$ satisfying 
  $\|\phi\|_{L^p}=(\int_0^T|\phi_t|^p\,\d t)^{1/p}<\infty$
  if $p\in [2,\infty)$ and 
    $\|\phi\|_{L^\infty}=\esssup_{t\in[0,T]} |\phi_t|<\infty$
  if $p=\infty$;
  \item
$L^2(\Om;E)$ is the space of $E$-valued $\cF$-measurable random variables $X$ satisfying 
$\|X\|_{L^2}=\sE[|X|^2]^{1/2}<\infty$;
\item
 $\cS^2(t,T;E)$, $t \in [0,T]$,
  is the space of 
 $E$-valued
$\sF$-progressively  measurable c\`{a}dl\`{a}g 
processes
$Y: \Om\t [t,T]\to E$ %$Y: [t,T]\t \Om\to E$ 
satisfying $\|Y\|_{\cS^2}=\sE[\sup_{s\in [t,T]}|Y_s|^2]^{1/2}<\infty$;
\item
 $\cH^2(t,T;E)$, $t \in [0,T]$, is the space of 
   $E$-valued $\sF$-progressively measurable
 processes 
$Z: \Om\t [t,T]\to E$ 
 satisfying $\|Z\|_{\cH^2}=\sE[\int_t^T|Z_s|^2\,\d s]^{1/2}<\infty$;
 \item
$\cH^2_\nu(t,T;E)$, $t \in [0,T]$, is the space of 
   $E$-valued $\sF$-progressively measurable
 processes 
$M: \Om\t [t,T]\t\sR^p_{0}\to E$ 
 satisfying $\|M\|_{\cH^2_\nu}=\sE[\int_t^T\int_{\sR^p_0}|M_s(u)|^2\nu(\d u)\,\d s]^{1/2}<\infty$,
 where
 $\sR^p_0\coloneqq \sR^p\setminus\{0\}$ and 
  $\nu$ is a $\sigma$-finite measure on  $\sR^p_0$.
\end{itemize}
For notational simplicity, 
%\textcolor{red}{denote}
we denote
$\cS^2(E)=\cS^2(0,T;E)$, $\cH^2(E)=\cH^2(0,T;E)$ and 
$\cH^2_\nu(E)=\cH^2_\nu(0,T;E)$. 
We shall also  denote by $\la \cdot,\cdot\ra$
 the usual inner product in a given Euclidean space,
  by $|\cdot|$ the norm induced by $\la \cdot,\cdot\ra$,
 by $A^\trans$ the transpose of a matrix $A$,
 and     by $C\in [0,\infty)$ 
a generic constant, which   depends only on the constants appearing in the assumptions and may take a different value at each occurrence.

\section{Lipschitz stability of %optimal feedback controls
linear-convex control problems}
\l{sec:lipschitz_stable}

\subsection{Problem formulation with nonsmooth costs}
\l{sec:lc_ctrl}

 In this section, we introduce the   linear-convex control problems
 with nonsmooth costs.
 
Let 
 $T>0$ be a given terminal time
 and 
$(\Om,\cF,\sP)$ be  a complete probability space, in which
two mutually independent processes, a $d$-dimensional Brownian motion $W$ and a Poisson random
measure $N(\d t,\d u)$ with compensator $\nu(\d u)\d t$, are defined. 
We assume that $\nu$ is a $\sigma$-finite measure on
$\sR^p_0$ equipped with its Borel field 
$\cB(\sR^p_0)$ 
and satisfies $\int_{\sR^p_0}\min(1,|u|^2)\, \nu(\d u)<\infty$.
We denote by $\tilde{N}(\d t,\d u)=N(\d t,\d u)-\nu(\d u)\d t$ the  compensated process of $N$ 
and by $\sF=(\cF_t)_{t\in [0,T]}$ the  filtration
generated by $W$ and $N$ and augmented by the $\sP$-null sets.

 % We start by introducing the control problem of interest. 
For any given 
%$\theta=(A,B)\in\sR^{n\t (n+k)}$ be a given parameter,
%and 
  initial state $x_0\in \sR^n$,
 we consider the following 
minimization problem
\bb\l{eq:lc}
V(x_0)=\inf_{\a\in \cH^2(\sR^k)} J(\a;x_0),
\q \textnormal{with}\q 
J(\a;x_0)=\sE\left[\int_0^T f(t,X^{x_0,\a}_t,\a_t)\, \d t+g(X_T^{x_0,\a})\right],
\ee
where for each $\a\in \cH^2(\sR^k)$, the process $X^{x_0,\a}$ satisfies the following controlled dynamics: 
\bb\l{eq:lc_sde}
\d X_t =b(t,X_t,\a_t)\,\d t+\sigma(t)\, \d W_t
+
\int_{\sR^p_0}
\gamma(t,u)\, \tilde{N}(\d t,\d u)
, \q t\in [0,T],
\q X_0=x_0,
\ee
where
$b$, $\sigma$,
$\gamma$,
$f$ and $g$ are given functions satisfying the following conditions:

 \begin{Assumption}\l{assum:lc_ns}
$b:[0,T]\t\sR^n\t \sR^k\to \sR^n$,
$\sigma:[0,T]\to \sR^{n\t d}$,
%$\sigma: \sR^n\t \sR^n\t \sR^k\to \sR^{n\t d}$,
${\gamma}:[0,T]\t \sR^p_0\to \sR^{n}$,
$f:[0,T]\t \sR^n\t \sR^k\to\sR\cup \{\infty\}$, $g:\sR^n\to\sR$ 
are measurable functions
 such that for some  $L\ge 0$ and $\lambda>0$,

\begin{enumerate}[(1)]
\item \l{item:lc_linear}
there exist measurable functions 
$(b_0,b_1,b_2):[0,T]\to \sR^{n}\t \sR^{n\t n}\t \sR^{n\t k}$ 
such that
$b(t,x,a)=b_0(t)+b_1(t)x+b_2(t)a$
 for all $(t,x,a)\in [0,T]\t  \sR^n\t \sR^k$,
with
$
\|b_0\|_{L^2}+\|b_1\|_{L^\infty}+\|b_2\|_{L^\infty}+\|\sigma\|_{L^2}
+\big(\int_0^T\int_{\sR^p_0} |{\gamma}(t,u)|^2\,\nu(\d u)\d t\big)^{1/2}
\le L$.

\item\l{item:lc_g}
 $g$ is  convex and differentiable with
an $L$-Lipschitz derivative
such that 
$|\nabla g(0)|\le L$.

\item\l{item:f0R}
there exist functions 
 $f_0:[0,T]\t \sR^n\t \sR^k\to\sR$ 
and 
$\cR:  \sR^k\to \sR\cup \{\infty\}$
such that 
$$
f(t,x,a)={f}_0(t,x,a)+\cR( a), \q  \fa (t,x,a)\in [0,T]\t \sR^n\t \sR^k.
$$
For all $(t,x)\in [0,T]\t \sR^n$,
$f_0(t,x,\cdot)$ is convex in $\sR^k$,
$f_0(t,\cdot,\cdot)$ 
is 
differentiable in $\sR^n\t \sR^k$
with 
an $L$-Lipschitz derivative,
and 
$|f_0(t,0,0)|+|\p_{(x,a)}f_0(t,0,0)|\le L$.
Moreover, 
$\cR$ is proper, lower semicontinuous, and convex.
\footnotemark
\footnotetext{We say a function $\cR:\sR^k\to \sR\cup\{\infty\}$ is proper if it has a nonempty effective domain 
$\operatorname{dom} \cR\coloneqq \{a\in \sR^k \mid \cR(a)<\infty\}$.}

\item \l{item:lc_f}
for all 
$t\in [0,T]$, $(x,a),(x',a')\in \sR^n\t \sR^k$, and $\eta\in [0,1]$,
\bb\l{eq:strong_convex}
\eta f(t,x,a)+(1-\eta ) f(t,x',a')
\ge 
f(t,\eta x+(1-\eta) x', \eta a+(1-\eta) a')
+ \eta(1-\eta)\tfrac{\lambda}{2}|a-a'|^2.
\ee

\end{enumerate}
\end{Assumption}

\begin{Remark}\l{rmk:integrand}
Throughout this paper, 
let $\dom \cR=\{a\in \sR^k\mid \cR(a)<\infty\}$
be the effective domain of $\cR$ (or equivalently
the effective domain of $f$). 
Under 
 (H.\ref{assum:lc_ns}),
we can show  that 
both the function 
$f$ 
and its 
conjugate function
\bb\l{eq:conjugate}
[0,T]\t \sR^n\t \sR^k\ni (t,x,z)\mapsto f^*(t,x,z)\coloneqq \sup\{\la a,z\ra-f(t,x,a)\mid a\in \sR^k\}\in \sR\cup\{\infty\}
\ee
are 
normal convex integrands in the sense of
  \cite[Section 14]{rockafellar2009variational}
  and hence  measurable,
which are    crucial for the 
well-definedness of the control problem \eqref{eq:lc}
and the
characterization of optimal controls.
Furthermore,
the strong convexity condition (H.\ref{assum:lc_ns}\ref{item:lc_f}) enables us to establish the Lipschitz stability of 
feedback controls to \eqref{eq:lc},
which is essential for the analysis of learning algorithms.

Our analysis and results can be extended to  control problems with 
time-space dependent nonsmooth cost function
$\cR: [0,T]\t \sR^n\t \sR^k\to  \sR\cup \{\infty\}$ by assuming 
$\cR$ is a normal convex integrand and satisfies 
suitable subdifferentability conditions.
For notational simplicity and clarity, we choose to refrain from 
further generalization.
%providing this level of generality without the motivation from specific applications.

\end{Remark}

Note that (H.\ref{assum:lc_ns})
allows the diffusion coefficient $\sigma$ to be degenerate,
hence the stability results in {Section \ref{sec:robust_fb}}
apply to deterministic control problems. 
Moreover, (H.\ref{assum:lc_ns}) requires neither
 the effective domain $\dom \cR$  
to be closed 
nor 
the  function $\cR$ to be bounded or continuous on $\dom \cR$,
which is important for problems in engineering and machine learning,
as shown in the following examples.
 
\begin{Example}[Sparse and switching controls]
\l{example:sparse}
Let 
 $\bA\subset \sR^k$ be a nonempty closed convex set, 
$\delta_\bA$ be the indicator of $\bA$ satisfying 
$\delta_\bA(x)=0$ for $x\in \bA$ and 
$\delta_\bA(x)=\infty$ for $x\in \sR^k\setminus\bA$,
and 
$\ell:\sR^k\to \sR$ be a lower semicontinuous and convex function.  
Then $\cR\coloneqq \ell+\delta_\bA$ satisfies 
(H.\ref{assum:lc_ns}\ref{item:f0R}).
In particular, 
by setting $\ell\equiv 0$, we can consider the linear-convex control problems 
with smooth running costs and control constraints 
(see e.g.~\cite{bensoussan1984stochastic} and \cite[Theorem 5.2 on p.~137]{yong1999stochastic}),
which include   the most commonly used linear-quadratic models as special cases.

More importantly, it is well-known in optimal control literature 
(see e.g.~\cite{clason2016convex,exarchos2018stochastic,li2019sparse} and references therein)
 that, 
one can employ a nonsmooth function $\ell$ involving $L^1$-norm of controls 
to enhance the sparsity and  switching property of 
optimal controls,
which are practically  important for 
mimimum fuel problems 
and optimal device placement problems. 
Here by 
sparsity we  refer to the situation where  the whole vector $\a_t$ is
zero, while by switching control we refer to the phenomena where 
at most one  coordinate of $\a_t$ is non-zero at each $t$.

\end{Example}

\begin{Example}[Regularized relaxed  controls]
\l{example:relax}
Consider a  regularized control problem
arising from reinforcement learning
(see e.g.~\cite{wang2019exploration,guo2020entropy,vsivska2020gradient,reisinger2021regularity}),
 whose cost function $f$ is of the following form:
\bb\l{eq:f_relaxed}
 f(t,x,a)={f}_0(t,x)+\la {f}_1(t,x),a\ra +\rho D_ \mathfrak{f}(a || \mu)
\q  \fa (t,x,a)\in [0,T]\t \sR^n\t \sR^k,
\ee
where  $f_0:[0,T]\t \sR^n\to \sR,f_1:[0,T]\t \sR^n\to \sR^k$ are given functions,
$\rho>0 $ is a regularization parameter, 
and 
$D_ \mathfrak{f}(\cdot || \mu):\sR^k\to \sR\cup \{\infty\}$ is 
an $\mathfrak{f}$-divergence defined as follows. 
Let  $\Delta_k\coloneqq\{
a\in [0,1]^k\mid 
\sum_{i=1}^k a_i=1\}$,
  ${\mu}=(\mu_i)_{i=1}^k\in \Delta_k\cap (0,1)^k$, 
and $\mathfrak{f}:[0,\infty)\to \sR\cup \{\infty\}$ 
be a lower semicontinuous function 
 which  satisfies 
 $\mathfrak{f}(0)=\lim_{x\to 0} \mathfrak{f}(x)$, $\mathfrak{f}(1)=0$ and 
 $\mathfrak{f}$ is $\kappa_\mu$-strongly convex on $[0,\tfrac{1}{\min_i\mu_i}]$ with a constant $\kappa_\mu>0$.
Then, the   $\mathfrak{f}$-divergence $D_ \mathfrak{f}(\cdot || \mu):\sR^k\to \sR\cup \{\infty\}$ satisfies
$D_ \mathfrak{f}(a || \mu)=\infty$ for $a\not \in \Delta_k$ and 
$$
D_ \mathfrak{f}(a || \mu)\coloneqq 
\sum_{i=1}^k \mu_i \mathfrak{f}\big(\tfrac{a_i}{\mu_i}\big)\in\sR\cup\{\infty\}
\q \fa a\in \Delta_k.
$$ 
One can easily see from  $\mathfrak{f}(1)=0$ and the lower semicontinuity of $\mathfrak{f}$
that 
$D_ \mathfrak{f}(\cdot || \mu)$ is a proper, lower semicontinuous function with effective domain
$ \dom D_ \mathfrak{f}(\cdot || \mu)\subset \Delta_k$. 
Moreover, by  the $\kappa_\mu$-strong convexity of   $\mathfrak{f}$, 
we have for all $a,\tilde{a}\in \Delta_k$, $\eta\in [0,1]$ that 
\begin{align*}
&\eta D_ \mathfrak{f}(a || \mu)+(1-\eta) D_ \mathfrak{f}(\tilde{a} || \mu)
\\
&=\sum_{i=1}^k \mu_i \Big(\eta\mathfrak{f}\big(\tfrac{a_i}{\mu_i}\big)+(1-\eta)\mathfrak{f}\big(\tfrac{\tilde{a}_i}{\mu_i}\big)\Big)
\ge \sum_{i=1}^k \mu_i 
\Big(\mathfrak{f}\big(\tfrac{\eta a_i+(1-\eta) \tilde{a}_i}{\mu_i}\big)
+\eta(1-\eta)\frac{\kappa_\mu}{2}|\tfrac{a_i-\tilde{a}_i}{\mu_i}|^2
\Big)
\\
&\ge 
D_ \mathfrak{f}(\eta a+(1-\eta) \tilde{a} || \mu)
+\eta(1-\eta)\frac{\kappa_\mu}{2\max_i\mu_i}|{a-\tilde{a}}|^2,
\end{align*}
which implies the $\tfrac{\kappa_\mu}{\max_i\mu_i}$-strong convexity of $D_ \mathfrak{f}(\cdot || \mu)$ in $\sR^k$.
It is clear that for suitable choices of $f_0,f_1$, the function $f$ in \eqref{eq:f_relaxed} satisfies
 (H.\ref{assum:lc_ns}\ref{item:f0R}).

It is important to notice that an  $\mathfrak{f}$-divergence $D_ \mathfrak{f}(\cdot || \mu)$ 
is in general   non-differentiable and unbounded on its effective domain.
For example, 
one may consider 
the relative entropy (with $\mathfrak{f}(s)=s\log s$) and the squared Hellinger divergence
(with $\mathfrak{f}(s)=2(1-\sqrt{s})$),  which are
not  subdifferentiable at the boundary of  $\Delta_k$.
Moreover, 
the reverse relative entropy (with $\mathfrak{f}(s)=-\log s$)
and the Neyman's $\chi^2$ divergence (with $\mathfrak{f}(s)=\tfrac{1}{s}-1$)
are unbounded  near the boundary of  $\Delta_k$.

\end{Example}

\subsection{Construction of  optimal feedback controls}
In this section, 
we
apply the maximum principle to \eqref{eq:lc} 
and 
explicitly construct optimal feedback controls of \eqref{eq:lc} 
  based on the associated coupled
FBSDE.

The following proposition shows that 
 under (H.\ref{assum:lc_ns}),  
 the control problem \eqref{eq:lc}
 admits a unique  optimal open-loop control.

\begin{Proposition}
Suppose (H.\ref{assum:lc_ns}) holds and let $x_0\in \sR^n$. 
Then the cost functional $J(\a; x_0):\cH^2(\sR^k)\to \sR\cup \{ \infty\}$
is proper, lower semicontinuous, and $\lambda$-strongly convex.
Consequently, 
$J(\cdot; x_0)$ 
admits a unique minimizer  $\a^{x_0}$ in $\cH^2(\sR^k)$.
\end{Proposition}
\begin{proof}
The desired properties of 
$J$ follow directly from
 the corresponding properties of $f$, $g$ in 
 (H.\ref{assum:lc_ns})
and the fact that  \eqref{eq:lc_sde} has affine coefficients.
The  well-posedness  of minimizers 
%of \eqref{eq:lc}
then
follows from the  standard 
theory of
strongly convex 
minimization 
problems
on Hilbert spaces
(see e.g.~\cite[Lemma 2.33 (ii)]{bonnans2013perturbation}).
\end{proof}

We then proceed to study optimal feedback controls of \eqref{eq:lc}.
The classical control theory shows that 
under  suitable coercivity and convexity conditions,
the optimal open-loop control of \eqref{eq:lc} can be expressed in a feedback form,
i.e., there exists a  measurable function
$\psi:[0,T]\t \sR^n\to \sR^k$ such that 
$\a^{x_0} =\psi(t,X^{x_0,\a^{x_0}}_t)$
for  $\d \sP \otimes \d t$ a.e.~(see \cite{{nicole1987compactification}} for the case with controlled jump-diffusions and smooth costs
and \cite{haussmann1990existence} for the case with controlled diffusions
and nonsmooth costs).
However,
since these
non-constructive 
 proofs are based on a measurable selection theorem,
the resulting feedback policy $\psi$ may not be unique, 
and may be   unstable with respect to perturbations of the state dynamics.
%Moreover, 
%the non-constructive argument prevents us from implementing such a feedback strategy
%in practice.

In the subsequent analysis, 
we give a constructive proof of the existence of Lipschitz continuous feedback controls
by exploiting the linear-convex structure of the control problem
 \eqref{eq:lc}-\eqref{eq:lc_sde}.
Such a feedback control 
can be explicitly represented  as 
solutions of a suitable FBSDE,
and hence is Lipschitz stable with respect to  
perturbations of underlying models (see Theorem \ref{thm:feedback_stable}).

We first present the precise 
definitions of  feedback controls
and the associated state processes.

\begin{Definition}\l{def:fb}
Let $\cV$ be the following 
space of 
 feedback controls:
 \begin{equation}\l{eq:lipschitz_feedback}
\cV \coloneqq 
\left\{ 
\psi: [0,T]\t \sR^n\to \sR^k
\,\middle\vert\, 
%\|v\|_{\textrm{Lip}}
%\coloneqq 
%\sup_{
%\substack{
%(t,x,y)\in [0,T]\t \sR^n\t \sR^n,
%\\ x\not =y}
%}
%\frac{|v(t,x)-v(t,y)|}{|x-y|}
%<\infty.
%\\
\begin{aligned}
&\textnormal{$\psi$ is measurable and there exists $ C\ge 0$ such that
}
\\
&\textnormal{
for all $(t,x,y)\in [0,T]\t \sR^n\t \sR^n$,
 $|\psi(t,0)|\le C$
}
\\
&
\textnormal{
 and 
 $|\psi(t,x)-\psi(t,y)|\le C|x-y|$.
}
 \end{aligned}
\right\}
\end{equation}
For any given  $x_0\in \sR^n$
and  $\psi\in \cV$,
we say $X^{x_0,\psi}\in \cS^2(\sR^n)$ is the state process associated with 
$\psi$   if it satisfies the following dynamics:
\bb\l{eq:lc_sde_fb}
\d X_t =b(t,X_t,{\psi}(t,X_t))\,\d t+\sigma(t)\, \d W_t
+\int_{\sR^p_0}
\gamma(t,u)\, \tilde{N}(\d t,\d u),
, \q t\in [0,T],
\q X_0=x_0.
\ee
We say $\psi\in \cV$ is an optimal feedback control of \eqref{eq:lc}
if
it holds 
%for all $t\in [0,T]$  that 
for  $\d \sP \otimes \d t$ a.e.~that
$\a^{x_0}_t=\psi(t,X^{x_0,\psi}_t)$,
where $\a^{x_0}\in \cH^2(\sR^n)$ is 
the  optimal control 
of \eqref{eq:lc}. 

\end{Definition}

We then proceed to establish a maximum principle 
for feedback controls of the control problem \eqref{eq:lc} with non-smooth costs.
Let  $H:[0,T]\t \sR^n\t \sR^k\t \sR^n\to \sR\cup\{\infty\}$
and ${\phi}:[0,T]\t  \sR^n\t  \sR^n\to \sR^k$
such that for
all $(t,x,a,y)\in [0,T]\t \sR^n\t \sR^k \t \sR^n$,
\begin{align}\l{eq:hamiltonian}
H(t,x, a,y)&\coloneqq \la b(t,x,a), y\ra  +f(t,x,a),
\q
\phi(t,x,y)\coloneqq\argmin_{a\in \sR^k} H(t,x, a,y)\in \dom \cR.
%\l{eq:hamiltonian}
\end{align}
The following   lemma shows that the function  $\phi$
is well-defined and measurable.

\begin{Lemma}\l{lemma:measurable_phi}
 Suppose (H.\ref{assum:lc_ns}) holds.
Then the function ${\phi}:[0,T]\t \sR^n\t  \sR^n\to \sR^k$ 
defined  in \eqref{eq:hamiltonian}
is  measurable
and
satisfies 
for all  $t\in [0,T]$, $x,y\in \sR^n$ that 
\begin{align}
\begin{split}
\phi(t,x,y)
&=
\p_z f^* (t,x,-b_2(t)^\trans y),
\end{split}
\end{align}
where the function 
$f^*$
is   
defined  in 
\eqref{eq:conjugate}.

 \end{Lemma}

  \begin{proof}
Let
$f^*:[0,T]\t \sR^n\t \sR^k\to \sR\cup \{\infty\}$
be   the   function  
defined  in 
\eqref{eq:conjugate}.
Recall that 
for each $(t,x)\in [0,T]\t \sR^n$,
$f(t,x,\cdot)$ is
$\lambda$-strongly convex and lower semicontinuous.
Hence %\textcolor{blue}
{by \cite[Theorems 11.3 and 11.8]{rockafellar2009variational},}
$f^*(t,x,\cdot)$  is finite and differentiable on $\sR^k$,
and 
   $\p_zf^*(t,x,z)=\argmax_{a\in \dom \cR} \big(\la a, z\ra -f(t,x,a)\big)$
  for all $z\in \sR^k$.
  Moreover, %\textcolor{blue}
  {
  by \cite[Theorem E4.2.1]{hiriart2004fundamentals},}
  $\sR^k\ni z\mapsto \p_z f^*(t,x,z)\in \sR^k$ 
  is  $1/\lambda$-Lipschitz continuous.
Hence, from 
 the definition of $\phi$ and (H.\ref{assum:lc_ns}\ref{item:lc_linear}), 
for all  $t\in [0,T]$, $x,y\in \sR^n$,
\begin{align}\l{eq:phi_f*}
\begin{split}
\phi(t,x,y)
&=
\argmin_{a\in \sR^k}
\Big(\la b(t,x,a), y\ra +f(t,x,a)\Big)
=
\argmax_{a\in \sR^k}
\Big(\la a, -b_2(t)^\trans y\ra -f(t,x,a)\Big)
\\
&=
\p_z f^* (t,x,-b_2(t)^\trans y).
\end{split}
\end{align}
Note that the measurability of $f^*$  (see Remark \ref{rmk:integrand})
 implies that the derivative $\p_z f^*$
is  measurable, 
 which along with the continuity of $z\mapsto \p_z f^*(t,x,z)$ leads to the measurability of $\phi$.
\end{proof}

With the measurable function $\phi$ in hand, 
for each $(t,x)\in [0,T]\t \sR^n$, let us consider 
 the following coupled FBSDE on $[t,T]$: for all $s\in [t,T]$,
\begin{subequations}\l{eq:lc_fbsde}
\begin{alignat}{2}
\mathrm{d}X_s&=
b(s,X_s,\phi(s,X_s,Y_s))\, \d s
+\sigma(s)\, \d W_s
+\int_{\sR^p_0}
\gamma(s,u)\, \tilde{N}(\d s,\d u),
&&
\q X_t=x,
\l{eq:lc_fbsde_fwd}
\\
\mathrm{d}Y_s&
=-\p_x H(s,X_s,\phi(s,X_s,Y_s),Y_s)\,\d s
+Z_s\, \d W_s
+\int_{\sR^p_0}
M_s\, \tilde{N}(\d s,\d u),
&&
\q Y_T=\nabla g(X_T).
\end{alignat}
\end{subequations}
We say 
a tuple of processes 
 $(X^{t,x},Y^{t,x}, Z^{t,x},M^{t,x})
\in \sS(t,T)\coloneqq 
 \cS^2(t,T;\sR^n)\t \cS^2(t,T;\sR^n)\t \cH^2(t,T;\sR^{n\t d})
\t \cH^2_\nu(t,T;\sR^{n})$ 
  is a solution to 
\eqref{eq:lc_fbsde}
(on $[t,T]$ with initial condition $X^{t,x}_t=x$)
if 
it satisfies \eqref{eq:lc_fbsde} 
 $\sP$-almost surely.

The next lemma presents several important properties of the Hamiltonian $H$ and the function $\phi$
defined  in \eqref{eq:hamiltonian},
which are essential for  the well-posedness and stability of \eqref{eq:lc_fbsde}.

\begin{Lemma}\l{lemma:lipschitz_phi}
 Suppose (H.\ref{assum:lc_ns}) holds.
Let  ${\phi}:[0,T]\t \sR^n\t  \sR^n\to \sR^k$ 
be  the function
defined  in \eqref{eq:hamiltonian}.
Then
there exists a constant $C$
%depending only on  $\lambda, L$, 
such that for all 
$t\in [0,T]$ and  $(x,y),(x',y')\in  \sR^n\t  \sR^n$,
$|\phi(t,0, 0)|\le C$,
$|\phi(t,x,y)-\phi(t,x',y')|\le C(|x-x'|+|y-y'|)$
and 
\begin{align}\l{eq:lc_mono}
\begin{split}
&\la b(t,x,\phi(t,x,y))- b(t,x',\phi(t,x',y')),y-y'\ra
\\
&\q  +\la -\p_x H(t,x,\phi(t,x,y),y)+\p_x H(t,x',\phi(t,x',y'),y'),x-x'\ra 
\\
&\le -\lambda |\phi(t,x,y)-\phi(t,x',y')|^2,
\end{split}
 \end{align}
 with the constant $\lambda$ in  (H.\ref{assum:lc_ns}).

 \end{Lemma}

  \begin{proof}

We start by  showing the  boundedness of $\phi(\cdot,0,0)$ by considering 
$a(t)\coloneqq (\p_z f^*) (t,0,0)$ for each $t\in [0,T]$,
where 
$f^*:[0,T]\t \sR^n\t \sR^k\to \sR$
is  
defined as in 
\eqref{eq:conjugate}.
The fact that $f(t,0,\cdot)$ is proper, lower semicontinuous and convex implies that 
$0\in \widehat{\p}_a f(t,0,a(t))$ for all $t\in [0,T]$,
where $\widehat{\p}_a f(t,0,a(t))$ is the subdifferential of $f(t,0,\cdot)$ at $a(t)$.
Note that $f_0(t,0,\cdot)$ and $\cR$ are proper, lower semicontinuous, and convex,
and $\dom \cR\subset \dom f_0(t,0,\cdot)=\sR^k$. 
Hence by %we can obtain from 
\cite[Corollary 10.9]{rockafellar2009variational},
 $\widehat{\p}_a f(t,0,a)=\p_a f_0(t,0,a)+\widehat{\p} \cR(a)$
for all $(t,a)\in [0,T]\t \sR^k$,
where $\widehat{\p} \cR(a)$ is the subdifferential of $\cR$ at $a$.
Now  fix an arbitrary $t_0\in [0,T]$ and set $a_0=a(t_0)$.
The fact that  $0\in \widehat{\p}_a f(t_0,0,a_0)$ implies that 
$-\p_a f_0(t_0,0,a_0)\in \widehat{\p} \cR(a_0)$
and hence 
 $\p_a f_0(t,0,a_0)-\p_a f_0(t_0,0,a_0)\in \widehat{\p}_a f(t,0,a_0)$
 for all $t\in [0,T]$.
By  the strong convexity condition \eqref{eq:strong_convex}, %we have 
for all $t\in [0,T]$, $\xi_1\in \widehat{\p}_a f(t,0,a_0)$ and $\xi_2\in \widehat{\p}_a f(t,0,a(t))$, 
$$
\lambda |a_0-a(t)|^2\le \la \xi_1-\xi_2, a_0-a(t)\ra \le  |\xi_1-\xi_2||a_0-a(t)|.
$$
Taking $\xi_1=\p_a f_0(t,0,a_0)-\p_a f_0(t_0,0,a_0)$ and $\xi_2=0$ in the above inequality
yields
\begin{align*}
 |a_0-a(t)|\le   |\p_a f_0(t,0,a_0)-\p_a f_0(t_0,0,a_0)|/\lambda \le C,
\end{align*}
by the linear growth of $\p_a f_0(t,0,\cdot)$.
This implies that 
$|(\p_z f^*) (t,0,0)|\le C$ for all $t\in [0,T]$,
which
 along with \eqref{eq:phi_f*} leads to  the  desired uniform boundedness of $\phi(\cdot,0,0)$.

We proceed to establish the Lipschitz continuity of $\phi$ with respect to $(x,y)$.
The $1/\lambda$-Lipschitz continuity of $\p_z f^*(t,x,\cdot)$
and the boundedness of $b_2$
imply that $\phi$ is Lipschitz continuous in $y$, uniformly with respect to $(t,x)$.
It remains to show the Lipschitz continuity of $\p_z f^*$ with respect to $x$,
which along with \eqref{eq:phi_f*} leads to the desired Lipchitz continuity of $\phi$. 
For any given $(t,z)\in [0,T]\t \sR^k$ and $x,x'\in \sR^n$, let $a=\p_z f^*(t,x,z)$ and $a'=\p_z f^*(t,x',z)$.
Then we have $z\in  \widehat{\p}_a f(t,x,a)$ and $z\in  \widehat{\p}_a f(t,x',a')$.
Moreover, by 
 the convexity of $f(t,x,\cdot)$ for all $(t,x)\in[0,T]\t \sR^n$
and  similar arguments as above,
 we can show  that
$z-\p_a f_0(t,x',a')+\p_a f_0(t,x,a')\in  \widehat{\p}_a f(t,x,a')$,
which together with the convexity condition \eqref{eq:strong_convex}
and $z\in  \widehat{\p}_a f(t,x,a)$
leads to 
\begin{align*}
\lambda |a'-a|\le |z-\p_a f_0(t,x',a')+\p_a f_0(t,x,a')-z|\le L|x-x'|,
\end{align*}
where we have used the $L$-Lipchitz continuity of $\p_a f_0(t,\cdot,\cdot)$.
This finishes the proof of the Lipschitz continuity of 
$\p_zf^*(t,\cdot,\cdot)$ and 
$\phi(t,\cdot,\cdot)$.

Finally, we establish the monotonicity condition \eqref{eq:lc_mono}.
By  
 (H.\ref{assum:lc_ns}\ref{item:lc_f}),
%we can obtain 
for all $t\in [0,T]$, $x,x'\in\sR^n, a,a'\in \sR^k,y\in \sR^n$,
the function $\sR^n\t \sR^k \ni(x,a)\mapsto H(t,x,a,y)\in \sR^n\cup \{\infty\}$
satisfies the same convexity condition \eqref{eq:strong_convex}
as the function $f$, 
and hence 
\begin{align}\l{eq:H_sub}
H(t,x',a',y)-H(t,x,a,y)
\ge \la \xi, x'-x,a'-a\ra+ \tfrac{\lambda}{2} | a'-a|^2 \q \fa \xi\in \widehat{\p}_{(x,a)} H(t,x,a,y),
\end{align}
where $ \widehat{\p}_{(x,a)}H(t,x,a,y)$ denotes  
the subdifferential of the function $H(t,\cdot,y)$ at $(x,a)$. 
Moreover,
for any given $t\in [0,T]$ and $x,y\in \sR^n$,
the definition of $\phi$ in \eqref{eq:hamiltonian}
 implies that 
$0\in  \widehat{\p}_a H(t,x,\phi(t,x,y),y)$,
where $ \widehat{\p}_a H(t, x,\phi(t, x,y),y)$ denotes 
the subdifferential of the function $ H(t, x,\cdot,y)$ at $\phi(t,x,y)$. 
{%\color{blue}
Now recall that for 
any Euclidean space $E$, convex function
$F:E\to \sR\cup\{\infty\}$ and $x\in \dom F$, 
$v\in \widehat{\p} F(x)$ if and only if 
$\liminf_{\tau \to 0,\tilde{w}\to w} \frac{F(x+\tau \tilde{w})-F(x)}{\tau }\ge\la v, w\ra $ for all $w\in E$
(see e.g., Exercise 8.4 and Proposition 8.12 in \cite{rockafellar2009variational}).
Thus,
 for any  $t\in [0,T]$ and $x,y\in \sR^n$, 
  $0\in  \widehat{\p}_a H(t,x,\phi(t,x,y),y)$ yields
for all $z\in \sR^k$,
\bb
\label{eq:subderivative_1}
\liminf_{\tau \to 0,\tilde{z}\to z}
\frac{H(t, x,\phi(t, x,y)+\tau \tilde{z},y)
-H(t, x,\phi(t, x,y),y)
}{\tau }\ge \la 0, z\ra =0.
\ee
Moreover, 
by
%\eqref{eq:hamiltonian},
 the convexity of $H$   and the
continuity of $\p_x H$ in $(x,a)$,
for any  $t\in [0,T]$ and $x,y,w\in \sR^n$ and $z\in \sR^k$, 
\begin{align}
\label{eq:subderivative_2}
\begin{split}
&\liminf_{\tau \to 0,(\tilde{w},\tilde{z})\to (w,z)}
\frac{H(t, x+\tau \tilde{w},\phi(t, x,y)+\tau \tilde{z},y)
-H(t, x,\phi(t, x,y)+\tau \tilde{z},y)
}{\tau }
\\
&\ge \liminf_{\tau \to 0,(\tilde{w},\tilde{z})\to (w,z)}
\frac{\la \p_x H(t, x,\phi(t, x,y)+\tau \tilde{z},y)
,\tau \tilde{w}\ra 
}{\tau }
\ge  \la \p_x H(t, x,\phi(t, x,y), y)
, {w}\ra,
\end{split}
\end{align}
provided that $\phi(t, x,y)+\tau \tilde{z}\in \dom \cR$
(cf.~\eqref{eq:hamiltonian}).  
Then for any  $t\in [0,T]$ and $x,y\in \sR^n$,
adding up 
\eqref{eq:subderivative_1} and 
\eqref{eq:subderivative_2} 
and using the fact that $\phi(t,x,y)\in \dom\cR$
give
for all $(w,z)\in \sR^n \t \sR^k$, 
\begin{align*}
&\liminf_{\tau \to 0,(\tilde{w},\tilde{z})\to (w,z)}
\frac{H(t, x+\tau \tilde{w},\phi(t, x,y)+\tau \tilde{z},y)
-H(t, x,\phi(t, x,y), y)
}{\tau }
\ge  \la \p_x H(t, x,\phi(t, x,y), y)
, {w}\ra + \la 0, z\ra,
\end{align*}
which implies 
\bb\l{eq:partialH_subH}
(\p_xH(t,x,\phi(t,x,y), y), 0)\subset  \widehat{\p} _{(x,a)}H(t,x,\phi(t,x,y), y).
\ee
}\!\!  
Hence for all 
$t\in [0,T]$,
$(x_1,y_1),(x_2,y_2)\in \sR^n\t \sR^n$, 
we can define  $a_1=\phi(t,x_1,y_1), a_2=\phi(t,x_2,y_2)$ and deduce that 
\begin{align*}
&\la b(t,x_1,a_1)- b(t,x_2,a_2),y_1-y_2\ra +\la -\p_x H(t,x_1,a_1,y_1)+\p_x H(t,x_2,a_2,y_2),x_1-x_2\ra 
\\
&=
H(t,x_1,a_1,y_1)-H(t,x_2,a_2,y_1)-\la \p_x H(t,x_1,a_1,y_1),x_1-x_2\ra 
\\
&\q -
\big(H(t,x_1,a_1,y_2)-H(t,x_2,a_2,y_2)-\la \p_x H(t,x_2,a_2,y_2),x_1-x_2\ra 
\big)
\\
&\le -\lambda |a_1-a_2|^2,
\end{align*}
which finishes the proof of the desired monotonicity condition.
\end{proof}

%Based on Lemma \ref{lemma:lipschitz_phi},
The following proposition shows that
 \eqref{eq:lc_fbsde} admits a unique solution, which is Lipschitz 
 continuous with respect to the initial state.
 The proof
 is based on the stability of 
 \eqref{eq:lc_fbsde} under the generalized 
 monotonicity condition \eqref{eq:lc_mono}
 (see    Lemma \ref{lemma:mono_stab}),
% whose detailed step is similar to 
%that of 
and follows 
\cite[Corollary 2.4]{reisinger2020path}
for the case without jumps.
%  and hence 
%omitted.

\begin{Proposition}\l{prop:fbsde_wp}
Suppose (H.\ref{assum:lc_ns}) holds.
Then
for any given  $(t,x)\in [0,T]\t \sR^n$, 
the FBSDE \eqref{eq:lc_fbsde} admits 
 a unique solution
 $(X^{t,x},Y^{t,x}, Z^{t,x},M^{t,x})\in \sS(t,T)$.
% on $[t,T]$ with the initial condition $X^{t,x_0}_t=x_0$.
% \item\l{item:fbsde_moment} 
 Moreover, there exists a constant $C$
% , depending only on   $\lambda, L$, 
such that
 for all  $t\in [0,T]$ and $x,x'\in \sR^n$, 
$\|(X^{t,x},Y^{t,x}, Z^{t,x},M^{t,x})\|_{\sS(t,T)}\le C(1+|x|)$
and 
$%\begin{align*}
\|(X^{t,x}-X^{t,x'},Y^{t,x}-Y^{t,x'},Z^{t,x}-Z^{t,x'},M^{t,x}-M^{t,x'})\|_{\sS(t,T)}\le C|x-x'|.
$
\end{Proposition}

Now we are ready to  present 
the main result of this section, 
which constructs an optimal feedback control of \eqref{eq:lc}
 based on the Hamiltonian \eqref{eq:hamiltonian} and the solutions to the FBSDE \eqref{eq:lc_fbsde}.

\begin{Theorem}\l{thm:lc_fb}
Suppose (H.\ref{assum:lc_ns}) holds.
Let 
$\psi:[0,T]\t \sR^n\to \sR^k$ be the function 
defined as
\bb\l{eq:feedback}
\psi(t,x)\coloneqq \phi(t,x,Y^{t,x}_t), \q (t,x)\in [0,T]\t\sR^n,
\ee
where 
 the function ${\phi}$ is defined  in \eqref{eq:hamiltonian}.
 Then 
 there exists a constant $C$ such that 
$|\psi(t,0)|\le C$ and $|\psi(t,x)-\psi(t,x')|\le C|x-x'|$
  for all  $t\in [0,T]$, $x,x'\in \sR^n$.
Moreover,
for all $x_0\in \sR^n$,
$\psi$ is an optimal feedback control of
 \eqref{eq:lc}. 

\end{Theorem}

\begin{proof}
We first analyze the mapping $[0,T]\t\sR^n\ni (t,x)\mapsto v(t,x)\coloneqq Y^{t,x}_t\in \sR^n$.
Note by Proposition \ref{prop:fbsde_wp},
for any given 
$(t,x)\in [0,T]\t \sR^n$, the solution to  
\eqref{eq:lc_fbsde} (with initial time $t$ and initial state $x$)  is  pathwise unique
and
 Lipschitz continuous with respect to the initial state $x\in \sR^n$.
Hence, 
it is well-known that 
(see e.g., Theorem 3.1 and Remarks 3.2-3.3
in \cite{li2014lp})
 that 
 the  map $v$ can be identified with a deterministic function 
  in the space $\cV$ and 
it holds for all $(t,x)\in [0,T]\t \sR^n$ that
$ \sP(\fa s\in [t,T], Y^{t,x}_s=v(s,X^{t,x}_s))=1$.
Thus, from the regularity of $\phi$ and $v$,
$|\psi(t,0)|\le C$ and $|\psi(t,x)-\psi(t,x')|\le C|x-x'|$ for all $x,x'\in \sR^n$,
i.e., $\psi$ is in the space $\cV$.

Now let   $x_0\in \sR^n$ be a given initial state
and 
 $\tilde{\a}\in \cA$ satisfy 
    for   $\d \sP \otimes \d t$ a.e.~that
  $\tilde{\a}_t=\phi(t,X^{0,x_0}_t,Y^{0,x_0}_t)$.
  Then %we have 
  for   $\d \sP \otimes \d t$ a.e.,
  $\tilde{\a}_t%=\phi(t,X^{0,x_0}_t,Y^{0,x_0}_t)
  =\phi(t,X^{0,x_0}_t,v(t,X^{0,x_0}_t))=\psi(t,X^{0,x_0}_t)
  $,
and  $X^{0,x_0}$ is the solution to  \eqref{eq:lc_sde}
controlled by    $\tilde{\a}$, because
$(X^{0,x_0},Y^{0,x_0})$ satisfy 
\eqref{eq:lc_fbsde_fwd}.
Since the control problem \eqref{eq:lc}
admits an unique optimal control in $\cH^2(\sR^k)$, 
 it suffices to show 
that $\tilde{\a}$ is  optimal.
By  \eqref{eq:partialH_subH}, 
    for   $\d \sP \otimes \d t$ a.e.,
$$
(\p_xH(t,X^{0,x_0}_t,\phi(t,X^{0,x_0}_t,Y^{0,x_0}_t),Y^{0,x_0}_t), 0)\subset  \widehat{\p} _{(x,a)}H(t,X^{0,x_0}_t,\phi(t,X^{0,x_0}_t,Y^{0,x_0}_t),Y^{0,x_0}_t).
$$
Then for any given  $\a\in \cH^2(\sR^n)$ with the state process $X^{x_0,\a}$ satisfying the controlled dynamics  \eqref{eq:lc_sde},  
by
the definition of $H$ in \eqref{eq:hamiltonian}, 
 (H.\ref{assum:lc_ns}\ref{item:lc_g}) and \eqref{eq:H_sub},
\begin{align*}
&J(\a;x_0)-J(\tilde{\a};x_0)
\\
&=
\sE\left[
g(X_T^{x_0,\a})-g(X^{0,x_0}_T)
+
\int_0^T 
(H(t,X^{x_0,\a}_t, \a_t,Y^{0,x_0}_t)
-H(t,X^{0,x_0}_t, \tilde{\a}_t,Y^{0,x_0}_t))\, \d t\right]
\\
&\q 
-\int_0^T \la b(t, X^{x_0,\a}_t,\a_t)-b(t, X^{0,x_0}_t,\tilde{\a}_t), Y^{0,x_0}_t \ra\, \d t\bigg]
\\
&\ge 
\sE\bigg[ \la \nabla g(X^{0,x_0}_T),X^{x_0,\a}_T-X^{0,x_0}_T\ra 
+\int_0^T \la \p_xH(t,X^{0,x_0}_t,\phi(t,X^{0,x_0}_t,Y^{0,x_0}_t),Y^{0,x_0}_t), X^{x_0,\a}_t-X^{0,x_0}_t \ra\, \d t
\\
&\q 
-\int_0^T \la b(t, X^{x_0,\a}_t,\a_t)-b(t, X^{0,x_0}_t,\tilde{\a}_t), Y^{0,x_0}_t \ra\, \d t\bigg]
=0,
\end{align*}
where  the last equality is by applying  It\^{o}'s formula to
the process 
 $(\la X^{x_0,\a}_t-X^{0,x_0}_t, Y^{0,x_0}_t\ra)_{t\ge 0}$
 and by the FBSDE \eqref{eq:lc_fbsde}.
That is, 
 $\psi\in \cV$ is an optimal feedback control of \eqref{eq:lc}.
\end{proof}

\subsection{Lipschitz stability of  optimal feedback controls and associated costs}
\l{sec:robust_fb}

In this section, we 
establish 
the Lipschitz stability of
the optimal feedback controls 
constructed in Theorem \ref{thm:lc_fb}
and their associated costs:
that is, they are
are Lipschitz continuous
with respect to the perturbation in the coefficients of \eqref{eq:lc_sde}.
Such a Lipschitz stability property is crucial for the subsequent analysis of learning algorithms.

More precisely, for any given $x_0\in \sR^n$,  we consider a perturbed control problem 
where the cost functions $f,g$ are the same as those in \eqref{eq:lc}, 
and  for each $\a\in \cH^2(\sR^n)$, the corresponding state dynamics  
satisfies the following  perturbed  dynamics: 
\bb\l{eq:lc_sde_per}
\d X_t =\tilde{b}(t,X_t,\a_t)\,\d t+\tilde{\sigma}(t)\, \d W_t
+\int_{\sR^p_0}
\tilde{\gamma}(t,u)\, \tilde{N}(\d t,\d u), \q t\in [0,T],
\q X_0=x_0,
\ee
whose coefficients satisfy the following assumption:
\begin{Assumption}\l{assum:lc_ns_per} 
$\tilde{b}:[0,T]\t  \sR^n\t \sR^k\to \sR^n$,
 $\tilde{\sigma}:[0,T]\to \sR^{n\t d}$
 and 
 $\tilde{\gamma}:[0,T]\t \sR^p_0\to \sR^{n}$
 satisfy
 (H.\ref{assum:lc_ns}\ref{item:lc_linear})
with the same  constant $L$, i.e.,
there exist 
measurable functions 
$(\tilde{b}_0,\tilde{b}_1,\tilde{b}_2):[0,T]\to \sR^{n}\t \sR^{n\t n}\t \sR^{n\t k}$ 
 such that
 $\tilde{b}(t,x,a)=\tilde{b}_0(t)+\tilde{b}_1(t)x+\tilde{b}_2(t)a$
 for all $(t,x,a)\in [0,T]\t  \sR^n\t \sR^k$
 and 
$
\|\tilde{b}_0\|_{L^2}+\|\tilde{b}_1\|_{L^\infty}+\|\tilde{b}_2\|_{L^\infty}+\|\tilde{\sigma}\|_{L^2}
+\big(\int_0^T\int_{\sR^p_0} |\tilde{\gamma}(t,u)|^2\,\nu(\d u)\d t\big)^{1/2}
\le L.$

\end{Assumption}

Under (H.\ref{assum:lc_ns}) and (H.\ref{assum:lc_ns_per}), 
 Theorem \ref{thm:lc_fb} ensures that
an optimal feedback control of the perturbed control problem can be obtained by 
\bb\l{eq:feedback_per}
[0,T]\t \sR^n \ni (t,x)\mapsto \tilde{\psi}(t,x)\coloneqq \tilde{\phi}(t,x,\tilde{Y}^{t,x}_t)\in \sR^k,
\ee
where 
$\tilde{\phi}:[0,T]\t  \sR^n\t  \sR^n\to \sR^k$
%is the pointwise minimizer of the  perturbed Hamiltonian
%  $\tilde{H}:[0,T]\t \sR^n\t \sR^k\t \sR^n\to \sR\cup\{\infty\}$:
satisfies for all $(t,x,a,y)\in [0,T]\t \sR^n\t \sR^k \t \sR^n$ that
 \begin{align}\l{eq:hamiltonian_per}
\tilde{\phi}(t,x,y)\coloneqq\argmin_{a\in \sR^k} \tilde{H}(t,x, a,y),
\q \tilde{H}(t, x, a,y)&\coloneqq \la \tilde{b}(t, x,a), y\ra +{f}(t,x,a),
\end{align}
and for each $(t,x)\in [0,T]\t \sR^n$, 
$(\tilde{X}^{t,x},\tilde{Y}^{t,x},\tilde{Z}^{t,x}, \tilde{M}^{t,x})\in \sS(t,T)$
 is the solution to the following perturbed FBSDE:
for all $s\in[t,T]$,
\bb\l{eq:lc_fbsde_per}
\begin{alignedat}{2}
\mathrm{d}X_s&=
\tilde{b}(s,X_s,\tilde{\phi}(s,X_s,Y_s))\, \d s
+\tilde{\sigma}(s)\, \d W_s
+\int_{\sR^p_0}
\tilde{\gamma}(s,u)\, \tilde{N}(\d s,\d u),
&&\q X_t=x,
\\
\mathrm{d}Y_s&
=-\p_x \tilde{H}(s,X_s,\tilde{\phi}(s,X_s,Y_s),Y_s)\,\d s
+Z_s\, \d W_s
+\int_{\sR^p_0}
M_s\, \tilde{N}(\d s,\d u),
&&
\q Y_T=\nabla {g}(X_T).
\end{alignedat}
\ee

The following theorem quantifies the difference of optimal feedback controls in terms of  the  magnitude of perturbations in the coefficients.

\begin{Theorem}\l{thm:feedback_stable}
Suppose  (H.\ref{assum:lc_ns}) and (H.\ref{assum:lc_ns_per})  hold.
Let 
${\psi}, \tilde{\psi}
:[0,T]\t  \sR^n\t  \sR^n\to \sR^k$ be the functions  defined  in 
 \eqref{eq:feedback} and  \eqref{eq:feedback_per}, respectively.
Then
 there exists a constant $C$ such that
 ${|\psi(t,x)-\tilde{\psi}(t,x)|}\le C(1+|x|)\cE_{\textrm{per}}$
 for all $(t,x)\in [0,T]\t \sR^n$,
 with the constant $\cE_{\textrm{per}}$ defined by
 \begin{align}\l{eq:E_per}
\begin{split}
\cE_{\textrm{per}}
&\coloneqq
\|b_0-\tilde{b}_0\|_{L^2}
+
\|b_1-\tilde{b}_1\|_{L^\infty}
+\|b_2-\tilde{b}_2\|_{L^\infty}
+\|\sigma-\tilde{\sigma}\|_{L^2}
\\
&\q 
+
\bigg(
\int_0^T\int_{\sR^p_0} |\gamma(t,u)-\tilde{\gamma}(t,u)|^2\,\nu(\d u)\d t
\bigg)^{1/2}.
\end{split}
\end{align}

\end{Theorem}

\begin{proof}
Throughout this proof,
for each $(t,x)\in [0,T]\t \sR^n$, 
let 
$({X}^{t,x},{Y}^{t,x},{Z}^{t,x}, {M}^{t,x})\in \sS(t,T)$ and
$(\tilde{X}^{t,x},\tilde{Y}^{t,x},\tilde{Z}^{t,x}, \tilde{M}^{t,x})\in \sS(t,T)$
be the solutions to 
 \eqref{eq:lc_fbsde} and 
 \eqref{eq:lc_fbsde_per}, respectively,
 and
 let $C$ be a generic constant which is independent of 
$(t,x)\in [0,T]\t \sR^n$.
Then
by Proposition \ref{prop:fbsde_wp},
there exists 
%a constant 
$C\ge 0$ such that
for all $(t,x)\in [0,T]\t \sR^n$, 
$\|({X}^{t,x},{Y}^{t,x},Z^{t,x},M^{t,x})\|_{\sS(t,T)}\le C(1+|x|)$
and 
$ \|(\tilde{X}^{t,x},\tilde{Y}^{t,x},\tilde{Z}^{t,x},\tilde{M}^{t,x})\|_{\sS(t,T)}\le C(1+|x|)$.

We first
estimate the difference between 
$(X^{t,x},Y^{t,x},Z^{t,x},{M}^{t,x})$ and $(\tilde{X}^{t,x},\tilde{Y}^{t,x},\tilde{Z}^{t,x},\tilde{M}^{t,x})$
for a given $(t,x)\in [0,T]\t \sR^n$.
By  
Lemmas \ref{lemma:lipschitz_phi}
 and \ref{lemma:mono_stab},
\begin{align*}
 &\|({X}^{t,x}-\tilde{X}^{t,x},{Y}^{t,x}-\tilde{Y}^{t,x},{Z}^{t,x}-\tilde{Z}^{t,x},
 {M}^{t,x}-\tilde{M}^{t,x})\|_{\sS(t,T)}
 \\
 &\le C\bigg\{
\|b(\cdot,\tilde{X}^{t,x},\phi(\cdot,\tilde{X}^{t,x},\tilde{Y}^{t,x}))
-\tilde{b}(\cdot,\tilde{X}^{t,x},\tilde{\phi}(\cdot,\tilde{X}^{t,x},\tilde{Y}^{t,x}))
\|_{\cH^2}
\\
&\q 
+\|
\p_x H(\cdot,\tilde{X}^{t,x},\phi(\cdot,\tilde{X}^{t,x},\tilde{Y}^{t,x}),\tilde{Y}^{t,x})
-\p_x \tilde{H}(\cdot,\tilde{X}^{t,x},\tilde{\phi}(\cdot,\tilde{X}^{t,x},\tilde{Y}^{t,x}),\tilde{Y}^{t,x})\|_{\cH^2}
\\
&\q
+\|\sigma
-\tilde{\sigma}\|_{L^2}
+\bigg(
\int_0^T\int_{\sR^p_0} |\gamma(t,u)-\tilde{\gamma}(t,u)|^2\,\nu(\d u)\d t
\bigg)^{1/2}
\bigg\}.
\end{align*}
It remains to estimate the first two terms on the right-hand side of the above inequality. 
By  (H.\ref{assum:lc_ns}\ref{item:lc_linear}), %we have that
\begin{align*}
& \|b(\cdot,\tilde{X}^{t,x},\phi(\cdot,\tilde{X}^{t,x},\tilde{Y}^{t,x}))
-\tilde{b}(\cdot,\tilde{X}^{t,x},\tilde{\phi}(\cdot,\tilde{X}^{t,x},\tilde{Y}^{t,x}))
\|_{\cH^2}
\\
&\le
\|
b(\cdot,\tilde{X}^{t,x},\phi(\cdot,\tilde{X}^{t,x},\tilde{Y}^{t,x}))
-{b}(\cdot,\tilde{X}^{t,x},\tilde{\phi}(\cdot,\tilde{X}^{t,x},\tilde{Y}^{t,x}))
\|_{\cH^2} 
\\
&
\q +\|{b}(\cdot,\tilde{X}^{t,x},\tilde{\phi}(\cdot,\tilde{X}^{t,x},\tilde{Y}^{t,x}))
-\tilde{b}(\cdot,\tilde{X}^{t,x},\tilde{\phi}(\cdot,\tilde{X}^{t,x},\tilde{Y}^{t,x}))
\|_{\cH^2}
\\
&\le
\|b_2\|_{L^\infty}\|\phi(\cdot,\tilde{X}^{t,x},\tilde{Y}^{t,x})-\tilde{\phi}(\cdot,\tilde{X}^{t,x},\tilde{Y}^{t,x})\|_{\cH^2}
+
\|b_0-\tilde{b}_0\|_{L^2}+
\|b_1-\tilde{b}_1\|_{L^\infty}\|\tilde{X}^{t,x}\|_{\cH^2} 
\\
&\q 
+
\|b_2-\tilde{b}_2\|_{L^\infty}\|\tilde{\phi}(\cdot,\tilde{X}^{t,x},\tilde{Y}^{t,x})\|_{\cH^2}. 
\end{align*}
Note that 
by \eqref{eq:phi_f*},
 for all  $t\in [0,T]$ and $x,y\in \sR^n$, 
${\phi}(t,x,y)=(\p_z f^*) (t,x,-{b}_2(t)^\trans y)$ and
$\tilde{\phi}(t,x,y)=(\p_z f^*) (t,x,-\tilde{b}_2(t)^\trans y)$, 
where $f^*$ is the   function defined  in \eqref{eq:conjugate}.
Hence, %we can deduce 
from the $1/\lambda$-Lipschitz continuity of $\p_z f^*(t,x,\cdot)$ 
(see the proof of Lemma   \ref{lemma:measurable_phi}),
\begin{align}
\l{eq:phi-tilde_phi}
\|\phi(\cdot,\tilde{X}^{t,x},\tilde{Y}^{t,x})-\tilde{\phi}(\cdot,\tilde{X}^{t,x},\tilde{Y}^{t,x})\|_{\cH^2}
\le C \|b_2-\tilde{b}_2\|_{L^\infty}\|\tilde{Y}^{t,x}\|_{\cH^2} 
\le C (1+|x|)\cE_{\textrm{per}},
\end{align}
where  the last inequality
follows from the moment estimate of $\tilde{Y}^{t,x}$.  
Moreover, the regularity of $\tilde{\phi}$ (see Lemma  \ref{lemma:lipschitz_phi}) and the moment estimate of 
$(\tilde{X}^{t,x},\tilde{Y}^{t,x})$ imply that 
$\|\tilde{\phi}(\cdot,\tilde{X}^{t,x},\tilde{Y}^{t,x})\|_{\cH^2}\le C(1+|x|)$,
which shows that 
$
\|
b(\cdot,\tilde{X}^{t,x},\phi(\cdot,\tilde{X}^{t,x},\tilde{Y}^{t,x}))
-{b}(\cdot,\tilde{X}^{t,x},\tilde{\phi}(\cdot,\tilde{X}^{t,x},\tilde{Y}^{t,x}))
\|_{\cH^2} \le C (1+|x|)\cE_{\textrm{per}}.
$
By %further  using 
the definitions of $H$ and $\tilde{H}$,
the Lipschitz continuity of $\p_x f_0$ in  (H.\ref{assum:lc_ns}\ref{item:f0R})
and 
\eqref{eq:phi-tilde_phi},
\begin{align*}
 &\|
\p_x H(\cdot,\tilde{X}^{t,x},\phi(\cdot,\tilde{X}^{t,x},\tilde{Y}^{t,x}),\tilde{Y}^{t,x})
-\p_x \tilde{H}(\cdot,\tilde{X}^{t,x},\tilde{\phi}(\cdot,\tilde{X}^{t,x},\tilde{Y}^{t,x}),\tilde{Y}^{t,x})\|_{\cH^2}
\\
 &\le 
 \|
(b_1-\tilde{b}_1)^\trans\tilde{Y}^{t,x}\|_{\cH^2}
+
\|
\p_x f_0(\cdot,\tilde{X}^{t,x},{\phi}(\cdot,\tilde{X}^{t,x},\tilde{Y}^{t,x}))
-
\p_x f_0(\cdot,\tilde{X}^{t,x},\tilde{\phi}(\cdot,\tilde{X}^{t,x},\tilde{Y}^{t,x}))
\|_{\cH^2}
\\
&
\le C (1+|x|)\cE_{\textrm{per}}.
\end{align*}
Thus, we have proved the  stability estimate that 
$\|({X}^{t,x}-\tilde{X}^{t,x},{Y}^{t,x}-\tilde{Y}^{t,x},{Z}^{t,x}-\tilde{Z}^{t,x},
{M}^{t,x}-\tilde{M}^{t,x})\|_{\sS(t,T)}
\le C (1+|x|)\cE_{\textrm{per}}$.

We now establish the stability of feedback controls.
By 
\eqref{eq:phi_f*}
and the $1/\lambda$-Lipschitz continuity of $\p_z f^*(t,x,\cdot)$,
%we have 
% the definition of $\psi,\tilde{\psi}$
%and Lemma \ref{lemma:H_phi_diff}, 
%we can obtain
 for all $(t,x)\in [0,T]\t \sR^n$,
\begin{align*}
&|\psi(t,x)-\tilde{\psi}(t,x)|
=|
(\p_z f^*) (t,x,-{b}_2(t)^\trans Y^{t,x}_t)-
(\p_z f^*) (t,x,-\tilde{b}_2(t)^\trans \tilde{Y}^{t,x}_t))
|
\\
&\le
|{b}_2(t)^\trans Y^{t,x}_t-\tilde{b}_2(t)^\trans \tilde{Y}^{t,x}_t|/\lambda
\le
C(\|{b}_2-\tilde{b}_2\|_{L^\infty}| Y^{t,x}_t|+ |Y^{t,x}_t-\tilde{Y}^{t,x}_t|)
\\
&\le
C(\|{b}_2-\tilde{b}_2\|_{L^\infty} \| Y^{t,x}\|_{\cS^2}+ \|Y^{t,x}_t-\tilde{Y}^{t,x}_t\|_{\cS^2})
\le C (1+|x|)\cE_{\textrm{per}}.
\qedhere
\end{align*}
%which finishes the proof of the desired Lipschitz dependence.
\end{proof}

An important application of the Lipschitz stability of feedback controls (Theorem \ref{thm:feedback_stable}) 
is the analysis of model misspecification error of a given learning algorithm.
One essential component 
is to  
examine the performance of the feedback control $\tilde{\psi}$, computed based on the control problem \eqref{eq:lc}  
with 
the perturbed coefficients  $(\tilde{b},\tilde{\sigma},\tilde{\gamma},f,g) $,
on the true model with coefficients  $({b},{\sigma},{\gamma},f,g) $.
For any given $x_0\in \sR^n$,
implementing the feedback control $\tilde{\psi}$ 
on the original system \eqref{eq:lc_sde}
will lead to the sub-optimal cost:
\bb\l{eq:lc_per}
J(\tilde{\psi}; x_0)\coloneqq \sE\left[\int_0^T f(t,{X}^{x_0,\tilde{\psi}}_t,\tilde{\psi}(t,{X}^{x_0,\tilde{\psi}}_t))\, \d t+g({X}_T^{x_0,\tilde{\psi}})\right],
\ee
where ${X}^{x_0,\tilde{\psi}}\in \cS^2(\sR^n)$ is the 
state process (with coefficients $b$, $\sigma$ and $\gamma$) associated with $\tilde{\psi}$
(see  Definition  \ref{def:fb}).
%
%As an important application of the Lipschitz stability of feedback controls in Theorem \ref{thm:feedback_stable}, 
The following theorem shows that  the difference 
between this suboptimal cost $J(\tilde{\psi}; x_0)$ and the optimal cost $V$ in \eqref{eq:lc}
depends Lipschitz-continuously on the magnitude of perturbations in the coefficients.
%Such a stability result is essential for quantifying the model misspecification error of a learning algorithm
%(see Section \ref{sec:lc_learning}).

 \begin{Theorem}\label{thm:dynamics_stable}
  Suppose (H.\ref{assum:lc_ns}) and (H.\ref{assum:lc_ns_per})  hold.
Let ${\psi}\in \cV$ 
(resp.~$\tilde{\psi}\in \cV$)
be defined in \eqref{eq:feedback}
(resp.~\eqref{eq:feedback_per}),
and for each $x_0\in \sR^n$, let 
${X}^{x_0,{\psi}}\in \cS^2(\sR^n)$  (resp.~${X}^{x_0,\tilde{\psi}}\in \cS^2(\sR^n)$) be the 
state process \eqref{eq:lc_sde} associated with ${\psi}$ (resp.~$\tilde{\psi}$),
and let 
$V(x_0)$ (resp.~$J(\tilde{\psi}; x_0)$)
be defined in \eqref{eq:lc} (resp.~\eqref{eq:lc_per}).
Then there exists a constant $C$ such that for all 
$x_0\in \sR^n$, 
%we have 
$\|{X}^{x_0,{\psi}}-{X}^{x_0,\tilde{\psi}}\|_{\cS^2}\le C(1+|x_0|)\cE_{\textrm{per}}$
and 
$|V(x_0)-J(\tilde{\psi}; x_0)|\le C(1+|x_0|^2)\cE_{\textrm{per}}$,
  with the constant  $\cE_{\textrm{per}}$ defined in \eqref{eq:E_per}.
 \end{Theorem}

 To prove Theorem \ref{thm:dynamics_stable},
 we first establish that
the  composition of  $f$ and the optimal feedback control
is Lipschitz continuous,
even though  the cost function $f$ is merely lower semicontinuous in the control variable (cf.~(H.\ref{assum:lc_ns}\ref{item:f0R})).
The proof is based on the     Fenchel-Young identity:
$$
f(t,x,\p_z f^*(t,x,z))=\la z, \p_z f^*(t,x,z)\ra- f^*(t,x,z)
\in \sR,
\q
\fa (t,x,z)\in [0,T]\t \sR^n\t \sR^k,
$$
 the regularity of $f^*$
 and  Theorem \ref{thm:feedback_stable}, and has been given in Appendix \ref{appendix:technical results}.

 \begin{Lemma}\l{lemma:f_psi}
  Suppose (H.\ref{assum:lc_ns}) and (H.\ref{assum:lc_ns_per})  hold.
Let ${\psi},\tilde{\psi}:[0,T]\t \sR^n\to \sR^k$ be the  functions
defined  in
\eqref{eq:feedback}
and 
\eqref{eq:feedback_per}, respectively.
  Then there exists a constant $C$ such that for all 
  $t\in [0,T]$, 
  $x,x'\in \sR^n$,
  $|f(t,x,\psi(t,x))-f(t,x',\tilde{\psi}(t,x'))|\le C\Big((1+|x|+|x'|)|x-x'|+(1+|x|^2+|x'|^2)\cE_{\textrm{per}}\Big)$,
  where the constant  $\cE_{\textrm{per}}$ is defined in \eqref{eq:E_per}.
 \end{Lemma}

%
%The following theorem quantifies the difference of  value functions 
%in terms of the perturbation of coefficients.

\begin{proof}[Proof of Theorem \ref{thm:dynamics_stable}]
According to Theorems \ref{thm:lc_fb} and \ref{thm:feedback_stable},
there exists a constant $C$ such that 
  for all $t\in [0,T]$, $x,x'\in \sR^n$,
$|\psi(t,0)|+|\tilde{\psi}(t,0)|\le C$,
$|\psi(t,x)-\psi(t,x')|+
|\tilde{\psi}(t,x)-\tilde{\psi}(t,x')|\le C|x-x'|$,
and 
 ${|\psi(t,x)-\tilde{\psi}(t,x)|}\le C(1+|x|)\cE_{\textrm{per}}$.
%  Then, for any given $x_0\in \sR^n$, by
% a standard moment estimate of \eqref{eq:lc_sde_fb},
%   $\|{X}^{x_0,{\psi}}\|_{\cS^2}+\|{X}^{x_0,\tilde{\psi}}\|_{\cS^2}\le C(1+|x_0|)$,
% which along with 
% a stability estimate of \eqref{eq:lc_sde_fb} shows that 
%  \begin{align*}
%  \|{X}^{x_0,{\psi}}-{X}^{x_0,\tilde{\psi}}\|_{\cS^2}
% & \le C
% \bigg(\|b(\cdot,{X}^{x_0,\tilde{\psi}},{\psi}(\cdot,{X}^{x_0,\tilde{\psi}}))
%  -b(\cdot,{X}^{x_0,\tilde{\psi}},\tilde{\psi}(\cdot,{X}^{x_0,\tilde{\psi}}))\|_{\cH^2}
%  +\|\sigma-\tilde{\sigma}\|_{L^2}
%  \\
% &\q + \Big(
% \int_0^T\int_{\sR^p_0} |\gamma(t,u)-\tilde{\gamma}(t,u)|^2\,\nu(\d u)\d t
% \Big)^{1/2}
% \bigg)
%  \\
%  &
% \le C(\|{\psi}(\cdot,{X}^{x_0,\tilde{\psi}})
%  -\tilde{\psi}(\cdot,{X}^{x_0,\tilde{\psi}})\|_{\cH^2}
%  +\cE_{\textrm{per}})
%  \\
%  &
% \le 
% C(1+\|{X}^{x_0,\tilde{\psi}}\|_{\cH^2})\cE_{\textrm{per}}
% \le 
% C(1+|x_0|)\cE_{\textrm{per}}.
%  \end{align*}
Then, for any given $x_0\in \sR^n$, 
standard moment
and stability estimates of \eqref{eq:lc_sde_fb}
yield
  $\|{X}^{x_0,{\psi}}\|_{\cS^2}+\|{X}^{x_0,\tilde{\psi}}\|_{\cS^2}\le C(1+|x_0|)$ and 
 \begin{align*}
 \|{X}^{x_0,{\psi}}-{X}^{x_0,\tilde{\psi}}\|_{\cS^2}
& \le C
\|b(\cdot,{X}^{x_0,\tilde{\psi}},{\psi}(\cdot,{X}^{x_0,\tilde{\psi}}))
 -b(\cdot,{X}^{x_0,\tilde{\psi}},\tilde{\psi}(\cdot,{X}^{x_0,\tilde{\psi}}))\|_{\cH^2}
 \\
 &
\le C\|{\psi}(\cdot,{X}^{x_0,\tilde{\psi}})
 -\tilde{\psi}(\cdot,{X}^{x_0,\tilde{\psi}})\|_{\cH^2}
%  \\
%  &
\le 
C(1+\|{X}^{x_0,\tilde{\psi}}\|_{\cH^2})\cE_{\textrm{per}}
\le 
C(1+|x_0|)\cE_{\textrm{per}}.
 \end{align*}

We now proceed to estimate $|V(x_0)-\tilde{V}(x_0)|$
 for any given $x_0\in \sR^n$. 
By %using 
the mean value theorem,
(H.\ref{assum:lc_ns}\ref{item:lc_g})
and the Cauchy-Schwarz inequality, %we have that
 \begin{align*}
\sE[|g({X}^{x_0,{\psi}}_T)-g({X}^{x_0,\tilde{\psi}}_T)|]
&\le C\sE[|(1+|{X}^{x_0,{\psi}}_T|+|{X}^{x_0,\tilde{\psi}}_T|)|{X}^{x_0,{\psi}}_T-{X}^{x_0,\tilde{\psi}}_T|]
\\
&\le C(1+\|{X}^{x_0,{\psi}}_T\|_{L^2}+\|{X}^{x_0,\tilde{\psi}}_T\|_{L^2})\|{X}^{x_0,{\psi}}_T-{X}^{x_0,\tilde{\psi}}_T\|_{L^2}
\\
&\le
C(1+|x_0|^2)\cE_{\textrm{per}}.
 \end{align*}
 Moreover, %we can deduce 
 from Lemma \ref{lemma:f_psi} and the Cauchy-Schwarz inequality,  
 \begin{align*}
& \sE\left[\int_0^T |f(t,{X}^{x_0,{\psi}}_t,{\psi}(t,{X}^{x_0,{\psi}}_t))- f(t,{X}^{x_0,\tilde{\psi}}_t,\tilde{\psi}(t,{X}^{x_0,\tilde{\psi}}_t))|\, \d t\right]
 \\
 &\le
C \sE\left[\int_0^T 
\bigg(
(1+|{X}^{x_0,{\psi}}_t|+|{X}^{x_0,\tilde{\psi}}_t|)|{X}^{x_0,{\psi}}_t-{X}^{x_0,\tilde{\psi}}_t|
+(1+|{X}^{x_0,{\psi}}_t|^2+|{X}^{x_0,\tilde{\psi}}_t|^2)\cE_{\textrm{per}}
\bigg)\, \d t\right]
\\
 &\le
C \Big(
(1+\|{X}^{x_0,{\psi}}\|_{\cH^2}+\|{X}^{x_0,\tilde{\psi}}\|_{\cH^2})\|{X}^{x_0,{\psi}}-{X}^{x_0,\tilde{\psi}}\|_{\cH^2}
+(1+\|{X}^{x_0,{\psi}}\|_{\cH^2}^2+\|{X}^{x_0,\tilde{\psi}}\|_{\cH^2}^2)\cE_{\textrm{per}}
\Big)
\\
&\le C(1+|x_0|^2)\cE_{\textrm{per}}.
\end{align*}
Since $\psi$ is an optimal feedback control of \eqref{eq:lc}
with the initial state $x_0\in \sR^n$, 
%we can conclude 
the desired estimate $|V(x_0)-J(\tilde{\psi}; x_0)|\le C(1+|x_0|^2)\cE_{\textrm{per}}$ follows.
\end{proof}

\section{Regret analysis for  linear-convex reinforcement learning}
\l{sec:lc_learning}

The focus of this section is  the linear-convex reinforcement learning (RL) problem, 
where the drift coefficient of the state dynamics 
\eqref{eq:lc_sde}
is unknown to the controller,
and the objective is to control the system optimally
 while simultaneously 
 learning the dynamics.
We shall 
propose a greedy least-squares algorithm 
to solve such problems,
and 
show that the algorithm 
provides
a sublinear regret with
 high probability guarantees.
The analysis
of the regret bounds for the algorithm
 relies on 
the Lipschitz stability of  feedback controls  established in Section \ref{sec:lc_ctrl}.

\subsection{Reinforcement learning problem  and least-squares algorithm} 

The RL  problem goes as follows.
Let 
$x_0\in \sR^n$ be a given initial state
and 
$\theta^\star=(A^\star, B^\star)
\in \sR^{n\t (n+k)}$
be fixed but unknown constants,
 consider the following problem:
\bb\l{eq:lc_theta_star}
V(x_0;\theta^\star)=\inf_{\a\in \cH^2(\sR^k)} J^{\theta^\star}(\a;x_0),
\q \textnormal{with}\q 
J^{\theta^\star}(\a;x_0)=\sE\left[\int_0^T f(t,X^{x_0,\theta^\star,\a}_t,\a_t)\, \d t+g(X_T^{x_0,\theta^\star,\a})\right],
\ee
where for each $\a\in  \cH^2(\sR^k)$, the process $X^{x_0,\theta^\star,\a}$ satisfies the following controlled dynamics
 associated with the parameter $\theta^\star$: 
\bb\l{eq:lc_sde_theta_star}
\d X_t =(A^\star X_t+B^\star\a_t)\,\d t+\sigma\, \d W_t
+
\int_{\sR^p_0}
{\gamma}(u)\, \tilde{N}(\d t,\d u), \q t\in [0,T],
\q X_0=x_0,
\ee
with 
a
given constant 
$\sigma\in \sR^{n\t d}$
and 
given
functions 
${\gamma}:\sR^p_0\to \sR^{n}$,
$f:[0,T]\t \sR^n\t \sR^k\to\sR\cup \{\infty\}$ and $g:\sR^n\to\sR$.
If $\theta^\star=(A^\star, B^\star)$
were known, then \eqref{eq:lc_theta_star} is a control problem.

 It is clear that 
 \eqref{eq:lc_theta_star}-\eqref{eq:lc_sde_theta_star}
 is a special case of 
 \eqref{eq:lc}-\eqref{eq:lc_sde}
 with 
 $b(t,x,a)=A^\star x+B^\star a$,
 $\sigma(t)=\sigma$ and $\gamma(t,u)=\gamma(u)$ 
 for all $(t,x,a,u)\in [0,T]\t \sR^n\t \sR^k\t \sR^p_0$.
Hence, if $f$ and $g$ 
satisfy
(H.\ref{assum:lc_ns})
with for some  $L\ge 0$ and $\lambda>0$,
then
 \eqref{eq:lc_theta_star}-\eqref{eq:lc_sde_theta_star}
admits  an optimal feedback control $\psi^{\theta^\star}\in \cV$
as shown in Theorem \ref{thm:lc_fb}.
Note that 
to simplify  the presentation,
we assume  that   \eqref{eq:lc_sde_theta_star}
has time homogenous coefficients
%and  only drift coefficients are unknown
as in 
\cite{abbasi2011regret,mania2019certainty,%basei2020linear
 basei2020logarithmic},
but similar analysis can be performed if
%\textcolor{blue}
{the drift is a linear combination of given time-and-space-dependent basis functions with unknown weights}
%the coefficients are time inhomogeneous 
or 
the diffusion/jump coefficients
are also unknown. 

To solve  \eqref{eq:lc_theta_star}-\eqref{eq:lc_sde_theta_star}
with unknown $\theta^\star$,
in an episodic reinforcement learning framework, 
the controller improves their knowledge of %the underlying state dynamics 
 the parameter $\theta^\star$ %in \eqref{eq:lc_sde_theta_star}
through successive learning episodes.
In particular, 
for each episode $i\in \sN$, 
based on  her observations in the past episodes,
the controller executes a suitable control policy in $\psi_i\in \cV$,
whose associated state dynamics \eqref{eq:lc_sde_theta_star}
 leads to an expected cost $J^{\theta^\star}(\psi_i;x_0)$.
 To measure the performance of an learning  algorithm in this setting, 
one widely adopted criteria
 is the (expected) {regret} of the algorithm defined as follows
  (see e.g.~\cite{dann2017unifying,%basei2020linear
 basei2020logarithmic}): 
\begin{equation}\l{eq:regret}
R(N) = \sum_{i=1}^N \Big(J^{\theta^\star}(\psi_i;x_0) - V(x_0;\theta^\star)\Big),
\q \fa N\in\sN,
\end{equation}
where $N$ denotes 
 the total number of learning episodes.
Intuitively,
this regret 
  characterizes the cumulative loss from taking sub-optimal policies 
in all  episodes.

% In the following, we shall 
%propose an learning algorithm whose regret grows sublinearly with respect 
%to the number of episodes $N$.
%

To start, let us consider a greedy algorithm,
which  chooses
  the optimal feedback control based on the current estimation of the parameter,
  and provides a sublinear regret with respect 
to the number of episodes $N$.
More precisely, let $\theta=(A,B)\in \sR^{n\t (n+k)}$ be the current estimate of $\theta^\star$, then
the controller would 
exercise the optimal feedback control $\psi^\theta\in \cV$
defined  in Theorem \ref{thm:lc_fb}
for the control problem 
\eqref{eq:lc_theta_star}-\eqref{eq:lc_sde_theta_star}
with $\theta^\star$ replaced by $\theta$,
which leads to the state process
$X^{x_0,\theta}\in \cS^2(\sR^n)$ satisfying: % the following controlled SDE:
\bb\l{eq:lc_sde_theta}
\d X_t =(A^\star X_t+B^\star{\psi}^\theta(t,X_t))\,\d t+\sigma\, \d W_t
+
\int_{\sR^p_0}
\gamma(u)\, \tilde{N}(\d t,\d u),
 \q t\in [0,T],
\q X_0=x_0.
\ee
By  the martingale properties of stochastic integrals, we can then estimate 
 $\theta^\star$ based on the 
 process 
 $Z^{x_0,\theta}_t\coloneqq
\begin{pmatrix} X_t^{x_0,\theta} \\ \psi^\theta(t,X^{x_0,\theta}_t)\end{pmatrix}
$, $t\in [0,T]$, as follows:
 \bb\l{eq:theta_star_exp}
( \theta^\star)^\trans
=\bigg(
\sE\bigg[
\int_0^T
Z^{x_0,\theta}_t (Z^{x_0,\theta}_t)^\trans
\,\d t
\bigg]
\bigg)^{-1}
\sE\bigg[
\int_0^T
Z^{x_0,\theta}_t (\d X^{x_0,\theta}_t)^\trans
\bigg],
 \ee
 provided that 
%the matrix 
$ \sE\big[
\int_0^T
Z^{x_0,\theta}_t (Z^{x_0,\theta}_t)^\trans
\,\d t
\big]\in \sR^{(n+k)\t (n+k)}$ is invertible.
This motivates us to introduce an iterative procedure to estimate $\theta^{\star}$,
where the expectations in \eqref{eq:theta_star_exp} 
are replaced by 
 empirical averages over independent realizations.
More precisely, let 
$m\in \sN$ and 
$(X^{x_0,\theta,i}_t, \psi^\theta(t,X^{x_0,\theta,i}_t))_{t\in [0,T]}$, $i=1,\ldots,m$,
be trajectories of $m$ independent  realizations of the state and control processes,
we shall update the estimate $\theta$,
denoted by $\hat{\theta}$,
  according to 
\eqref{eq:theta_star_exp}:
\bb
\l{eq:theta_sum}
\hat{\theta}^\trans
\coloneqq 
\bigg(
\frac{1}{m}\sum_{i=1}^m
\int_0^T
Z^{x_0,\theta,i}_t (Z^{x_0,\theta,i}_t)^\trans
\,\d t
+\frac{1}{m}\sI_{n+k}
\bigg)^{-1}
\bigg(
\frac{1}{m}\sum_{i=1}^m
\int_0^T
Z^{x_0,\theta,i}_t (\d X^{x_0,\theta,i}_t)^\trans
\bigg),
\ee
where 
 $Z^{x_0,\theta,i}_t\coloneqq
\begin{pmatrix} X^{x_0,\theta,i}_t \\ \psi^\theta(t,X^{x_0,\theta,i}_t)
\end{pmatrix}
$
for all $t\in [0,T]$ and $i=1,\ldots,m$,
and 
$\sI$ is the $(n+k)\t (n+k)$ identity matrix
used to ensure the existence of the required matrix inverse.
This leads to the following greedy least-squares  (GLS) algorithm:

\begin{algorithm}[H]
  \caption{\textbf{Greedy least-squares  (GLS) algorithm}}
  \label{alg:greedy}
\begin{algorithmic}[1]
  \STATE \textbf{Input}: 
  Choose an initial estimation  
 $ \theta_0$ of $\theta^\star$
 and numbers of learning episodes $\{m_\ell\}_{\ell\in \sN\cup\{0\}}$. 
% the Number of episodes in each iteration $m_j>0$ for $j=0,1,\dots$.%exploration sequence $\sigma_j>0$ ($j\geq 0$).

 \FOR {$\ell=0, 1, \cdots$}
  \STATE 
  Obtain the  optimal feedback control $ \psi^{ \theta_\ell}$
  for \eqref{eq:lc_theta_star}-\eqref{eq:lc_sde_theta_star} with $\theta^\star= \theta_\ell$
  as in Theorem \ref{thm:lc_fb}.
 % \STATE Solve \eqref{eq:lc_fbsde} with $A^\star$, $B^\star$ and $\phi^\star$ replaced by $\hat{A}$, $\hat{B}$ and $\hat{\phi}$, respectively, to get $(\hat{X},\hat{Y},\hat{Z})$.
  \STATE  Execute the feedback control $\psi^{ \theta_\ell}$ for $m_\ell$ indepenent episodes, 
   and collect the trajectory data
   $(X^{x_0,\theta_\ell,i}_t, \psi^{\theta_\ell}(t,X^{x_0,\theta_\ell,i}_t))_{t\in [0,T]}$, $i=1,\ldots,m_\ell$.

    \STATE Obtain an updated estimation ${\theta}_{\ell+1}$ by
    using  \eqref{eq:theta_sum} and the $m_\ell$ trajectories collected above.
\ENDFOR 
\end{algorithmic}
\end{algorithm}

%\subsection{Standing assumptions and main results}
\subsection{Structural assumptions for learning problems}

In this section, we analyze the regret of Algorithm \ref{alg:greedy}
based on the 
following assumptions 
of the learning  problem \eqref{eq:lc_theta_star}-\eqref{eq:lc_sde_theta_star}.

\begin{Assumption}
\phantomsection
\l{assum:lc_rl}

%Assume the setting of (H.\ref{assum:lc_ns}) and 
%
\begin{enumerate}[(1)]
\item \l{item:fg}
Let $x_0\in \sR^n$,
$\theta^\star=(A^\star, B^\star)
\in \sR^{n\t (n+k)}$,
$\sigma\in \sR^{n\t d}$,
$\gamma:\sR^p_0\to \sR^n$,
$f:[0,T]\t \sR^n\t \sR^k\to\sR\cup \{\infty\}$
and $g:\sR^n\to\sR$ 
satisfy  (H.\ref{assum:lc_ns})
with some constants $L\ge 0$ and $\lambda>0$.
%
%
%\item \l{item:non_degenerate}
% $\theta^\star$ is identifiable, i.e.,
%the  optimal control 
% $\a^{x_0,\star}\in \cH^2(\sR^k)$ 
%and 
%  the  optimal
% state process $X^{x_0,\theta^\star,\a^{\star}}\in \cS^2(\sR^n)$
% of 
%\eqref{eq:lc_theta_star}-\eqref{eq:lc_sde_theta_star}
%(with initial state $x_0$ and parameter $\theta^\star$)
% satisfy the following linear independence condition:
%if $u_1\in \sR^n$ and  $u_2\in \sR^k$
%satisfy 
% $u_1^\trans X^{x_0,\theta^\star,\a^{\star}}_t +u_2^\trans  \a^{x_0,\star}_t=0$ for  
% $\d \sP \otimes \d t$ a.e.,
% then $u_1$ and $u_2$ are zero vectors.
 
\item \l{item:jump}
There exist  
 $\gamma_{\max}\ge 0$
 and
 $\vartheta\in [0,1]$
 such that 
 $\sup_{q\ge 2}
q^{-\vartheta}\big(
\int_{\sR^p_0}|{\gamma}(u)|^q\,\nu(\d u)
\big)^{1/q}\le \gamma_{\max}$.

\end{enumerate}
\end{Assumption}

\begin{Remark}
\label{rmk:rl_assum}
Condition 
(H.\ref{assum:lc_rl}\ref{item:fg}) 
implies that 
for each $\theta=(A,B)$, the control problem 
of \eqref{eq:lc_theta_star}-\eqref{eq:lc_sde_theta_star}
with $\theta^\star$ replaced by $\theta$
is a nonsmooth linear-convex control problem studied in Section 
\ref{sec:lipschitz_stable}.

Condition (H.\ref{assum:lc_rl}\ref{item:jump}) 
describes the large jumps of the pure jump process
$L_t\coloneqq \int_0^t \int_{\sR^p_0} \gamma(u)\, \tilde{N}(\d s,\d u)$, $t\in[0,T]$,
which 
enables
estimating the tail behaviour of the state process $X^\theta$,
and subsequently
quantifying the parameter estimation error of the least-squares estimator \eqref{eq:theta_sum}
(see Section \ref{sec:concentration}).
%
%ensures that the state variable $X^\theta_t$ is a sub-exponential random variable
%for all $t\in [0,T]$
%(see Proposition \ref{lemma:state_subexp}).
If
the jump coefficient $\gamma$ is bounded,
then one can easily  see 
from  $\int_{\sR^p_0}|{\gamma}(u)|^2\,\nu(\d u)<\infty$
that 
(H.\ref{assum:lc_rl}\ref{item:jump}) holds with $\vartheta=0$.
Another important case is when 
 $\gamma(u)=u$ for all $u\in \sR^p_0$,
under which the process $(L_t)_{t\in [0,T]}$ 
 is a  L\'{e}vy process   of pure jumps 
with L\'{e}vy measure $\nu(\d u)$.
In this case, 
(H.\ref{assum:lc_rl}\ref{item:jump}) holds with $\vartheta\in (0,1]$ 
if and only if 
$\big(\int_{\sR^p_0}|u|^q\,\nu(\d u)\big)^{1/q}\le \cO(q^\theta)$
as $q\to \infty$.

\end{Remark}

\begin{Assumption}
\label{assum:non_degenerate}
 $\theta^\star$ is identifiable, i.e.,
the  optimal control 
 $\a^{x_0,\star}\in \cH^2(\sR^k)$ 
and 
  the  optimal
 state process $X^{x_0,\theta^\star,\a^{\star}}\in \cS^2(\sR^n)$
 of 
\eqref{eq:lc_theta_star}-\eqref{eq:lc_sde_theta_star}
(with initial state $x_0$ and parameter $\theta^\star$)
 satisfy the following linear independence condition:
if $u_1\in \sR^n$ and  $u_2\in \sR^k$
satisfy 
 $u_1^\trans X^{x_0,\theta^\star,\a^{\star}}_t +u_2^\trans  \a^{x_0,\star}_t=0$ for  
 $\d \sP \otimes \d t$ a.e.,
 then $u_1$ and $u_2$ are zero vectors.
 
\end{Assumption}

Condition 
(H.\ref{assum:non_degenerate})
implies   that  the true parameter $\theta^\star$ can be uniquely identified
if we  observe sufficiently many trajectories of the optimal state and control processes
of \eqref{eq:lc_theta_star}-\eqref{eq:lc_sde_theta_star}.
% which ensures that
% \eqref{eq:theta_star_exp} is well-defined if the estimated parameter 
%$\theta$ is close to the true parameter $\theta^\star$
%(see Lemma \ref{Ubar_nondeg}).
Such a  self-exploration property  allows us to  design 
\emph{exploration-free} learning algorithms for \eqref{eq:lc_theta_star}-\eqref{eq:lc_sde_theta_star}.

%Such a condition 
%is naturally satisfied by most 
%finite-time horizon control problems since the optimal feedback controls 
%are typically time-dependent and linearly independent to the identity function (see 
%\cite{basei2020linear}
%for a sufficient condition of (H.\ref{assum:lc_rl}\ref{item:non_degenerate})
%in the linear-quadratic setting
%and a detailed discussion on  the self-exploration property
%of 
%finite-time horizon  stochastic control problems).

{%\color{blue}

The following proposition shows that 
 if the laws of the state processes are supported on the whole space,
then (H.\ref{assum:non_degenerate}) 
 is equivalent to a self-exploration property of 
the optimal feedback control.
The proof essentially 
follows the argument  of \cite[Lemma 6.1]{szpruch2021exploration},
%and  the  Girsanov theorem for jump-diffusion models in \cite{sorensen2020likelihood}.
and hence is omitted.

\begin{Proposition}
\label{prop:self_exploration}
Assume (H.\ref{assum:lc_rl}\ref{item:fg}).
Let
$\psi\in \cV$.
Assume that for all $t\in (0,T]$, and any 
open set $O\subset \sR^n$ with positive Lebesgue
measure, the  state process  $X^{\theta^\star,\psi}$ 
(defined by \eqref{eq:lc_sde_theta} with $\psi^\theta=\psi$) satisfies 
  that $\sP(\{\om\in \Om \mid X^{\theta^\star,\psi}_t(\om)\in O\}) >0$.
Then the following two statements are equivalent:
\begin{enumerate}[(a)]
\item\label{item:independent_open_loop}
if $u_1\in \sR^n$ and  $u_2\in \sR^k$
satisfy 
 $u_1^\trans X^{\theta^\star,\psi}_t +u_2^\trans \psi(t,X^{\theta^\star,\psi}_t)=0$ for  
 $\d \sP \otimes \d t$ a.e.,
 then $u_1$ and $u_2$ are zero vectors;
 \item\label{item:independent_closed_loop}
 if  $u_1\in \sR^n$ and $u_2\in \sR^k$ satisfy
$u_1^\trans x+u_2^\trans \psi(t,x)=0$
for almost every $(t,x)\in [0,T]\t \sR^n$,
then $u_1$ and $u_2$ are zero vectors.
\end{enumerate}
Consequently, 
suppose that 
(H.\ref{assum:lc_rl}\ref{item:fg}) holds
and 
 $\sigma \sigma^\trans$ is positive definite, 
then  (H.\ref{assum:non_degenerate})
holds if and only if 
the optimal feedback control  $\psi^{\theta^\star}$
of \eqref{eq:lc_theta_star} satisfies 
Item (b).

\end{Proposition}

%The equivalence between Items 
%\ref{item:independent_open_loop}
%and 
%\ref{item:independent_closed_loop}
%has been shown in  \cite[Lemma 6.1]{szpruch2021exploration},
%and with 
%non-degenerate Brownian motion noise,
% the  laws of the state variables have full supports 
%can be deduced from the  Girsanov theorem 
%in \cite{sorensen2020likelihood}.

Proposition \ref{prop:self_exploration} allows for more explicit expressions of
 (H.\ref{assum:non_degenerate}). 
 For instance, as shown in \cite[Proposition 3.9]{basei2020logarithmic},
for  quadratic cost functions 
 $g=0$ and 
$f(t,x,a)=x^\trans Q x +a^\trans R a$
with   positive definite matrices
$Q$ and
 $R$, 
%(see e.g., \cite{duncan1999adaptive, abbasi2011regret,mania2019certainty, basei2020logarithmic}),
(H.\ref{assum:non_degenerate})
holds if and only if $B^\star $ in  \eqref{eq:lc_sde_theta_star} is  full column rank.
Alternatively,
by \cite[Proposition 3.3]{szpruch2021exploration},
if 
 \eqref{eq:lc_theta_star}-\eqref{eq:lc_sde_theta_star}
 has a   bounded  action set,
i.e., 
 $\cR$ in  (H.\ref{assum:lc_ns}\ref{item:f0R}) has a bounded domain
$\dom \cR$  (cf.~Example \ref{example:sparse}),
then  
(H.\ref{assum:non_degenerate})
 holds if and only if the range of $\psi^{\theta^\star}$ contains $k$ linearly independent vectors. %, or equivalently it linearly spans $\sR^k$
% as shown in \cite[Proposition 3.3]{szpruch2021exploration}.
%
%See Remark \ref{rmk:hyperparameter} for  validating 
%(H.\ref{assum:non_degenerate}) numerically,
%and for designing  algorithms without (H.\ref{assum:non_degenerate}).

We remark that 
 for general linear-convex learning problems without  
 (H.\ref{assum:non_degenerate}),
 an explicit exploration is necessary for learning \cite{szpruch2021exploration} .
 Instead of merely employing greedy polices as in Algorithm \ref{alg:greedy},
they
dedicate certain episodes to actively explore the environment with some exploration policy $\psi^e$ satisfying  
Proposition \ref{prop:self_exploration}
Item \ref{item:independent_closed_loop}.
The  numbers of exploration and exploitation episodes 
are then balanced based on the performance gap 
in Theorem \ref{thm:dynamics_stable}
and the finite-sample accuracy of the   parameter estimator. 
Note, however, this explicit exploration may yield 
larger  regrets for algorithm in \cite{szpruch2021exploration}  than that in Theorem \ref{thm:regret}.
 
%avoids the  hyperparameter $m_0$, but 

}

\subsection{Main results on sublinear regret bounds}

We now state the main result % of this section, 
which shows that 
the regret of 
Algorithm \ref{alg:greedy}  grows at most sublinearly 
with respect to the number of  episodes,
provided that 
%the least-squares estimator 
%\eqref{eq:theta_sum}
%satisfies a suitable concentration inequality,
%and
the  hyper-parameters ${\theta}_0$ and $\{m_j\}_{j\in \sN\cup \{0\}}$
are chosen properly.
In particular, we shall 
choose an initial guess $\theta_0$ of $\theta^\star$
which satisfies 
the identifiability condition in 
(H.\ref{assum:non_degenerate})
%(H.\ref{assum:lc_rl}\ref{item:non_degenerate}),
and we shall also 
 double
  the number of learning episodes
 between two successive updates of the estimation of $\theta^\star$,
 which is a  commonly used strategy (the so-called doubling trick)
in the design of online learning algorithms
 (see e.g.~\cite{
 basei2020logarithmic}).
The proof of this theorem is given in Section \ref{sec:proof_regret}.

To simplify the notation,
we introduce the following quantifies 
for each $x_0\in \sR^n$,
$\theta=(A,B) \in \sR^{n\t (n+k)}$
and $m\in \sN$:
\bb\l{eq:UV}
\begin{alignedat}{3}
\ol{U}^{x_0,\theta}
&\coloneqq\sE\bigg[
\int_0^T
Z^{x_0,\theta}_t (Z^{x_0,\theta}_t)^\trans
\,\d t
\bigg],
\q 
&
\ol{V}^{x_0,\theta}
&\coloneqq\sE\bigg[
\int_0^T
Z^{x_0,\theta}_t (\d X^{x_0,\theta}_t)^\trans
\bigg],
\\
U^{x_0,\theta,m}
&\coloneqq\frac{1}{m}\sum_{i=1}^m
\int_0^T
Z^{x_0,\theta,i}_t (Z^{x_0,\theta,i}_t)^\trans
\,\d t,
\q 
&
V^{x_0,\theta,m}
&\coloneqq\frac{1}{m}\sum_{i=1}^m
\int_0^T
Z^{x_0,\theta,i}_t (\d X^{x_0,\theta,i}_t)^\trans,
\end{alignedat}
\ee
where 
  $X^{x_0,\theta}\in \cS^2(\sR^n)$ is
the solution of \eqref{eq:lc_sde_theta},
 $(X^{x_0,\theta,i})_{i=1}^m$ are 
independent copies of $X^{x_0,\theta}$,
and 
 $Z^{x_0,\theta}$ 
and
$(Z^{x_0,\theta,i})_{i=1}^m$ are defined as in
\eqref{eq:theta_star_exp} and 
\eqref{eq:theta_sum}, respectively.
For any given  symmetric matrix $A$,
we denote by 
$\lambda_{\min}(A)$  the smallest eigenvalue 
of $A$.

\begin{Theorem}
\l{thm:regret}
Suppose (H.\ref{assum:lc_rl}\ref{item:fg})
and (H.\ref{assum:non_degenerate})
 hold.
Assume further that 
% $\theta_0$ is  identifiable as in 
% (H.\ref{assum:lc_rl}\ref{item:non_degenerate}),
$\lambda_{\min}(
\ol{U}^{x_0,\theta_0}
)>0$,
 and for any given bounded set $\cK\subset \sR^{n\t (n+k)}$,
 there exist
 constants  $C_1,C_2>0$ and $\b\ge 1$,
such that the following concentration inequality
holds 
 for all $\eps\ge 0$, $m\in \sN$ and $\theta\in\mathcal{K}$,
\begin{align}
\l{eq:concentration_beta}
\begin{split}
&\max\Big\{
\sP\big(\big|U^{x_0,\theta,m}-\ol{U}^{x_0,\theta}\big|\ge \eps \big),
\sP\big(\big|V^{x_0,\theta,m}-\ol{V}^{x_0,\theta}\big|\ge \eps \big)
\Big\}
\\
&\q \le 
C_2\exp\Big(-C_1\min
\Big\{
\frac{m\eps^2}{C_2^2},
\Big(
\frac{m\eps}{C_2}
\Big)^{\frac{1}{\beta}}
\Big\}\Big).
\end{split}
\end{align}
 Then
there exists a constant $C_0>0$, such that 
 for all $C\ge C_0$ and $\delta\in (0,1/4)$,
 if we set $m_0=C(-\ln \delta)^\b$ and $m_\ell=2^\ell m_0$ for all $\ell\in \sN$,
 then  the regret of Algorithm \ref{alg:greedy} 
  (cf.~\eqref{eq:regret})
 satisfies the following properties:
\begin{enumerate}[(1)]
 \item\l{item:regret_PAC} 
 It holds   with probability at least $1-4\delta$ that 
 $R(N)\le C'\big(
 \sqrt{N}\sqrt{ \ln {N}}
+\sqrt{-\ln\delta }
\sqrt{N}
+(-\ln\delta)^{\beta}\ln {N}
 \big)$ for all  $N\in \sN$,
where  $C'$ is  a constant  independent of $N$ and $\delta$. 
\item\l{item:regret_as}
 It holds   with probability 1 that  $R(N)=\cO(\sqrt{N\ln N})$ as $N\to \infty$.
 \end{enumerate}

\end{Theorem}

The following theorem presents
a precise sublinear regret bound 
of
Algorithm \ref{alg:greedy}
 for the jump-diffusion model \eqref{eq:lc_sde_theta_star},
 depending on the jump sizes of the Poisson random measure.
The proof follows
 from 
 Theorem 
\ref{thm:regret} and
 Proposition \ref{prop:concentration_jump}.

\begin{Theorem}
\l{cor:regret_jump}
Suppose (H.\ref{assum:lc_rl}) 
and (H.\ref{assum:non_degenerate})
hold,
and 
% $\theta_0$ is  identifiable as in 
% (H.\ref{assum:lc_rl}\ref{item:non_degenerate}).
$\lambda_{\min}(
\ol{U}^{x_0,\theta_0}
)>0$.
 Then
there exists a constant $C_0>0$, such that 
 for all $C\ge C_0$ and $\delta\in (0,1/4)$,
 if we set $m_0=C(-\ln \delta)^{3+\vartheta}$ and $m_\ell=2^\ell m_0$ for all $\ell\in \sN$,
 then  the regret of Algorithm \ref{alg:greedy} 
  (cf.~\eqref{eq:regret})
 satisfies the following properties:
 \begin{enumerate}[(1)]
 \item%\l{item:regret_PAC} 
 It holds   with probability at least $1-4\delta$ that 
 $R(N)\le C'\big(
 \sqrt{N}\sqrt{ \ln {N}}
+\sqrt{-\ln\delta }
\sqrt{N}
+(-\ln\delta)^{3+\vartheta}\ln {N}
 \big)$ for all  $N\in \sN$,
where
$\vartheta$ is the constant in (H.\ref{assum:lc_rl}\ref{item:jump}) 
and
  $C'$ is  a constant  independent of $\vartheta, N$ and $\delta$. 
\item%\l{item:regret_as}
 It holds   with probability 1 that  $R(N)=\cO(\sqrt{N\ln N})$ as $N\to \infty$.
 \end{enumerate}

\end{Theorem}

In the case 
where 
% the state process
 \eqref{eq:lc_sde_theta_star}
  is only driven by the Brownian motion,
we can
exploit the sub-Gaussianity of the state process
and 
obtain a shaper regret bound
based on 
 Theorem 
\ref{thm:regret} and
 Proposition \ref{prop:concentration_diffusion}.
%Such a regret bound improves \cite[Theorem 2]{basei2020linear}
%(established for LQ problems) in terms of its dependence on $\ln\delta$ and $\ln N$.

\begin{Theorem}
\l{cor:regret_diffusion}
Suppose (H.\ref{assum:lc_rl}) 
and (H.\ref{assum:non_degenerate})
hold
with $\gamma_{\max}=0$,
and 
% $\theta_0$ is  identifiable as in 
% (H.\ref{assum:lc_rl}\ref{item:non_degenerate}).
$\lambda_{\min}(
\ol{U}^{x_0,\theta_0}
)>0$.
 Then
there exists a constant $C_0>0$, such that 
 for all $C\ge C_0$ and $\delta\in (0,1/4)$,
 if we set $m_0=C(-\ln \delta)$ and $m_\ell=2^\ell m_0$ for all $\ell\in \sN$,
 then  the regret of Algorithm \ref{alg:greedy} 
  (cf.~\eqref{eq:regret})
 satisfies the following properties:
 \begin{enumerate}[(1)]
 \item
 It holds   with probability at least $1-4\delta$ that 
 $R(N)\le C'\big(
 \sqrt{N}\sqrt{ \ln {N}}
+\sqrt{-\ln\delta }
\sqrt{N}
+(-\ln\delta)\ln {N}
 \big)$ for all  $N\in \sN$,
where  $C'$ is  a constant  independent of $N$ and $\delta$. 
\item
 It holds   with probability 1 that  $R(N)=\cO(\sqrt{N\ln N})$ as $N\to \infty$.
 \end{enumerate}

\end{Theorem}

{%\color{blue}

\begin{Remark}\label{rmk:hyperparameter}
The condition $\lambda_{\min}(
\ol{U}^{x_0,\theta_0}
)>0$
in Theorems 
\ref{cor:regret_jump}
and
\ref{cor:regret_diffusion}
ensures that 
the greedy policy $\psi^{\theta_0}$ explores the parameter space and improves the accuracy of parameter estimation. 
%As mentioned in Remark \ref{rmk:rl_assum},
By Proposition \ref{prop:self_exploration},
if  \eqref{eq:lc_sde_theta} has  nondegenerate Brownian noises,
then 
it suffices to choose $\theta_0$ such that 
the corresponding greedy policy  $\psi^{\theta_0}$ 
enjoys the exploration property
 stated in Item \ref{item:independent_closed_loop}.
%\eqref{exploration}.

The choice of  $m_0=C_0(-\ln \delta)^\beta$ 
along with (H.\ref{assum:non_degenerate})
%(H.\ref{assum:lc_rl}\ref{item:non_degenerate})
ensures that 
$(\theta_\ell)_{\ell\in \sN}$ stays close to $\theta^\star$
so that \eqref{eq:concentration_beta}
is applicable. 
Here
 $\delta$ is an arbitrarily small constant indicating  the agent's confidence   of the regret bound, 
and $C_0$ is a  constant depending  on 
the  exploration strength of $\psi^{\theta^\star}$,
namely the constant 
$\lambda_{\min}(
\ol{U}^{x_0,\theta^\star})>0$ 
(see    Section \ref{sec:proof_regret}).
Note that
to analyze algorithm regrets,
 it is common to 
 assume some a-priori information on the true parameter 
 and the 
  algorithm being 
  initialized 
   with sufficiently many 
learning episodes
(see e.g., \cite{dean2018regret}).
 Obtaining
an explicit dependence of $C_0$ on model parameters, however, could be challenging.
A practical strategy for validating  
(H.\ref{assum:non_degenerate})
%(H.\ref{assum:lc_rl}\ref{item:non_degenerate})
and  for choosing the initial episode $m_0$
is to ensure that the obtained  estimations
$(\theta_\ell)_{\ell\in \sN}$ 
remain bounded and  that
the resulting greedy policies $(\psi^{\theta_\ell})_{\ell\in \sN}$
satisfy
Proposition \ref{prop:self_exploration}
Item \ref{item:independent_closed_loop}.
Our numerical experiments in Section \ref{sec:numerical}
 demonstrate that 
the performance of Algorithm \ref{alg:greedy}  is stable 
with respect to $m_0$, and that
a  small $m_0$  in general suffices to  guarantee a sublinear regret. 

%Finally, we refer the reader to a recent work \cite{szpruch2021exploration}
% for general linear-convex learning problems without  
% (H.\ref{assum:non_degenerate}).
%% (H.\ref{assum:lc_rl}\ref{item:non_degenerate}).
% Instead of merely employing greedy polices as Algorithm \ref{alg:greedy},
% \cite{szpruch2021exploration} 
%dedicates certain episodes to actively explore the environment with some exploration policy $\psi^e$ satisfying  \eqref{exploration}.
%The precise number of exploration episodes 
%is chosen based on the performance gap 
%in Theorem \ref{thm:dynamics_stable}
%and the finite-sample accuracy of the corresponding  parameter estimator. 
%Due to the explicit exploration,
%the algorithm in \cite{szpruch2021exploration} 
%avoids the  hyperparameter $m_0$, but 
%the resulting regret may be larger than that in Theorem \ref{thm:regret}.

\end{Remark}
}
 
\subsection{Proofs of sublinear regret bounds}
This section is devoted to the proofs of 
%Theorem \ref{thm:regret}
%and Theorems \ref{cor:regret_jump}-\ref{cor:regret_diffusion}.
Theorem \ref{thm:regret}, \ref{cor:regret_jump} and \ref{cor:regret_diffusion}.

As we have seen in Theorems \ref{cor:regret_jump}-\ref{cor:regret_diffusion},
an essential step for estimating the regret of Algorithm \ref{alg:greedy}
is to establish the concentration inequality  \eqref{eq:concentration_beta}
for the least-squares estimator \eqref{eq:theta_sum}.
Compared to the classical learning problems with Brownian-motion-driven state dynamics
(see e.g.~\cite{
 basei2020logarithmic}),
 the presence of jumps in the state dynamics creates a crucial difficulty in quantifying the precise value 
of $\b$ in \eqref{eq:concentration_beta},
 since the state variable $X^{\theta}$ is in general not sub-Gaussian, and hence 
\eqref{eq:concentration_beta}
  does not hold with $\b=1$.

 In the subsequent analysis, 
 we overcome the above difficulty 
 by introducing a notation of sub-Weibull random variables 
 as in \cite{kuchibhotla2018moving}
 and 
 establishing that 
both deterministic and stochastic  integrals
preserve 
 sub-Weibull random variables
in Section \ref{sec:sub-Weibull}. 
We then show  in   Section \ref{sec:concentration}
that  \eqref{eq:theta_sum} behaves like sub-Weibull random variables and  \eqref{eq:concentration_beta}
 holds with some $\b\ge 1$, 
 provided that the jumps of the state dynamics are sub-exponential.
Finally, we prove the general regret result 
Theorem \ref{thm:regret} for Algorithm \ref{alg:greedy} 
in Section \ref{sec:proof_regret}.

\subsubsection{Step 1: Analysis of sub-Weibull random variables}
\l{sec:sub-Weibull}

The first step is to analyze
integrals of 
sub-Weibull random variables.
%which is essential for the subsequent analysis of
%the least-squares estimator \eqref{eq:theta_sum}
%in a  jump-diffusion setting.
%
We start by  recalling the precise definition
of sub-Weibull random variables
in terms of their Orlicz norms (see \cite{kuchibhotla2018moving}).
\begin{Definition}
\label{definition:sub-wellbull}
For every $\a>0$,
let  $\Psi_\a:[0,\infty)\to \sR$ such that  $\Psi_\a(x)=e^{x^\a}-1$ for all $x\ge 0$,
and 
let $\|\cdot\|_{\Psi_\a}$ be the corresponding 
$\Psi_\a$-Orlicz (quasi-)norm such that for any given  random variable $X$,
$$\|X\|_{\Psi_\a}\coloneqq \inf\left\{t>0\mid \sE\left[\Psi_\a\left(\tfrac{|X|}{t}\right)\right]\le 1\right\}.$$
Then 
a random variable $X$ is said to be sub-Weibull of order
$\a>0$, denoted by $X\in \subW(\a)$, if
$\|X\|_{\Psi_\a}<\infty$.
\end{Definition}

Note that  $\|\cdot\|_{\Psi_\a}$ is a norm if and only if $\a\ge 1$, as otherwise the triangle inequality does not hold.
%One can show that
%if $X$ is $\subW(\a)$ with some $\a>0$, then
%$\sP(|X|>t)\le 2\exp(-{t^\a}/{\|X\|^\a_{\Psi_\a}})$ for all $t\ge 0$.
Examples of sub-Weibull random variables include 
sub-Gaussian and 
sub-exponential  random variables,
which correspond to 
$\subW(2)$ and $\subW(1)$, respectively. 
We point out that 
the class of sub-Weibull random
variables is closed under multiplication and addition,
and for all $\a>0$, there exists a constant $C_\a$, depending only on $\a$, such that 
\bb\l{eq:sub_weibull_lp}
C_\a^{-1}\sup_{q\ge 1}q^{-1/\a}\|X\|_{L^q}
\le \|X\|_{\Psi_\a}
\le 
C_\a\sup_{q\ge 1}q^{-1/\a}\|X\|_{L^q}
\ee
for all random variables $X$
(see
\cite[Appendix A]{gotze2021concentration}  
for  a proof of these  properties).

We now present several important lemmas regarding the behavior of integrals of sub-Weibull random variables.
The first lemma shows that 
 deterministic integral of a product of sub-Weibull random
variables is still sub-Weibull. 
The proof is based on 
Definition \ref{definition:sub-wellbull}
and H\"{o}lder's inequality,
and is given in Appendix \ref{appendix:technical results}.

\begin{Lemma}\l{lemma:deterministic_integral}
For all $\a>0$ and  every stochastic  process $X,Y:\Om\t [0,T]\to \sR$,
$$
\left\|\int_0^T XY\, \d t\right\|_{\Psi_{\a/2}}
\le \left\|\bigg(\int_0^T |X|^2\, \d t\bigg)^{\frac{1}{2}}\right\|_{\Psi_{\a}}
 \left\|\bigg(\int_0^T |Y|^2\, \d t\bigg)^{\frac{1}{2}}\right\|_{\Psi_{\a}}.
$$
\end{Lemma}

The second lemma shows that 
 stochastic integrals 
 preserve the property of being 
 sub-Weibull random
variables. 
The proof is based on the equivalent characterization \eqref{eq:sub_weibull_lp}
of sub-Weibull random variables 
and   Burkholder's inequality,
whose details are given in  Appendix \ref{appendix:technical results}.

\begin{Lemma}\l{lemma:stochastic_integral_jump}
There exists  $C\ge 0$
such that 
 for all
$\sigma\in \sR^{d}$,
 $X\in \cS^2(\sR)$
and  every 
measurable function
$\gamma: \sR^p_0\to \sR$
satisfying  (H.\ref{assum:lc_rl}\ref{item:jump}), 
%that 
%\begin{enumerate}[(1)]
%\item\l{item:subexp_subexp}
%$\|\int_0^T X_t\sigma(t)^\trans\, \d W_t\|_{\Psi_1}\le CK \|(\int_0^T |X_t|^2\, \d t)^{1/2}\|_{\Psi_2}$.
%\item
%\l{item:subweb_subweb}
$\|\int_0^T X_t\sigma^\trans\, \d W_t\|_{\Psi_{1/2}}\le C |\sigma|\|(\int_0^T |X|^2\, \d t)^{\frac{1}{2}}\|_{\Psi_1}$
and 
$$
\left\|\int_0^T\int_{\sR^p_0} X_t\gamma(u)\,  \tilde{N}(\d t,\d u)\right\|_{\Psi_{1/(3+\vartheta)}}
\le C\gamma_{\max}
\bigg(
\sup_{p\ge 2}
\bigg\|
\bigg(\int_0^T|X_t|^q\d t 
\bigg)^{\frac{1}{q}}
\bigg\|_{\Psi_1}
\bigg),
$$
with the constants 
$\gamma_{\max}$ and $\vartheta$  in (H.\ref{assum:lc_rl}\ref{item:jump}).

%\end{enumerate}

\end{Lemma}

%\begin{Remark}
\l{rmk:bdg}
Lemma \ref{lemma:stochastic_integral_jump}
 focuses on the case where 
$(\int_0^T |X|^2\, \d t)^{{1}/{2}}\in \subW(1)\setminus \subW(2)$,
which is important for 
control problems whose  state dynamics
is
 driven by a Poisson random measure. 
Hence we establish  the sub-Weibull properties of the stochastic integrals 
by applying 
the Burkholder's inequality to
 estimate the growth of their $L^q$-norms,
  precise order of which depends on the constants 
$C_q$  and $\tilde{C}_q$ 
in the  inequalities 
\eqref{eq:bdg_brownian}
and
 \eqref{eq:bdg_poisson}.
% Since 
% the estimates of these constants may not be sharp, 
% it is unclear whether the orders of  
% the sub-Weibull estimates in Lemma 
% \ref{lemma:stochastic_integral_jump} is optimal. 

In the case where 
$(\int_0^T |X|^2\, \d t)^{{1}/{2}}\in  \subW(2)$,
we can establish the optimal  sub-Weibull order 
$\int_0^T X_t\sigma^\trans\, \d W_t\in  \subW(1)$.
Such a characeterization  is essential for obtaining a sharper regret bound
of Algorithm \ref{alg:greedy} 
when 
the state dynamics  
is  only driven by the Brownian motion.
The proof is based on the Girsanov theorem 
and is given in  Appendix \ref{appendix:technical results}.

%\end{Remark}

\begin{Lemma}\l{lemma:stochatic_integral_diffusion}
There exists $C\ge 0$
such that 
 for all
$\sigma\in \sR^{d}$
and
 $X\in \cS^2(\sR)$,
$\|\int_0^T X_t\sigma^\trans\, \d W_t\|_{\Psi_1}\le C |\sigma|\|(\int_0^T |X|^2\, \d t)^{\frac{1}{2}}\|_{\Psi_2}$.

%\end{enumerate}

\end{Lemma}

\subsubsection{Step 2: Concentration inequalities for the least-squares estimator}
\l{sec:concentration}

Based on the fact that    sub-Weibull properties
are preserved 
under algebraic and integral operations
as shown in Section \ref{sec:sub-Weibull},
we now quantify  the precise tail behavior of the least-squares estimator \eqref{eq:theta_sum}, 
namely the constant $\b$ in \eqref{eq:concentration_beta},
for the jump-diffusion model \eqref{eq:lc_sde_theta_star}.

We start by establishing the sub-exponential properties of Lipschitz functionals
of the state process $X^{\theta}$
driven by both  Brownian motions
and  Poisson random measures
as in \eqref{eq:lc_sde_theta}.
The proof follows as a special case of 
 \cite{ma2010transportation}
and is given in Appendix \ref{appendix:technical results}.

\begin{Lemma}
\l{lemma:state_subexp}
Suppose 
(H.\ref{assum:lc_rl})
%(H.\ref{assum:lc_rl}\ref{item:fg}\ref{item:jump})
%and 
%(H.\ref{assum:lc_rl}\ref{item:jump})
holds. 
Let $K\in \sR$ and  $\theta=(A,B) \in \sR^{n\t (n+k)}$ 
satisfy $|\theta|\le K$. 
Then there exists  $C\ge 0$,
depending only on $K$, $T$ and the constants 
in 
(H.\ref{assum:lc_rl}),
%(H.\ref{assum:lc_rl}\ref{item:fg}\ref{item:jump}),
 such that for all  $x_0\in \sR^n$
and for every  Lipschitz continuous function
$\mathfrak{f}:(\sD([0,T];\sR^n),d_\infty)\to \sR$,
%with  $\|\mathfrak{f}\|_{\textnormal{Lip}}=1$,
the solution $X^{x_0,\theta}$ of \eqref{eq:lc_sde_theta}
satisfies
$\|\mathfrak{f}(X^{x_0,\theta})\|_{\Psi_1}\le C(\|\mathfrak{f}\|_{\textnormal{Lip}}+|\sE[\mathfrak{f}(X^{x_0,\theta})]|)$,
where 
$\sD([0,T];\sR^n)$ is the space of $\sR^n$-valued
 c\`adl\`{a}g functions on $[0,T]$
 endowed with the 
uniform metric
$d_\infty$, 
and
 $\|\mathfrak{f}\|_{\textnormal{Lip}}$ 
 is the Lipschitz constant 
 of $\mathfrak{f}$.
(cf.~Lemma \ref{lemma:sde_transportation}).

\end{Lemma}

We now  characterize
the parameter $\beta$ in 
the  concentration inequality
\eqref{eq:concentration_beta}
based on Lemmas \ref{lemma:deterministic_integral}, \ref{lemma:stochastic_integral_jump}
and  \ref{lemma:state_subexp}.

% for the random variables 
%in the least-squares estimator \eqref{eq:theta_sum}.
%and consequently an order $\cO(\sqrt{N})$ regret bound of Algorithm \ref{alg:greedy}.

\begin{Proposition}\l{prop:concentration_jump}
Suppose 
(H.\ref{assum:lc_rl})
%(H.\ref{assum:lc_rl}\ref{item:fg}\ref{item:jump})
%and 
%(H.\ref{assum:lc_rl}\ref{item:jump})
holds
and 
let 
%$x_0\in \sR^n$
%and 
$\cK\subset \sR^{n\t (n+k)}$
be a bounded set.
Then
 there exist constants  $C_1,C_2\ge 0$
such that 
%the inequality
\eqref{eq:concentration_beta}
holds
for all $\eps\ge 0$, $m\in \sN$ and $\theta\in \cK$
 with $\beta=3+\vartheta$,
where $\vartheta$ is the constant in (H.\ref{assum:lc_rl}\ref{item:jump}).

\end{Proposition}

\begin{proof}

Throughout this proof, 
let  $\theta$ be a given constant satisfying 
$|\theta|\le K$ for some  $K\ge 0$.
For notational simplicity,
we 
shall 
omit the dependence 
on $(x_0,\theta)$ in the subscripts
of all random variables,
and 
denote by $C_2$ a generic constant, 
which
 is independent of $m$ and the precise value of $\theta$,
and depends possibly on $K$, $x_0$,
the constants in 
%(H.\ref{assum:lc_rl}\ref{item:fg}\ref{item:jump})
(H.\ref{assum:lc_rl})
and the  dimensions.

Note that 
for each $i=1,\ldots, m$, 
the entries of $\int_0^T
Z^{i}_t (Z^{i}_t)^\trans
\,\d t$
are one of the three cases:
\begin{equation}
\l{eq:ZZ_term}
\int_0^T
X^{i}_{\ell,t}
X^{i}_{j,t}\, \d t,
\q
\int_0^T
X^{i}_{\ell,t}
\psi^\theta(t,X^{i}_t)_j\, \d t,
\q
\int_0^T
\psi^\theta(t,X^i_t)_{\ell}
\psi^\theta(t,X^i_t)_{j}\, \d t
\end{equation}
where $X^i_{\ell,t}$   and 
$ \psi^\theta(t,X^i_t)_{\ell}$
are 
 the $\ell$th-entry 
of 
$X^i_{\ell,t}$
and
$ \psi^\theta(t,X^i_t)_{\ell}$, respectively.
Similarly, the entries of 
$\int_0^T
Z^{i}_t (\d X^{i}_t)^\trans
$  are
one of the two cases:
\begin{align}
%\begin{split}
&\int_0^T
X^{i}_{\ell,t} (A^\star X^i_t)_j\,\d t+
\int_0^TX^{i}_{\ell,t} (B^\star{\psi}^\theta(t,X^i_t))_j\,\d t
+\int_0^TX^{i}_{\ell,t} \sigma_j\, \d W^i_t
+\int_0^T 
\int_{\sR^p_0}X^{i}_{\ell,t} \gamma(u)_j\, \tilde{N}^i(\d t,\d u),
\nb
\\
&
\int_0^T
\psi^\theta(t,X^i_t)_{\ell}(A^\star X^i_t)_j\,\d t+
\int_0^T\psi^\theta(t,X^i_t)_{\ell} (B^\star{\psi}^\theta(t,X^i_t))_j\,\d t
+\int_0^T\psi^\theta(t,X^i_t)_{\ell} \sigma_j\, \d W^i_t
\l{eq:ZdX_term}
\\
&\q \q 
+\int_0^T\int_{\sR^p_0}\psi^\theta(t,X^i_t)_{\ell}\gamma(u)_j\, \tilde{N}^i(\d t,\d u),
\nb
\end{align}
where $\sigma_j$ is the $j$-th row of $\sigma\in \sR^{n\t d}$,
$\gamma_j$ is the $j$-th entry of the function 
$\gamma:\sR^p_0\to \sR^n$, 
$(W^i)_{i=1}^m$ are $m$-independent $d$-dimensional Brownian motion,
and $(\tilde{N}^i)_{i=1}^m$ are 
$m$-independent 
compensated Poisson random measures.
By 
the definitions of
 $U^{x_0,\theta,m}, V^{x_0,\theta,m}$ in \eqref{eq:UV},
and
 the inequality that 
$\sP(|\sum_{i=1}^\ell X_i|\ge \eps)
\le \sum_{i=1}^\ell
\sP(|X_i|\ge \eps/\ell)$
for all
$\ell\in \sN$ and 
 random variables $(X_i)_{i=1}^\ell$, 
 it suffices to obtain a concentration inequality for each term 
 in 
 \eqref{eq:ZZ_term} and  \eqref{eq:ZdX_term}.

%By virtue of  the condition that
Since $|\theta|\le K$, 
%we can obtain  from 
by
Theorem  \ref{thm:lc_fb},
there exists  $C_2\ge 0$
such that 
$|\psi^\theta(t,0)|\le C_2$ and $|\psi^\theta(t,x)-\psi^\theta(t,x')|\le C_2|x-x'|$
  for all  $t\in [0,T]$, $x,x'\in \sR^n$.
  Then standard moment estimates of 
\eqref{eq:lc_sde_theta} 
(with the initial condition $x_0$) shows that 
$\|X^i\|_{\cS^2(\sR^n)}\le C_2$ for all $i=1,\ldots, m$,
with a constant $C_2$ depending on $x_0$.  
Then,
for each 
$q\ge 2$,
$\ell=1,\ldots, n$
and $j=1,\ldots k$, 
we consider the functions 
${\mathfrak{f}}^{(q)}_{\ell}, 
\ol{\mathfrak{f}}^{(q)}_j:
(\sD([0,T];\sR^n),d_\infty)\to \sR$
satisfying for all $\rho\in 
\sD([0,T];\sR^n)$ that 
$\mathfrak{f}^{(q)}_\ell(\rho)=\big(\int_0^T|\rho_{\ell,t}|^q\d t 
\big)^{\frac{1}{q}}
$
and 
$\ol{\mathfrak{f}}^{(q)}_j(\rho)=\big(\int_0^T|\psi^\theta(t,\rho_{t})_j|^q\d t 
\big)^{\frac{1}{q}}
$,
where $\rho_{\ell,t}$ is the $\ell$th component of $\rho_t$
and $\psi^\theta(t,\rho_{t})_j$ is the $j$th component of $\psi^\theta(t,\rho_{t})$.
One can easily show that 
$\mathfrak{f}^{(q)}_\ell(0)=0$
and 
$
|\ol{\mathfrak{f}}^{(q)}_j(0)|,
\|\mathfrak{f}^{(q)}_\ell\|_{\textnormal{Lip}}, 
\|\ol{\mathfrak{f}}^{(q)}_j\|_{\textnormal{Lip}} \le C$,
which along with 
Lemma \ref{lemma:state_subexp}
implies that 
$
\|\big(\int_0^T|X^i_{\ell,t}|^q\,\d t 
\big)^{\frac{1}{q}}
\|_{\Psi_1}\le C
$
and 
$
\|
\big(\int_0^T
|\psi^\theta(t,X^i_t)_{j}|^q\,\d t
\big)^{\frac{1}{q}}
\|_{\Psi_1}\le C
$,
uniformly with respect to $i,\ell,j,q, \theta$.
Hence, we can obtain
from Lemmas \ref{lemma:deterministic_integral} and \ref{lemma:stochastic_integral_jump}
 a uniform bound for the 
$\|\cdot\|_{\Psi_{1/(3+\vartheta)}}$-norms
of all the terms in   \eqref{eq:ZZ_term} and  \eqref{eq:ZdX_term}.

Consequently, 
we can deduce the desired concentration inequality
by 
applying Lemma \ref{lemma:concentration_alpha_subWeibull}
(with 
$\a=1/(3+\vartheta)$, 
$N=m$ and $\eps'=m\eps$)
to each component of 
the zero-mean random varables
$\big(\int_0^T
Z^{i}_t (Z^{i}_t)^\trans
\,\d t-\ol{U}\big)_{i=1}^m$
and 
$
\big(
\int_0^T
Z^{i}_t (\d X^{i}_t)^\trans
-\ol{V}\big)_{i=1}^m
$.
\end{proof}

The following proposition improves the concentration inequality
in Proposition \ref{prop:concentration_jump}
for the case without jumps.

\begin{Proposition}\l{prop:concentration_diffusion}
Suppose 
(H.\ref{assum:lc_rl})
%(H.\ref{assum:lc_rl}\ref{item:fg}\ref{item:jump})
holds
with
$\gamma_{\max}=0$
and 
let 
%$x_0\in \sR^n$
%and 
$\cK\subset \sR^{n\t (n+k)}$
be a bounded set.
Then 
there exist
 constants  $C_1,C_2\ge 0$
such that 
\eqref{eq:concentration_beta}
holds
for all $\eps\ge 0$, $m\in \sN$ and $\theta\in  \cK$
 with $\beta=1$.

\end{Proposition}

\begin{proof}
We first refine the result of  Lemma \ref{lemma:state_subexp}
and prove 
Lipschitz functionals
of 
 the state process $X^{x_0,\theta}$ is sub-Gaussian. 
%Since (H.\ref{assum:lc_rl}\ref{item:jump}) holds
% with $\gamma_{\max}=0$,
% we have $\gamma(u)=0$, $\nu$-a.e.~on  $ \sR^p_0$. 
By \cite[Theorem 1.1 and Corollary 4.1]{djellout2004transportation}, 
there exists  $C\ge 0$
such that 
for all $x_0\in \sR^n$ and 
  for 
 every 
  Lipschitz continuous function
$\mathfrak{f}:(\sD([0,T];\sR^n),d_\infty)\to \sR$
with $\|\mathfrak{f}\|_{\textnormal{Lip}}\le 1$,
$\sE\big[\exp\big(\lambda(\mathfrak{f}(X^{x_0,\theta})-\sE[\mathfrak{f}(X^{x_0,\theta})])\big)\big]
\le 
\exp\big(C^2\lambda^2\big)
$
 for all $\lambda>0$,
which along with
 \cite[Proposition 2.5.2 (v)]{vershynin2018high}
 implies that 
 $\|\mathfrak{f}(X^{x_0,\theta})-\sE[\mathfrak{f}(X^{x_0,\theta})]\|_{\Psi_2}\le C$
for some constant $C$,
uniformly with respect to $x_0\in \sR^n$, $\theta\in \cK$ and 
$\mathfrak{f}:(\sD([0,T];\sR^n),d_\infty)\to \sR$
satisfying $\|\mathfrak{f}\|_{\textnormal{Lip}}\le 1$.
Then, we can deduce from the fact that 
$\|\cdot\|_{\Psi_2}$ is a norm 
that 
$\|\mathfrak{f}(X^{x_0,\theta})\|_{\Psi_2}\le C(\|\mathfrak{f}\|_{\textnormal{Lip}}+|\sE[\mathfrak{f}(X^{x_0,\theta})]|)$
for all  $x_0\in \sR^n$, $\theta\in \cK$ and 
Lipschitz continuous functions 
$\mathfrak{f}$.

We then proceed along
the proof of
Proposition \ref{prop:concentration_jump}.
For each
 $i=1,\ldots, m$, 
all  entries of $\int_0^T
Z^{i}_t (Z^{i}_t)^\trans
\,\d t$
are given  in  \eqref{eq:ZZ_term},
and 
all  entries of 
$\int_0^T
Z^{i}_t (\d X^{i}_t)^\trans
$  are
given by
(cf.~\eqref{eq:ZdX_term}):
\begin{align}\l{eq:ZdX_term_diffusion}
\begin{split}
&\int_0^T
X^{i}_{\ell,t} (A^\star X^i_t)_j\,\d t+
\int_0^TX^{i}_{\ell,t} (B^\star{\psi}^\theta(t,X^i_t))_j\,\d t
+\int_0^TX^{i}_{\ell,t} \sigma_j\, \d W^i_t,
\\
&
\int_0^T
\psi^\theta(t,X^i_t)_{\ell}(A^\star X^i_t)_j\,\d t+
\int_0^T\psi^\theta(t,X^i_t)_{\ell} (B^\star{\psi}^\theta(t,X^i_t))_j\,\d t
+\int_0^T\psi^\theta(t,X^i_t)_{\ell} \sigma_j\, \d W^i_t,
\end{split}
\end{align}
for all 
$\ell=1,\ldots, n$
and
$j=1,\ldots, k$,
where we have 
omitted the dependence 
on $(x_0,\theta)$ in the subscripts
for notational simplicity. 
Hence,
by following the same argument as in Proposition \ref{prop:concentration_jump},
we can show there exists a constant $C$, such that 
for all 
$i=1,\ldots, m$, 
$\ell=1,\ldots, n$, 
$j=1,\ldots, k$
and $\theta\in \cK$,
we have 
 $
\|\big(\int_0^T|X^i_{\ell,t}|^2\,\d t 
\big)^{\frac{1}{2}}
\|_{\Psi_2}\le C
$
and 
$
\|
\big(\int_0^T
|\psi^\theta(t,X^i_t)_{j}|^2\,\d t
\big)^{\frac{1}{2}}
\|_{\Psi_2}\le C
$.
Then, we can obtain
from Lemmas \ref{lemma:deterministic_integral} and \ref{lemma:stochatic_integral_diffusion}
 a uniform bound for the 
$\|\cdot\|_{\Psi_{1}}$-norms
of all entires of 
$\int_0^T
Z^{i}_t (Z^{i}_t)^\trans
\,\d t$
and $\int_0^T
Z^{i}_t (\d X^{i}_t)^\trans
$.
Consequently, 
we can 
apply Lemma \ref{lemma:concentration_alpha_subWeibull}
(with $\a=1$, $N=m$ and $\eps'=m\eps$)
to each entry of 
$\big(\int_0^T
Z^{i}_t (Z^{i}_t)^\trans
\,\d t-\ol{U}\big)_{i=1}^m$
and 
$
\big(
\int_0^T
Z^{i}_t (\d X^{i}_t)^\trans
-\ol{V}\big)_{i=1}^m
$,
and deduce the desired concentration inequality.
\end{proof}

\subsubsection{Step 3: Proof of general regret bounds}
\l{sec:proof_regret}

After demonstrating how to verify 
\eqref{eq:concentration_beta}
based on the precise jump sizes in the state dynamics,
it remains to  establish
the general regret result in 
Theorem \ref{thm:regret}
under the assumption that 
\eqref{eq:concentration_beta}
holds for 
some $\b\ge 1$.

We start by showing that under (H.\ref{assum:lc_rl}\ref{item:fg}) and (H.\ref{assum:non_degenerate}),
the expression 
\eqref{eq:theta_star_exp} is well-defined 
if  $\theta$ is a sufficiently accurate estimation of the true parameter $\theta^\star$.

\begin{Lemma}\label{Ubar_nondeg} 
Suppose (H.\ref{assum:lc_rl}\ref{item:fg}) and (H.\ref{assum:non_degenerate}) hold.
Then there exist constants  $\eps_0>0$
and $\tau_0>0$, such that 
for all 
$\theta\in \mathcal{K}_0\coloneqq\{\theta\in \sR^{n\t (n+k)}\mid |\theta-\theta^\star|\leq \eps_0\}$, 
we have 
$\lambda_{\min}(
\ol{U}^{x_0,\theta}
)\ge \tau_0$,
where 
$\ol{U}^{x_0,\theta}$ 
is defined as in
\eqref{eq:UV}
and $\lambda_{\min}(A)$ is the smallest eigenvalue 
of a symmetric matrix $A$.

 \end{Lemma}
\begin{proof}
Since $\ol{U}^{x_0,\theta^\star}$ is positive semidefinite,
we shall prove  $\lambda_{\min}(
\ol{U}^{x_0,\theta^\star})>0$
by assuming that 
 $\lambda_{\min}(\ol{U}^{x_0,\theta^\star})=0$.
Then we see 
  there exists a non-zero vector $u=\begin{pmatrix}
  u_1\\ u_2
  \end{pmatrix}\in\mathbb{R}^{n+k}$
 with $u_1\in\mathbb{R}^n$ and $u_2\in\mathbb{R}^k$,
 such that 
 $u^\trans \ol{U}^{x_0,\theta^\star}u=0$.
By  the definition of 
$\ol{U}^{x_0,\theta^\star}$ in \eqref{eq:UV},
we can deduce that 
$\sE[\int_0^T|u^\trans Z^{x_0,\theta^\star}_t |^2\,\d t]=0$,
which along with 
the definition of  $Z^{x_0,\theta^\star}_t$ in \eqref{eq:theta_star_exp}
implies for
 $\d \sP \otimes \d t$ a.e.~that 
 $u_1^\trans X^{x_0,\theta^\star,\a^{\star}}_t +u_2^\trans  \a^{x_0,\star}_t=0$.  
This contradicts to (H.\ref{assum:non_degenerate}),
%(H.\ref{assum:lc_rl}\ref{item:non_degenerate}),
which leads to the desired inequality that $\lambda_{\min}(
\ol{U}^{x_0,\theta^\star})>0$.

 We then show that 
 the map 
%  $ \sR^{n\t (n+k)}\ni \theta\mapsto  Z^{x_0,\theta} (Z^{x_0,\theta})^\trans\in \cH^2(\sR^{(n+k)\t (n+k)})$ 
$ \sR^{n\t (n+k)}\ni \theta\mapsto  \ol{U}^{x_0,\theta}\in \sR$ 
 is continuous. 
 Theorem 
 \ref{thm:dynamics_stable} shows that 
  the map $ \sR^{n\t (n+k)}\ni \theta\mapsto  X^{x_0,\theta} \in \cH^2(\sR^{n})$ is continuous. 
Moreover, Theorems \ref{thm:lc_fb} 
and  \ref{thm:feedback_stable}
imply  that 
 there exists a constant $C\ge 0$, such that for all $\theta\in \sR^{n\t (n+k)}$ satisfying $|\theta-\theta^\star|\le 1$,
  $t\in [0,T]$ and $x,x'\in \sR^n$,
   we have  that 
   $|\psi^\theta(t,0)|\le C$,
$|\psi^\theta(t,x)-\psi^\theta(t,x')|\le C|x-x'|$
and 
   ${|\psi^\theta(t,x)-\psi^{\theta^\star}(t,x)|}\le C(1+|x|)|\theta-\theta^\star|$,
from which we can deduce that 
\begin{align*}
|\psi^\theta(t,x)-\psi^{\theta^\star}(t,x')|
&\le 
|\psi^\theta(t,x)-\psi^{\theta^\star}(t,x)|+|\psi^{\theta^\star}(t,x)-\psi^{\theta^\star}(t,x')|
\\
&\le 
 C(1+|x|)|\theta-\theta^\star|+ C|x-x'|.
\end{align*}
 Hence, %we have
 for all $\theta\in \sR^{n\t (n+k)}$ with $|\theta-\theta^\star|\le 1$,
%  that 
  \begin{align*}
\|\psi^\theta(\cdot,X^{x_0,\theta})-\psi^{\theta^\star}(\cdot,X^{x_0,\theta^\star})\|_{\cH^2}
&\le 
 C(1+\|X^{x_0,\theta}\|_{\cH^2})|\theta-\theta^\star|+ C\|X^{x_0,\theta}-X^{x_0,\theta^\star}\|_{\cH^2},
\end{align*}
  which along with the continuity of the map
$\sR^{n\t (n+k)}\ni\theta\mapsto  X^{x_0,\theta}\in \cH^2(\sR^n) $
implies that 
   the map
$\sR^{n\t (n+k)}\ni\theta\mapsto  \psi^\theta(\cdot,X^{x_0,\theta})\in \cH^2(\sR^k) $
is continuous.
% Since the entires of 
% $Z^{x_0,\theta} (Z^{x_0,\theta})^\trans$ involve only the products of
% $X^{x_0,\theta}$ and 
% $\psi^\theta(\cdot,X^{x_0,\theta})$, we can conclude
Since the entires of 
$\ol{U}^{x_0,\theta}$ involve only the expectations of products of
$X^{x_0,\theta}$ and 
$\psi^\theta(\cdot,X^{x_0,\theta})$,
 the desired continuity of 
% $ \sR^{n\t (n+k)}\ni \theta\mapsto  Z^{x_0,\theta} (Z^{x_0,\theta})^\trans\in \cH^2(\sR^{(n+k)\t (n+k)})$
the map
$ \sR^{n\t (n+k)}\ni \theta\mapsto  \ol{U}^{x_0,\theta}\in \sR$ 
follows.

Finally, by  the continuity of the minimum eigenvalue function, 
%we can  conclude that 
clearly
$\sR^{n\t (n+k)}\ni \theta\mapsto \lambda_{\min}(
\ol{U}^{x_0,\theta})\in \sR$
is continuous, which along with 
the fact that 
 $\lambda_{\min}(
\ol{U}^{x_0,\theta^\star})>0$
leads to the desired result.
 \end{proof}

 We then  quantify the estimation error of the least-squares estimator
\eqref{eq:theta_sum}
by assuming the concentration inequality \eqref{eq:concentration_beta}
holds for the compact set  $\cK_0$  in 
Lemma \ref{Ubar_nondeg}.

\begin{Proposition}\label{theta_conc}

Suppose (H.\ref{assum:lc_rl}\ref{item:fg})
and (H.\ref{assum:non_degenerate})
 hold.
Let $\cK_0$  be the  set in 
Lemma \ref{Ubar_nondeg}.
Assume further that
%for any given bounded set $\cK\subset \sR^{n\t (n+k)}$,
 there exist
 constants  $C_1,C_2>0$ and $\b\ge 1$
such that \eqref{eq:concentration_beta}
holds 
 for all $\eps\ge 0$, $m\in \sN$ and $\theta\in\mathcal{K}_0$.
% that there exist
% constants  $C_1,C_2>0$ and $\b\ge 1$,
%such that the following concentration inequality
%holds 
% for all $\eps\ge 0$, $m\in \sN$ and $\theta\in\mathcal{K}_0$,
%\begin{align}
%\l{eq:concentration_beta}
%\begin{split}
%&\max\Big\{
%\sP\big(\big|U^{x_0,\theta,m}-\ol{U}^{x_0,\theta}\big|\ge \eps \big),
%\sP\big(\big|V^{x_0,\theta,m}-\ol{V}^{x_0,\theta}\big|\ge \eps \big)
%\Big\}
%\\
%&\q \le 
%C_2\exp\Big(-C_1\min
%\Big\{
%\frac{m\eps^2}{C_2^2},
%\Big(
%\frac{m\eps}{C_2}
%\Big)^{\frac{1}{\beta}}
%\Big\}\Big),
%\end{split}
%\end{align}
%where   $\mathcal{K}_0$ is the set in 
%Lemma \ref{Ubar_nondeg}
%and
% $U^{x_0,\theta,m}, V^{x_0,\theta,m}$ are defined in \eqref{eq:UV}.
%
Then there exist constants  $\bar{C}_1,\bar{C}_2\ge 0$,
%independent of  $m$ and $\theta$, 
such that 
for all
$\theta\in \cK_0$ and
 $\delta\in (0,1/2)$, if 
$m\geq \bar{C}_1(-\ln\delta )^\beta$, 
 then we have with probability at least $1-2\delta$ that 
\bb\l{eq:parameter_estimation}
|\hat{\theta}-\theta^\star|\leq \bar{C}_2\bigg(\sqrt{\frac{-\ln\delta}{m}}+\frac{(-\ln\delta)^{\beta}}{m}+\frac{(-\ln\delta)^{2\beta}}{m^2}\bigg),
\ee
where 
$\hat{\theta}$ denotes the transpose of the left-hand side of \eqref{eq:theta_sum}
associated with $\theta$.
% (computed from $\theta$). 
\end{Proposition}

\begin{proof}
Throughout the proof, 
let $\delta\in(0,1/2)$ and $\theta\in \cK_0$ be fixed
and let   $\|\cdot\|_2$ be the matrix norm induced by Euclidean norms.
The invertibility of    $\ol{U}^{x_0,\theta}$  (see Lemma \ref{Ubar_nondeg}) 
implies that 
\eqref{eq:theta_star_exp} is well-defined, 
which along with \eqref{eq:theta_sum} 
leads to %the estimate  that 
\begin{equation}
\label{est}
\begin{split}
\|\hat{\theta}-\theta^\star\|_2& = \|(U^{x_0,\theta,m} + \tfrac{1}{m} \sI)^{-1}V^{x_0,\theta,m}-(\ol{U}^{x_0,\theta})^{-1}\ol{V}^{x_0,\theta}\|_2 \\
&\leq\|(U^{x_0,\theta,m} +  \tfrac{1}{m} \sI)^{-1}-(\ol{U}^{x_0,\theta})^{-1}\|_2 \|V^{x_0,\theta,m}\|_2 \\
&\quad+ \|(\ol{U}^{x_0,\theta})^{-1}\|_2 \|V^{x_0,\theta,m}-\ol{V}^{x_0,\theta}\|_2.
\end{split}
\end{equation}

We now estimate each term in the right-hand side of \eqref{est}.
By  Lemma \ref{Ubar_nondeg}, %we have  that
$\lambda_{\min}(\ol{U}^{x_0,\theta})\ge \tau_0$
for some $\tau_0>0$,
which implies that 
  $\|(\ol{U}^{x_0,\theta})^{-1}\|_2\leq 1/\tau_0$.
  Moreover, %
 by setting the right-hand side of \eqref{eq:concentration_beta} to be $\delta$, 
we can deduce 
 with probability at least $1-2\delta$ that 
 $|U^{x_0,\theta,m}-\ol{U}^{x_0,\theta}| \leq \delta_m$ 
 and 
 $|V^{x_0,\theta,m}-\ol{V}^{x_0,\theta}| \leq \delta_m$
 with the constant $\delta_m$ given by 
 \bb\l{eq:delta_m}
 \delta_m\coloneqq
 \max \left\{
\bigg(\frac{C_2^2}{C_1m}\ln\bigg(\frac{C_2}{\delta}\bigg)\bigg)^{\frac{1}{2}}, \frac{C_2}{m} \left(\dfrac{1}{C_1}\ln\bigg(\frac{C_2}{\delta}\bigg)\right)^{\beta}\right\},
 \ee
 where we have assumed without loss of generality that 
 $C_2\ge 1$.
 
 Let 
  $m$ be a sufficiently large constant satisfying 
 $\delta_m+1/m\le \tau_0/2$.
The fact that  $\|\cdot\|_2\le |\cdot|$ indicates
 with probability at least $1-2\delta$ that 
$ \|U^{x_0,\theta,m} +  \tfrac{1}{m} \sI-\ol{U}^{x_0,\theta}\|_2
 \le \frac{1}{m}+\delta_m\le \tfrac{\tau_0}{2}$,
 which in turn yields
 \begin{align*}
\lambda_{\min}( U^{x_0,\theta,m}+\tfrac{1}{m}\sI)
\ge 
\lambda_{\min}( \ol{U}^{x_0,\theta})
-\|U^{x_0,\theta,m} +  \tfrac{1}{m} \sI-\ol{U}^{x_0,\theta}\|_2
\ge \tfrac{\tau_0}{2},
 \end{align*}
 or equivalently $\|(U^{x_0,\theta,m}+\tfrac{1}{m}\sI)^{-1}\|_2\leq 2/\tau_0$.  
Then, since
$A^{-1}-(A+B)^{-1}=(A+B)^{-1}BA^{-1}$
for all nonsingular matrices $A$ and $A+B$,
we have with probability at least $1-2\delta$ that, 
\begin{equation*}
\begin{split}
\|(U^{x_0,\theta,m} +  \tfrac{1}{m} \sI)^{-1}-(\ol{U}^{x_0,\theta})^{-1}\|_2
&=
\|(\ol{U}^{x_0,\theta}+U^{x_0,\theta,m} +  \tfrac{1}{m} \sI-\ol{U}^{x_0,\theta})^{-1}-(\ol{U}^{x_0,\theta})^{-1}\|_2 
\\
& \le
 \|(U^{x_0,\theta,m}+ \tfrac{1}{m}\sI)^{-1}\|_2
\|(\ol{U}^{x_0,\theta})^{-1}\|_2
\|(U^{x_0,\theta,m} + \textstyle \frac{1}{m} \sI)-\ol{U}^{x_0,\theta}\|_2\\
& \leq \tfrac{2}{\tau_0^2}(\tfrac{1}{m}+\delta_m),
\end{split}
\end{equation*}
which along with
 the inequality that 
$
\|V^{x_0,\theta,m}\|_2 \le  \|\ol{V}^{x_0,\theta}\|_2 + |V^{x_0,\theta,m} - \ol{V}^{x_0,\theta}|
$
allows us to derive the following estimate from \eqref{est}:
\begin{align*}
\|\hat{\theta}-\theta^\star\|_2
\le 
\tfrac{2}{\tau_0^2}(\tfrac{1}{m}+\delta_m)
 ( \|\ol{V}^{x_0,\theta}\|_2 +\delta_m)
 +\tfrac{\delta_m}{\tau_0}.
\end{align*}
Note that $\|\ol{V}^{x_0,\theta}\|_2 $
is uniformly bounded for all $\theta\in \cK_0$
by the compactness of $\cK_0$
and the continuity of the map $\theta\mapsto 
\ol{V}^{x_0,\theta}$
(cf.~Lemma \ref{Ubar_nondeg}).
Thus, 
by  the condition that $\b\ge 1$ 
and the definition of $\delta_m$ in
\eqref{eq:delta_m},
we see that  there exists a constant $\bar{C}_2$,
depending only on
 $C_1,C_2$, $\beta$, $\tau_0$,
 and the constants in  (H.\ref{assum:lc_rl}\ref{item:fg}),
such that
the desired estimate \eqref{eq:parameter_estimation}
holds 
 with probability at least $1-2\delta$,
 provided that 
$m$ satisfies   $\delta_m+1/m\le \tau_0/2$.
Since $\beta\ge 1$ and $\delta\le 1/2$,  we see there exists 
$\bar{C}_1\ge 0$, independent of $m$, $\delta$ and $\theta$,
such that the  inequality  \eqref{eq:parameter_estimation}
holds for  all $m$ satisfying 
  $m\ge \bar{C}_1(-\ln \delta)^\b$.
%  which finishes the proof of the desired statement.
 %
\end{proof}

Now we are ready to present the proof of Theorem \ref{thm:regret}.

\begin{proof}[Proof of Theorem \ref{thm:regret}]
We start by proving Item \ref{item:regret_PAC}.
Then, by  
the assumptions  that 
$\lambda_{\min}(
\ol{U}^{x_0,\theta_0})>0$ and 
\eqref{eq:concentration_beta}
holds for $\cK=\ol{\cK}_0\coloneqq\cK_0\cup\{\theta_0\}$
with $\cK_0$ from Lemma \ref{Ubar_nondeg},
we can  extend  Proposition \ref{theta_conc} to show that  
\eqref{eq:parameter_estimation} holds for all $\theta\in \ol{\cK}_0$, 
 $\delta\in (0,1/2)$ and $m\geq \bar{C}_1(-\ln\delta )^\beta$,
with some constants  $\bar{C}_1,\bar{C}_2\ge 1$
depending on $\ol{\cK}_0$.
In the subsequent analysis,
we   fix   $\delta\in (0,1/4)$
and 
for all  $\ell\in \sN\cup\{0\}$,
we define   $\delta_\ell=2^{-\ell}\delta$, 
and 
let $\theta_{\ell+1}$ be generated by
 using  \eqref{eq:theta_sum} with   $m=m_\ell$ and $\theta=\theta_\ell$.
 We shall specify the precise choice of $m_0$ later.

In the sequel, 
we assume without loss of generality that $\eps_0/(3\bar{C}_2)\le 1$ and 
$\bar{C}_2/\eps_0\ge \bar{C}_1$,
where $\eps_0>0$ is the constant in the definition of $\cK_0$ (see Lemma \ref{Ubar_nondeg}).
We first show that 
there exists  $\hat{C}_0>0$, independent of $\delta$,
such that if $m_0\ge  \hat{C}_0(-\ln \delta)^\b$, then 
for all $\ell \in \sN\cup\{0\}$,
\bb\l{eq:m_0_hatC_0}
\bar{C}_2\bigg(\sqrt{\frac{-\ln\delta_\ell}{m_\ell}}+\frac{(-\ln\delta_\ell)^{\beta}}{m_\ell}+\frac{(-\ln\delta_\ell)^{2\beta}}{m^2_\ell}\bigg)
\le \eps_0.
\ee
%where $\eps_0>0$ is the constant in the definition of $\cK_0$ (see Lemma \ref{Ubar_nondeg}).
%Let us assume without loss of generality that $\eps_0/(3\bar{C}_2)\le 1$.
By the assumption that 
$\eps_0/(3\bar{C}_2)\le 1$,
 it suffices to show that  for all $\ell\in \sN\cup\{0\}$,
$
-\ln\delta_\ell/m_\ell\le 
( \eps_0/(3\bar{C}_2))^2$ 
and ${(-\ln\delta_\ell)^{\beta}}/{m_\ell}\le \eps_0/(3\bar{C}_2)$.
Given  $\beta\ge 1$ and $\delta_\ell< 1/4$,  it suffices to ensure 
$m_\ell\ge C(-\ln\delta_\ell)^{\beta}$ for all $\ell\in \sN\cup\{0\}$, where $C$ is a sufficiently large constant independent of $\delta$ and $\ell$. 
By  the definitions of $(\delta_\ell)_{\ell\in \sN}$
and $(m_\ell)_{\ell\in \sN}$ and the fact that $\delta< 1/4$,
 the desired condition 
 can be achieved by choosing $m_0\ge \hat{C}_0(-\ln \delta)^\b$, for a   sufficiently large constant $\hat{C}_0$ satisfying
\begin{align*}
\sup_{\ell\in  \sN\cup\{0\}, \delta\in(0,\frac{1}{4})}
\frac{ (-\ln (2^{-\ell}\delta))^{\beta}}{2^\ell(-\ln \delta)^\b}
=
\sup_{\ell\in  \sN\cup\{0\},\delta\in(0,\frac{1}{4})}
2^{-\ell} \Big(\frac{\ell\ln 2}{-\ln \delta} +1\Big)^{\beta}
\le 
\sup_{\ell\in  \sN\cup\{0\}}
2^{-\ell} \Big(\frac{\ell}{2} +1\Big)^{\beta}
\le \hat{C}_0<\infty.
\end{align*}

Now we choose  $m_0\ge \max( \hat{C}_0, \bar{C}_1)(-\ln \delta)^\b$,
%which in particular ensures  
%\eqref{eq:parameter_estimation}  for all $\theta\in\ol{\cK}_0$, 
% $\delta\in (0,1/2)$ and $m\ge m_0$.
 and show by induction that 
for all $k\in \sN\cup\{0\}$, it holds  with probability  at least $1-2\sum_{\ell=0}^{k-1}\delta_\ell$ that 
 $\theta_\ell\in \ol{\cK}_0$ for all $\ell=0,\ldots, k$ and
\bb\l{eq:theta_k}
|{\theta}_{k}-\theta^\star|_2\leq
\begin{cases}
|\theta_0-\theta^\star|_2, & k=0,\\
 \bar{C}_2\bigg(\sqrt{\frac{-\ln\delta_k}{m_k}}+\frac{(-\ln\delta_k)^{\beta}}{m_k}+\frac{(-\ln\delta_k)^{2\beta}}{m_k^2}\bigg),
 & k\in \sN.
 \end{cases}
\ee
The statement clearly holds for $k=0$.
Now suppose that the induction statement holds for some $k\in  \sN\cup\{0\}$.
%Given that
Conditioning on $\theta_k\in\ol{\cK}_0$, 
%which holds with probability  at least $1-2\sum_{\ell=0}^{k-1}\delta_\ell$
%by the induction hypothesis,
we can apply
\eqref{eq:parameter_estimation} with 
$\theta=\theta_k$,
$\delta=\delta_k<1/2$ and 
$m=m_{k}\ge  \bar{C}_1(-\ln\delta_k )^\beta$
(see \eqref{eq:m_0_hatC_0} and 
$\bar{C}_2/\eps_0\ge \bar{C}_1$),
and deduce 
  with probability   
 at least $1-2\delta_k$
 that
  \eqref{eq:theta_k} holds for the index $k+1$,
which along with  \eqref{eq:m_0_hatC_0}
shows that $\theta_{k+1}\in \cK_0\subset \ol{\cK}_0$.
 Since  the induction hypothesis implies that $\theta_k\in\ol{\cK}_0$ holds with probability  at least $1-2\sum_{\ell=0}^{k-1}\delta_\ell$, one can deduce that the induction statement also holds $k+1$.
 
 The above induction argument shows 
 that if $m_0=C(-\ln \delta)^\b$
 for any constant $C\ge C_0\coloneqq \max( \hat{C}_0, \bar{C}_1)$, then
%we have 
with   probability  at least $1-2\sum_{\ell=0}^{\infty}\delta_\ell=1-4\delta$,
 $\theta_k\in \ol{\cK}_0$ and 
\eqref{eq:theta_k} holds
for all $k\in \sN\cup\{0\}$. 
Now let us  assume such a setting,
and observe that 
the $i$-th trajectory is generated with control $\psi^{\theta_\ell}$
if $i\in (\sum_{j=0}^{\ell-1} m_j,\sum_{j=0}^{\ell}m_j]=( m_0(2^\ell-1), m_0(2^{\ell+1}-1)]$ for $\ell\in \sN\cup \{0\}$
(cf.~Algorithm \ref{alg:greedy}).
Then we can 
 apply Theorem \ref{thm:dynamics_stable} 
%(note that \eqref{eq:lc_per} agrees with 
%$J^{\theta^\star}(\psi_i;x_0)$ in the present setting)
and deduce for all $N\in \sN$
that 
\begin{align}
%\begin{split}
R(N)&
\le \sum_{\ell=0}^{\lceil \log_2(\frac{N}{m_0}+1)\rceil-1 } m_\ell \Big(J^{\theta^\star}(\psi^{\theta_\ell};x_0) - V(x_0;\theta^\star)\Big)
\leq
C' \sum_{\ell=0}^{\lceil \log_2(\frac{N}{m_0}+1)\rceil-1 } m_\ell |\theta_\ell-\theta^\star|
\nb\\
&\le C'm_0+C'\sum_{\ell=1}^{\lceil \log_2(\frac{N}{m_0}+1)\rceil-1 } \bigg( \sqrt{(-\ln\delta_\ell)m_\ell}+(-\ln\delta_\ell)^{\beta}\Big(1+\frac{(-\ln\delta_\ell)^{\beta}}{m_\ell}\Big)\bigg)
\nb\\
&\le
C'(-\ln \delta)^\b+C'\sum_{\ell=1}^{\lceil \log_2(\frac{N}{m_0}+1)\rceil-1 } \bigg( \sqrt{(-\ln\delta_\ell)m_\ell}+(-\ln\delta_\ell)^{\beta}\bigg),
\l{eq:R_N}
%\end{split}
\end{align}
where 
we have denoted by
$C'$ a generic constant independent of $\ell, N, \delta$,
and used 
 the fact that 
 ${(-\ln\delta_\ell)^{\beta}}/{m_\ell}\le C'$ for the last inequality (cf.~the choice of $\hat{C}_0$).
We then derive an upper bound  of \eqref{eq:R_N}.
By  virtue of the inequality that 
% $$\sqrt{(-\ln\delta_\ell)m_\ell}
% =\sqrt{(\ell \ln 2-\ln\delta)2^\ell m_0}
% \le C'(\sqrt{\ell 2^\ell} \sqrt{m_0}+\sqrt{(-\ln\delta) m_0}\sqrt{2}^\ell),
% \q \fa \ell\in \sN,
% $$
% we have  
% \begin{align*}
% \begin{split}
% &\sum_{\ell=1}^{\lceil \log_2(\frac{N}{m_0}+1)\rceil-1 }
% \sqrt{(-\ln\delta_\ell)m_\ell}
% \le 
% C'\Big(
% \sqrt{ \ln {N}}\sqrt{m_0}
% \sqrt{2}^{ \log_2(\frac{N}{m_0}+1)}
% +\sqrt{(-\ln\delta) m_0}\sqrt{2}^{ \log_2(\frac{N}{m_0}+1)}
% \Big)
% \\
% &\q \q =
% C'\big(
% \sqrt{ \ln {N}}
% +\sqrt{-\ln\delta }\big)
% \sqrt{N+(-\ln \delta)^\b}
% \le 
% C'\big(
% \sqrt{ \ln {N}}
% +\sqrt{-\ln\delta }\big)
% \big(
% \sqrt{N}+(-\ln \delta)^{\b/2}\big).
% \end{split}
% \end{align*}
$\sqrt{(-\ln\delta_\ell)m_\ell}
=\sqrt{(\ell \ln 2-\ln\delta)2^\ell m_0}
\le C'\sqrt{(\ell-\ln\delta) m_0}\sqrt{2}^\ell$
for all $\ell\in \sN$,
we have  
\begin{align*}
\begin{split}
&\sum_{\ell=1}^{\lceil \log_2(\frac{N}{m_0}+1)\rceil-1 }
\sqrt{(-\ln\delta_\ell)m_\ell}
\le 
C'
\sqrt{(\ln {N}-\ln\delta) m_0}
\sqrt{2}^{ \log_2(\frac{N}{m_0}+1)}
\le
C'
\sqrt{(\ln {N}-\ln\delta) 
(N+(-\ln\delta)^\b)
}.
\end{split}
\end{align*}
Moreover, 
by  
$\ln\delta_\ell=-\ell\ln2+\ln \delta$ and 
H\"{o}lder's inequality, %we have 
 \begin{align*}
\begin{split}
\sum_{\ell=1}^{\lceil \log_2(\frac{N}{m_0}+1)\rceil-1 }
(-\ln\delta_\ell)^{\beta}
\le 
\sum_{\ell=1}^{C' \ln {N}}
C'((\ell\ln2)^\b+(-\ln \delta)^\b)
\le
C'\big(
(\ln {N})^{\b+1}
+\ln {N}(-\ln\delta)^{\beta}\big).
\end{split}
\end{align*}
Consequently,  from \eqref{eq:R_N}, $\b\ge 1$
and 
the inequality
$\sqrt{x+y}\le \sqrt{x}+\sqrt{y}$
for all $x,y\ge 0$,
it is clear 
for all $N\in \sN$, 
$R(N)\le C'\big(
\sqrt{N}\sqrt{ \ln {N}}
+\sqrt{-\ln\delta }
\sqrt{N}
+(-\ln\delta)^{\beta}\ln {N}\big)$
for some constant $C'$ independent of $\b$ and $N$,
which finishes the proof of Item \ref{item:regret_PAC}. 

We are ready to show Item \ref{item:regret_as}.
For each $N\in \sN\cap[3,\infty)$, we define $\delta_N=1/N^2$ and  the event
$A_N=\{R(N)> C'(
\sqrt{N}\sqrt{ \ln {N}}
+\sqrt{-\ln\delta_N }
\sqrt{N}
+(-\ln\delta_N)^{\beta}\ln {N})\}$.   Item  \ref{item:regret_PAC} shows that 
$\sum_{N=3}^\infty\sP(A_N)\le 4\sum_{N=3}^\infty \delta_N<\infty$.
Hence,  from the
 Borel-Cantelli lemma, 
$\sP(\lim\sup_{N\to \infty} A_N)=0$,
which along with the definition of $\delta_N$
 implies the desired conclusion.
\end{proof}

{%\color{blue}
\section{Extension: RL problems with controlled diffusion}
\label{sec:ctrl_diffusion}

In this section, we extend our framework to analyze the {%\color{blue}
regret order}
%non-asymptotic  performance 
of learning algorithms for   general continuous-time RL problems,
whose state dynamics involves 
controlled diffusion. 
To simplify the presentation, 
we
focus on    entropy-regularized problems 
studied in
\cite{wang2019exploration,vsivska2020gradient,reisinger2021regularity,tang2021exploratory}
and 
outline the essential steps of the argument.

For each $\theta=(A,B)\in \sR^{n\t (n+k)}$, define $V(\cdot;\theta):[0,T]\t \sR^n\to \sR$  by 
\bb\label{eq:lc_entropy}
V(t,x;\theta)\coloneqq \inf_{\a\in \cH^2(\sR^k)} 
\sE\left[\int_t^T 
 f(s,X^{t,x,\a}_s,\a_s)
\, \d s+g(X_T^{t,x,\a})\right],
\q \fa (t,x)\in [0,T]\t \sR^n, 
\ee
  where for each $\a\in \cH^2(\sR^k)$,
  $X^{t,x,\a}\in \cS^2(\sR^n)$ satisfies 
 the  controlled dynamics: 
\bb\l{eq:lc_sde_entropy}
\d X_s =(AX_s+B\a_s)\,\d s+\sigma(s,X_s,\a_s)\, \d W_s, \q s\in [t,T],
\q X_t=x.
\ee 
The functions 
$f:[0,T]\t \sR^n\t \sR^k\to \sR\cup\{\infty\}$ and 
  $\sigma:[0,T]\t \sR^n\t \sR^k\to \sR^{n\t d} $ are such that 
  for all $(t,x)\in [0,T]\t \t \sR^n$
  and $a=(a_i)_{i=1}^k\in \sR^k$,  
 \begin{align}\l{eq:lc_entropy_coefficient}
  f(t,x,a)= \sum_{i=1}^k \ol{f}_i(t,x)a_i +\cR_{\textrm{en}}(a), 
  \q 
  \sigma(t,x,a)\sigma(t,x,a)^\trans= 
   \sum_{i=1}^k
   \ol{ \sigma}_i (t,x)\ol{\sigma}_i (t,x)^\trans a_i,
\end{align}
where for each $i=1,\ldots k$,  $\ol{f}_i:[0,T]\t \sR^n\to \sR$, $\ol{\sigma}_i:[0,T]\t \sR^n\to \sR^{n\t d}$ are  some given   functions and 
$\cR_{\textrm{en}}:\sR^k\to \sR\cup\{\infty\}$
is  Shannon's entropy 
 function (cf.~Example \ref{example:relax}) such that 
\bb\label{eq:entropy}
\cR_{\textrm{en}}(a)=
\begin{cases}
\sum_{i=1}^k a_i\ln (a_i), & a\in \Delta_k\coloneqq\{
a\in [0,1]^k\mid 
\sum_{i=1}^k a_i=1\},\\
\infty, & a\in \sR^k\setminus \Delta_k.
\end{cases}
\ee
To avoid needless technicalities, we assume $(\ol{f}_i)_{i=1}^k$, $(\ol{\sigma}_i)_{i=1}^k$ and $g$ to be  bounded and  sufficiently regular as in Proposition \ref{prop:stability_entropy}.
%This implies the value function $V(\cdot; \theta)$ is well-defined for each $\theta$. 

Note that  
\eqref{eq:entropy} restricts  control processes to those taking values in $\Delta_k$.
Hence, if $\ol{\sigma}_\ell(t,x)\equiv \ol{\sigma}$ for some $\ol{\sigma}\in \sR^{n\t d}$, then
\eqref{eq:lc_entropy}-\eqref{eq:lc_sde_entropy}
is a special case of the linear-convex model 
studied in Sections \ref{sec:lipschitz_stable}-\ref{sec:lc_learning}.
 Consequently, 
Theorem \ref{thm:regret} can be applied to study 
%non-asymptotic 
the {%\color{blue} 
regret order} of GLS algorithms for 
 \eqref{eq:lc_entropy}-\eqref{eq:lc_sde_entropy}
with given initial time and state $(t,x)\in [0,T]\t \sR^n$ but with unknown 
parameter $\theta$.

To analyze %non-asymptotic performance 
the {%\color{blue} 
regret order} of learning algorithms  with general 
$\sigma$, 
a crucial step is to extend  Theorem \ref{thm:dynamics_stable} and 
quantify the  performance     of  a greedy policy 
from an incorrect model.
The fact that control affects the diffusion coefficients 
 complicates the stability analysis of   optimal feedback controls 
(i.e., Theorem \ref{thm:feedback_stable}) for  \eqref{eq:lc_entropy}-\eqref{eq:lc_sde_entropy}. 
The following proposition 
  proves a linear performance gap under the 
 condition that the value function $V(t,x;\theta)$ in \eqref{eq:lc_entropy} is sufficiently regular in $t, x$ and $\theta$.
Recall 
that the
  first and second-order derivatives
of a sufficiently regular value function 
   can be represented by
solutions to the associated FBSDE \eqref{eq:lc_fbsde}
(see e.g., \cite[Theorem 4.1, p.~250]{yong1999stochastic}). Hence  
the linear performance gap can  also be established by assuming sufficient regularity of  the solution process $(Y^{t,x},Z^{t,x})$ to \eqref{eq:lc_fbsde}, whose details are omitted here.

\begin{Proposition}
\label{prop:stability_entropy}
For each $\theta\in \sR^{n\t (n+k)}$, let 
$V(\cdot;\theta):[0,T]\t \sR^n\to \sR$ be defined by 
\eqref{eq:lc_entropy}.
Suppose that    $(\ol{f}_i)_{i=1}^k, (\ol{\sigma}_i)_{i=1}^k \subset  C^{0,1}([0,T]\t \sR^n)$, $g\in C^1(\sR^n)$,
and
 there exists  $\mathfrak{m}:[0,\infty)\to [0,\infty)$ such that for all
$(t,x)\in [0,T]\t \sR^n$ and 
 $\theta,\theta'\in \sR^{n\t (n+k)}$, 
 $\f{\p}{\p t}V(\cdot; \theta)$ is continuous,
  $\|V(\cdot; \theta)\|_{C^{0,3}([0,T]\t \sR^n)}\le \mathfrak{m}(|\theta|)$, 
\bb
 |\nabla_x V(t,x;\theta)-\nabla_x V(t,x;\theta')|
 +|\operatorname{Hess}_x V(t,x;\theta)-\operatorname{Hess}_x V(t,x;\theta')|
 \le (\mathfrak{m}(|\theta|)+\mathfrak{m}(|\theta|'))|\theta-\theta'|(1+|x|).
\ee
Then for all $\theta\in \sR^{n\t (n+k)}$, there exists $\psi^\theta\in \cV$ such that 
\begin{enumerate}[(1)]
\item \label{item:feedback_entropy}
$\psi^\theta$ is an optimal feedback control of 
\eqref{eq:lc_entropy}-\eqref{eq:lc_sde_entropy} satisfying 
for all $x_0\in \sR^n$ and  $\theta\in \sR^{n\t (n+k)}$, $V (0,x_0;\theta)=J(\psi^\theta;x_0,\theta)$,
 where for each 
 $\psi\in \cV$,
$$
J(\psi;x_0,\theta)\coloneqq\sE\left[\int_0^T f(t,X^{x_0,\theta, \psi}_t,\psi(t,X^{x_0,\theta, \psi}_t))\, \d t+g(X^{x_0,\theta, \psi}_T)\right],
$$
and  
$X^{x_0,\theta, \psi}\in \cS^2(\sR^n)$ satisfies  the following dynamics: 
$$
  \d X_t =
  \theta 
\begin{psmallmatrix}
X_t
\\
\psi(t,X_t)
\end{psmallmatrix}
\,\d t+
\sigma(t,X_t,\psi(t,X_t))
 \d W_t,
\q t\in[0,T], 
\q 
X_0=x_0,
$$
\item \label{item:performance_gap_entropy}
for all $x_0\in \sR^n$ and $R\ge 0$  there exists a constant $C$ such that 
for all $\theta,\theta'\in \sR^{n\t (n+k)}$ with $|\theta|,|\theta'|\le R$,
$$|J(\psi^{\theta'};x_0,\theta)-J(\psi^{\theta};x_0,\theta)|\le C|\theta'-\theta|.$$
\end{enumerate}
\end{Proposition}

Proposition \ref{prop:stability_entropy} relies on the regularity and Lipschitz stability of the value function $V$.  For instance, 
if all coefficients are bounded and sufficiently smooth, and $\sigma$ satisfies 
the uniform parabolicity condition, 
then 
 $C^{2+\a}$ regularity results for fully nonlinear parabolic PDEs (see e.g., the  Evan-Kryolv theorem in \cite[Theorems 6.4.3 and 6.4.4, p.~301]{krylov1987nonlinear}) and 
a bootstrap argument would ensure that for any given $\theta$, the  function $V(\cdot, \theta)$  is continuously differentiable in $t$ and three-time continuously differentiable in $x$. 
Due to  the unbounded drift coefficient of \eqref{eq:lc_sde_entropy},
the  boundedness in the 
$C^{0,3}([0,T]\t \sR^n)$-norm and the locally Lipschitz continuity of $V$ in $\theta$ follow from  an extension of the Schauder estimate  (see e.g., \cite{krylov2009elliptic}) to nonlinear parabolic equations with unbounded coefficients in the
whole space.

With Proposition \ref{prop:stability_entropy},
we can then quantify  %non-asymptotic 
the regrets of  GLS algorithms  (see Algorithm \ref{alg:greedy})
for  \eqref{eq:lc_entropy}-\eqref{eq:lc_sde_entropy}
with unknown drift parameter $\theta$ and known diffusion coefficient $\sigma$. 
By the boundedness of $\sigma$ and the regularity of $\psi^\theta$, % and  \cite{ma2010transportation},
one can prove 
Proposition \ref{prop:concentration_diffusion} in the present setting. 
Hence,
Theorem \ref{thm:regret} (with $\beta=1$ in \eqref{eq:concentration_beta}) shows that Algorithm \ref{alg:greedy} enjoys a sublinear regret as shown in Theorem \ref{cor:regret_diffusion}.

\begin{proof}[Proof of Proposition \ref{prop:stability_entropy}]
For any given $\theta=(A,B)\in \sR^{n\t (n+k)}$,
the regularity of $V(\cdot;\theta)$ and 
 \cite[Proposition 3.5, p.~182]{yong1999stochastic} imply that $V(\cdot;\theta)$ is the unique classical solution to 
 the associated HJB equation. That is,  for all $(t,x)\in [0,T)\t \sR^d$, 
$$
\tfrac{\p }{\p t}V(t,x)+
\inf_{a\in \Delta_k}
\Big(
\tfrac{1}{2}\tr(\sigma(t,x,a)\sigma(t,x,a)^\trans \operatorname{Hess}_x V(t,x))+
\la Ax+Ba, \nabla_x V(t,x) \ra+f(t,x,a)\Big)=0,
$$
and $V(T,x)=g(x)$ for all $x\in \sR^d$.
By \eqref{eq:lc_entropy_coefficient}, for all $(t,x)\in [0,T]\t \sR^n$, 
\begin{align*}
\psi^\theta(t,x)
&\coloneqq \arg\min_{a\in \Delta_k}
\big(
\tfrac{1}{2}\tr(\sigma(t,x,a)\sigma(t,x,a)^\trans \operatorname{Hess}_x V(t,x;\theta))+
\la Ba, \nabla_x V(t,x;\theta) \ra+f(t,x,a)\big)
\\
&=
\nabla\cR_{\textrm{en}}^*\big(
-\tfrac{1}{2}\tr(\ol{\sigma}(t,x)\ol{\sigma}(t,x)^\trans \operatorname{Hess}_x V(t,x;\theta))
-B^\trans \nabla_x V(t,x;\theta)-\ol{f}(t,x) \big), 
\end{align*}
where for all $z\in \sR^k$, 
$\cR_{\textrm{en}}^*(z)=\sup_{a\in \Delta_k}(\la a,z\ra -\cR_{\textrm{en}}(a))= \ln\sum_{i=1}^k\exp(z_i)$,
and 
$$
\tr(\ol{\sigma}(t,x)\ol{\sigma}(t,x)^\trans \operatorname{Hess}_x V(t,x;\theta))
=\begin{pmatrix}
\tr(\ol{\sigma}_1(t,x)\ol{\sigma}_1(t,x)^\trans \operatorname{Hess}_x V(t,x;\theta))
\\
\vdots
\\
\tr(\ol{\sigma}_k(t,x)\ol{\sigma}_k(t,x)^\trans \operatorname{Hess}_x V(t,x;\theta))
\end{pmatrix},
\q
 \ol{f}(t,x)=\begin{pmatrix}
 \ol{f}_1(t,x)
 \\
 \vdots
\\
 \ol{f}_k(t,x)
\end{pmatrix}.
$$ 
The Lipschitz continuity of $\nabla \cR_{\textrm{en}}^*$ and the regularity assumptions imply that 
$\psi^\theta\in \cV$ and the corresponding state process $X^{x_0,\theta,\psi^\theta}$ is well defined.
Then a standard verification argument 
(see e.g.,
\cite[Theorem 6.6, p.~278]{yong1999stochastic}) shows $\psi^\theta$ is an optimal feedback control and finishes the proof of 
Item \ref{item:feedback_entropy}.

To prove Item \ref{item:performance_gap_entropy},
Fix $x_0\in \sR^n$ and $R\ge 0$  and let $C$ be a generic constant independent of $\theta$.
Note that 
the Fenchel-Young identity gives that 
$ \cR_{\textrm{en}}^*(\nabla \cR_{\textrm{en}}^*(z))=\la z, \nabla \cR_{\textrm{en}}^*(z)\ra -\cR_{\textrm{en}}^*(z)$ for all $z\in \sR^k$,
which along with \eqref{eq:lc_entropy_coefficient} implies that  
for all $(t,x,\theta)\in [0,T]\t \sR^n\t \sR^{n\t (n+k)}$,
\begin{align*}
f(t,x,\psi^\theta(t,x))
&=-\la \tfrac{1}{2}\tr(\ol{\sigma}(t,x)\ol{\sigma}(t,x)^\trans \operatorname{Hess}_x V(t,x;\theta))
+B^\trans \nabla_x V(t,x;\theta),\psi^\theta(t,x)\ra 
\\
&\q
-\cR_{\textrm{en}}^*(-\tfrac{1}{2}\tr(\ol{\sigma}(t,x)\ol{\sigma}(t,x)^\trans \operatorname{Hess}_x V(t,x;\theta))
-B^\trans \nabla_x V(t,x;\theta)-\ol{f}(t,x)).
\end{align*}
By the regularity assumptions of the coefficients and the function $V$, 
for all $t\in [0,T]$, $x,x'\in \sR^n$ and 
$\theta,\theta'\in \sR^{n\t (n+k)}$ with $|\theta|,|\theta'|\le R$, there exists $C\ge 0$ such that 
\begin{align*}
|\psi^{\theta'}(t,x')-\psi^{\theta}(t,x)|+
|f(t,x',\psi^{\theta'}(t,x'))-f(t,x,\psi^\theta(t,x))|\le 
C\big(|x-x'|+(1+|x'|+|x|)|\theta-\theta'|\big).
\end{align*}
 Proceeding along the lines of 
the proof of 
Theorem \ref{thm:dynamics_stable}
leads to the desired estimate in Item \ref{item:performance_gap_entropy}.
\end{proof}
 }

{%\color{blue} 

\section{Numerical experiments}
\label{sec:numerical}
%This section will test the theoretical findings and  Algorithm \ref{alg:greedy} through numerical experiments.

In this section, we test the theoretical findings and  Algorithm \ref{alg:greedy} through numerical experiment on
 a three-dimensional  LQ RL problem considered in \cite{dean2018regret,dean2020sample}.
Our experiments show  the convergence of the least-squares estimations  to  the true parameters as the number of episodes increases,  as well as the sublinear cumulative regret as indicated  in Theorem \ref{cor:regret_diffusion}.
Moreover, it confirms that
the state coefficient $A^\star$ is easier to learn than the control coefficient $B^\star$, consistent with the observations in
 \cite{dean2020sample}. 
Our numerical result shows that a rough estimation of the control parameter $B^\star$ is 
often 
sufficient to design a nearly optimal  feedback control, and that the Algorithm \ref{alg:greedy} is robust with respect to the initial batch size $m_0$.

\paragraph{Problem setup.}
We consider a three-dimensional LQ RL problems over the time horizon $[0,T]$ with $T=1.5$,
where the linear state dynamics \eqref{eq:lc_sde_theta_star}
has the initial state $x_0$ and  unknown coefficients $\theta^\star=(A^\star, B^\star)\in \sR^{3\t (3+3)}$ chosen as in 
\cite{dean2018regret,dean2020sample}:
\[ A^\star=\left[ \begin{array}{ccc}
1.01 & 0.01 & 0\\
0.01 & 1.01 & 0.01\\
0 & 0.01 & 1.01
\end{array} \right], \;
B^\star = \sI_3, \;\sigma=\sI_3,\; \gamma\equiv 0,\; x_0=0,
\]
with $\sI_3$ being the $3\times3$ identity matrix,
and  the cost functional \eqref{eq:lc_theta_star}
involves   quadratic functions 
$g\equiv 0$ and
 $f(t,x,a)=(x^\trans  Q x+a^\trans R a)/2$,
with
$Q=0.1 \sI_3 $ and $R= \sI_3$.
As mentioned in \cite{dean2018regret,dean2020sample},
this state  dynamics corresponds to a marginally unstable graph Laplacian system where adjacent
nodes are weakly connected, 
which arises naturally from consensus and distributed averaging problems.
Since 
the cost  
penalizes 
 the control inputs more than the states,
 it is essential to learn the unstable components of $A^\star$ 
 and perform control on these components
 in order to  achieve an optimal cost. 
Note that this LQ RL problem 
satisfies (H.\ref{assum:lc_rl}); 
see the last paragraph of Remark \ref{rmk:rl_assum}.

The numerical experiments are coded using   Python.
 Algorithm \ref{alg:greedy} is initialized with   $m_0=4$
and the initial guess 
$A_0=\left[\begin{smallmatrix} 
1.6243 & -0.6118 & -0.5282\\
-1.0730 & 0.8654 & -2.3015\\
1.7448 & -0.7612 &  0.3190
\end{smallmatrix}\right]
$
and 
$B_0=\left[\begin{smallmatrix} 
-0.2494 & 1.4621& -2.0601\\
-0.3224 & -0.3841& 1.1338\\
-1.0999 & -0.1724 & -0.8779
\end{smallmatrix}\right]
$,
whose entries are  sampled independently from the standard  normal distribution.
%Our experiments with  different choices  of $m_0$ and $( A_0, B_0)$
%show that 
%the growth rate of the algorithm regret is very robust with respect to these hyperparameters,
%whose details are  omitted  here for brevity.
For each $\ell\in \sN\cup\{0\}$,
given the current estimate $\theta_\ell=(A_\ell,B_\ell)$ 
of $\theta^\star$,     
classical LQ control theory 
(see e.g.,~\cite{yong1999stochastic})
shows  that
solutions to 
 \eqref{eq:lc_fbsde}
can be found analytically via   Riccati equations,
and 
the greedy policy $\psi^{\theta_\ell}$
is given by $\psi^{\theta_\ell}(t,x)=   -R^{-1}B^\trans P^{\theta_\ell}_t x$, 
where $P^{\theta_\ell}$ is the unique positive semidefinite solution to  
\begin{equation}
\label{riccati1}
%\begin{cases}
\tfrac{\d}{\d t}P_t + A_\ell^\trans P_t + P_tA_\ell  - P_t(B_\ell R^{-1}B_\ell^\trans)P_t + Q=0, \; t \in (0,T);
\quad 
P_T=0.
%\end{cases}
\end{equation}
We  solve  \eqref{riccati1} numerically via a high-order Runge-Kutta method
on a uniform time grid  with stepsize $T/100$, 
and then 
 simulate  $m_\ell=2^\ell m_0$ independent trajectories of the  state dynamics 
\eqref{eq:lc_sde_theta}
(controlled by $\psi^{\theta_\ell}$)
using the Euler-Maruyama method on the same time grid. 
To estimate  statistical properties 
of the algorithm regret \eqref{eq:regret},
we 
 execute  Algorithm \ref{alg:greedy}  for  100 independent runs,
 where 
among different executions,
the observed state trajectories are simulated based on independent Brownian motion increments. 

\paragraph{Performance with $m_0=4$.} 
 Figure \ref{fig:3d}
 exhibits 
the  performance of Algorithm \ref{alg:greedy} 
 for this  LQ-RL problem,
 where the solid lines and the shallow areas 
 indicate the sample mean and the 95\% confidence interval over 100 repeated experiments. The numerical results indicate  that  algorithm \ref{alg:greedy} manages to  learn the parameters over time while incurring a desirable sublinear regret, which is consistent with our theoretical result in Theorem \ref{cor:regret_diffusion}. More precisely,
 \begin{itemize}
     \item  Figure \ref{fig:est_err_3d} presents
 the 
  logarithmic relative error
of the estimate $(A_\ell,B_\ell)$ 
  (in the Frobenius norm)
after the $\ell$-th  update
for $\ell\in \{0,\ldots, 10\}$. 
One can observe that  the  estimate $(A_\ell,B_\ell)_{\ell}$ converge to the true parameter 
 $(A^\star,B^\star)$ as the number of episodes increases.
 Our experiment shows that 
it is much easier to learn the state coefficient $A^\star$ than the control coefficient $B^\star$,
which is consistent with the observation in 
 \cite{dean2020sample}
for other adaptive control schemes. 
\item  Figure \ref{fig:subopt_3d} presents the relative error between
  the  expected cost
$ J^{\theta^\star}(\psi^{\theta_\ell};x_0) $ 
and the optimal  expected cost
$ J^{\theta^\star}(\psi^{\theta^\star};x_0) $. 
One can see that  
a rough estimate of the control parameter $B^\star$ is 
often 
sufficient to design a nearly optimal  feedback control.
In particular, after the 10-th update ($\ell=10$), although 
the relative approximation errors of $A_\ell$ and $B_\ell$ are 2.7\% and 24.9\%, respectively,
the cost of  $\psi^{\theta_\ell}$
 approximates the optimal cost accurately with a  relative error 0.6\%. 
 \item  Figure \ref{fig:regret_3d} presents the cumulative regret over episodes. One can see that the small performance gap results in a slowly growing algorithm regret. In fact, 
performing a linear regression for logarithms
of expected regret and episode
 shows that the  regret after the $N$-th episode is of the magnitude 
$\cO(N^{0.34})$,
which is slightly better than the theoretical upper bound in Theorem \ref{cor:regret_diffusion}. 
 \end{itemize}

% These results clearly demonstrate the robust performance of Algorithm \ref{alg:greedy}, even with a small initial episode $m_0$  and an initial  guess 
%   $(A_0,B_0)$ that are  far from the true parameters. 

\begin{figure}[H]
\centering
\begin{subfigure}{0.32\textwidth}
    \includegraphics[trim=18 5 30 25, clip,  width=\textwidth]{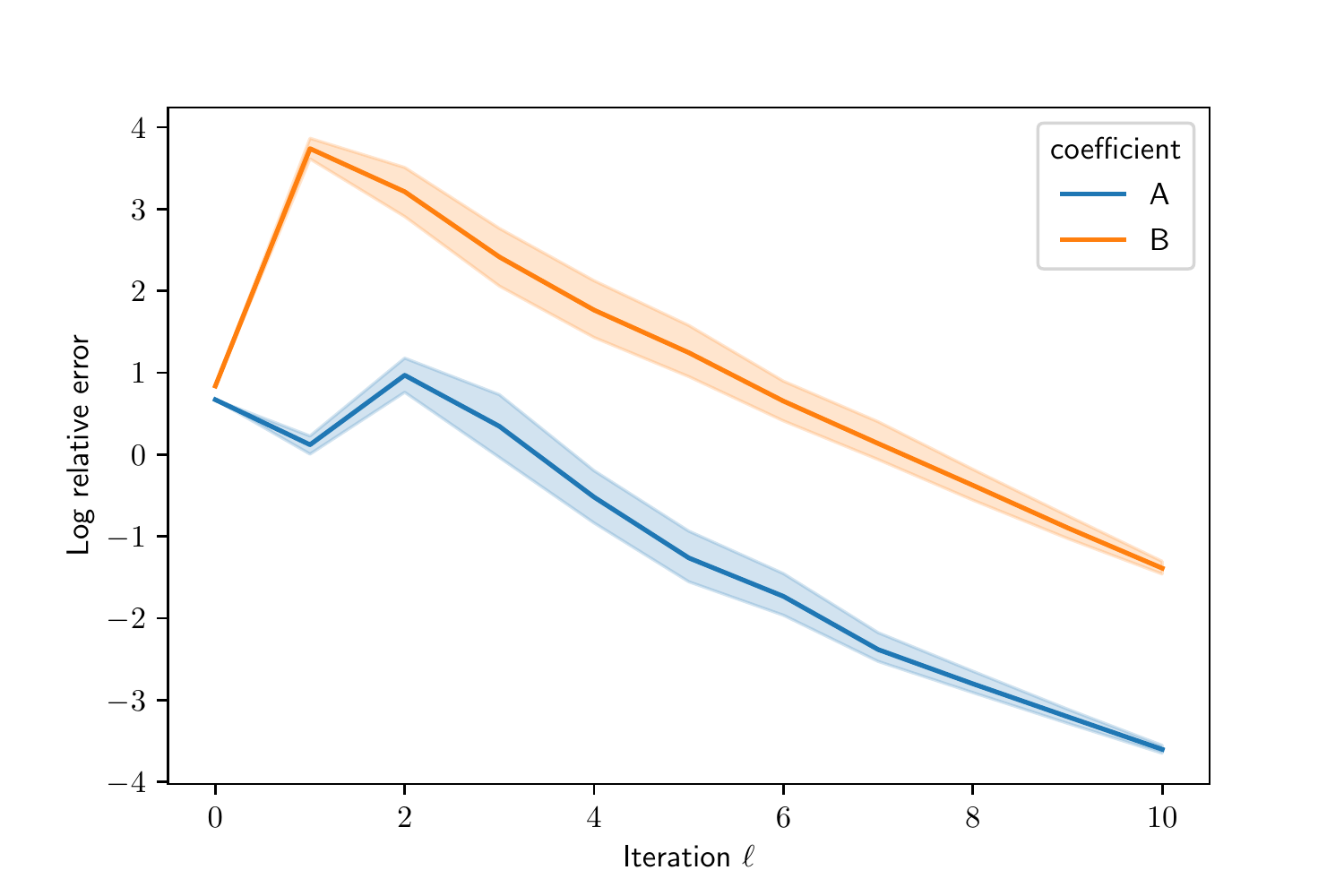}
    \caption{Parameter estimation errors.}
    \label{fig:est_err_3d}
\end{subfigure}
%\hfill
\begin{subfigure}{0.32\textwidth}
    \includegraphics[trim=18 5 30 25, clip, width=\textwidth]{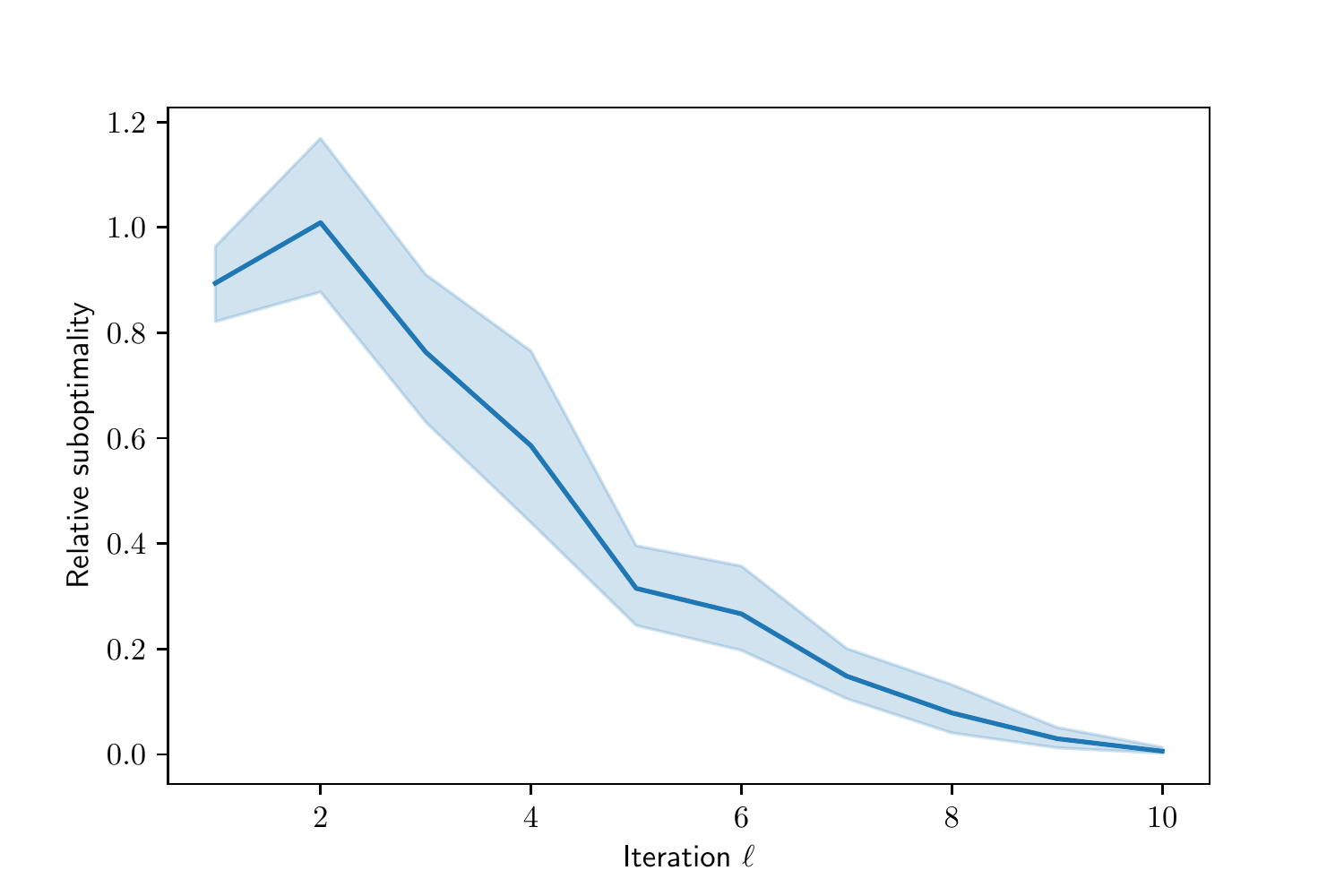}
    \caption{Cost suboptimality gap.}
    \label{fig:subopt_3d}
\end{subfigure}
%\hfill
\begin{subfigure}{0.32\textwidth}
    \includegraphics[trim=18 5 30 25, clip,  width=\textwidth]{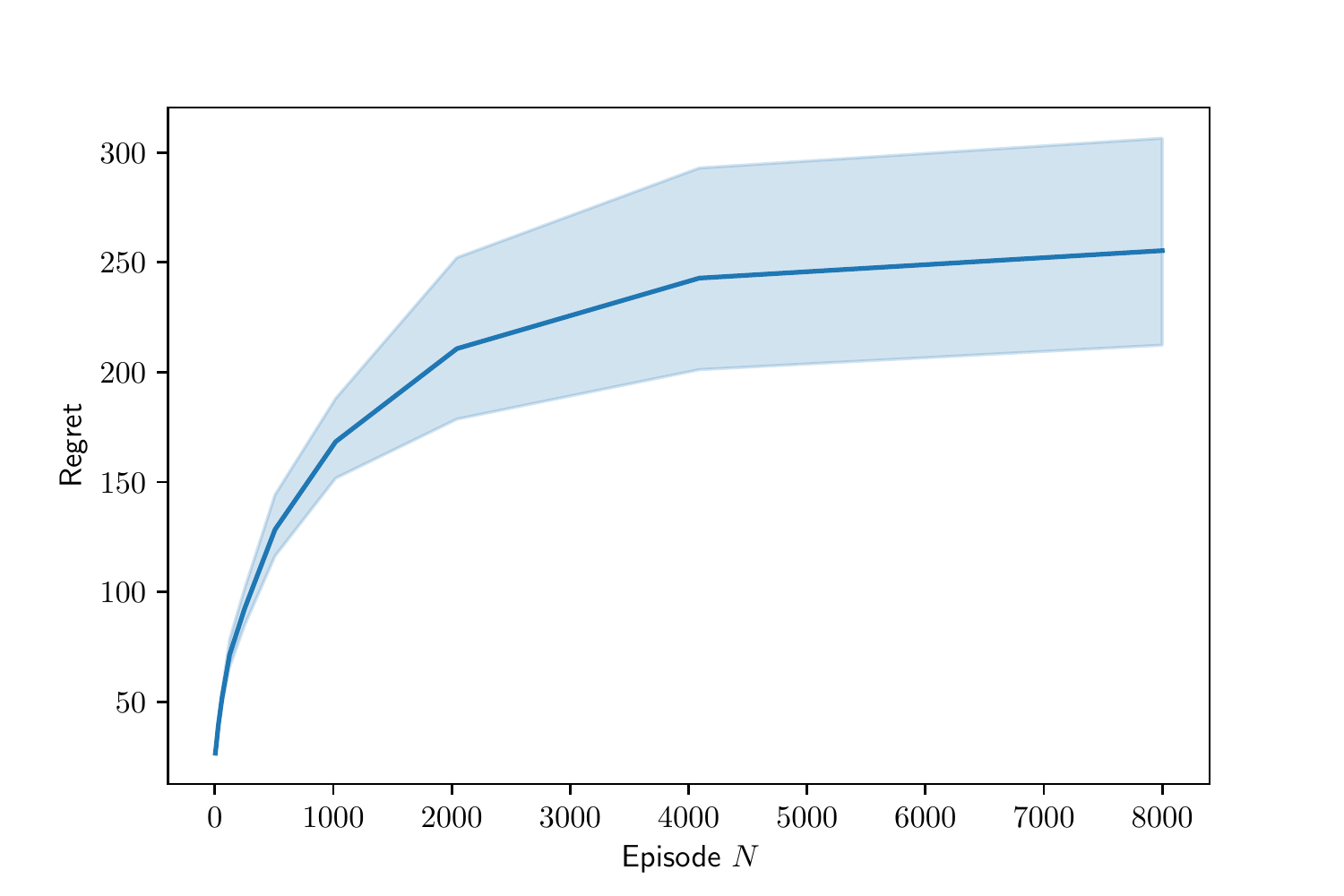}
    \caption{Algorithm regret.}
    \label{fig:regret_3d}
\end{subfigure}
        
\caption{Performance of Algorithm \ref{alg:greedy} for the LQ-RL problem ($m_0=4$).}
\label{fig:3d}
\end{figure}

% \begin{figure}[H]
% \centering
% \begin{subfigure}{0.32\textwidth}
%     \includegraphics[trim=18 5 30 25, clip,  width=\textwidth]{figures/AB_3d_N_8000_m0_2_new.pdf}
%     \caption{Parameter estimation errors.}
%     \label{fig:est_err_3d_m0_2}
% \end{subfigure}
% %\hfill
% \begin{subfigure}{0.32\textwidth}
%     \includegraphics[trim=18 5 30 25, clip, width=\textwidth]{figures/subopt_3d_N_8000_m0_2_new}
%     \caption{Cost suboptimality gap.}
%     \label{fig:subopt_3d_m0_2}
% \end{subfigure}
% %\hfill
% \begin{subfigure}{0.32\textwidth}
%     \includegraphics[trim=18 5 30 25, clip,  width=\textwidth]{figures/regret_3d_N_8000_m0_2_new.pdf}
%     \caption{Algorithm regret.}
%     \label{fig:regret_3d_m0_2}
% \end{subfigure}
        
% \caption{Performance of Algorithm \ref{alg:greedy} for the LQ-RL problem ($m_0=2$).}
% \label{fig:3d_m0_2}
% \end{figure}
\paragraph{Robustness with respect to the initial batch size $m_0$.}
We next demonstrate the robustness of  Algorithm \ref{alg:greedy} by performing computations with  $m_0=1$  and 
fixing  other settings as above. 
 The results are shown in Figure \ref{fig:3d_m0_1}.
 Note that the smaller  initial batch size $m_0$  %leads to more frequent parameter updates, and 
 makes the learning more challenging. % as suggested by our theoretical results \ref{??}. 
 By comparing  the results
 against those with $m_0=4$, one can see that our algorithm is robust and performs well with the small $m_0$. In particular, we see that 
 \begin{itemize}
     \item 
Estimating parameters with fewer sample 
trajectories
leads to larger  parameter estimation errors  with suboptimality gaps,
especially for the first few iterations.
It also leads to a wider range of $(A_\ell, B_\ell)_{\ell}$ 
 among different algorithm executions
and hence 
 a larger variance of the algorithm regret.
 \item  As the number of episodes increases, the estimate $(A_\ell,B_\ell)_{\ell}$ converge to the true parameter $(A^\star,B^\star)$ and the suboptimality gap quickly converges to $0$, see Figures \ref{fig:est_err_3d_m0_1} and \ref{fig:subopt_3d_m0_1}.
 The algorithm regret  grows  sublinearly (see Figure \ref{fig:regret_3d_m0_1}),
 and the  regret after the $N$-th episode is of the magnitude $\cO(N^{0.51})$.
 This confirms the theoretical results in Theorem \ref{cor:regret_diffusion} even for a small $m_0$.
 \end{itemize}

% Since at the first several iterations, the numbers of sample trajectories used to update parameters are quite limited, we observe huge increase in both parameter estimation errors and cost suboptimality gap, which are larger compared to the case of $m_0=4$. Such limited samples at beginning leads to a wider range of $A_\ell$ and $B_\ell$ and therefore a larger variance on the algorithm regret. However, it is still obvious to see in Figures \ref{fig:est_err_3d_m0_1} and \ref{fig:subopt_3d_m0_1} that as the number of episodes increases, the estimate $(A_\ell,B_\ell)_{\ell}$ converge to the true parameter $(A^\star,B^\star)$ and the suboptimality gap converges to $0$. We can also observe a sub-linear regret in Figure \ref{fig:regret_3d_m0_1}. In particular, after 12-th update ($\ell=12$), the relative approximation errors of $A_\ell$ and $B_\ell$ are 3.00\% and 18.3\% respectively; the cost of  $\psi^{\theta_\ell}$
%  approximates the optimal cost accurately with a relative error 1.6\%, which is slightly larger compared to the case where $m_0=4$. Performing a linear regression for logarithms
% of expected regret and episode
%  shows that the  regret after the $N$-th episode is of the magnitude $\cO(N^{0.51})$. These results further demonstrate the robustness of Algorithm \ref{alg:greedy}, given its good performance of dealing with large perturbations due to the choice of smaller $m_0$.
 
\begin{figure}[H]
\centering
\begin{subfigure}{0.32\textwidth}
    \includegraphics[trim=18 5 30 25, clip,  width=\textwidth]{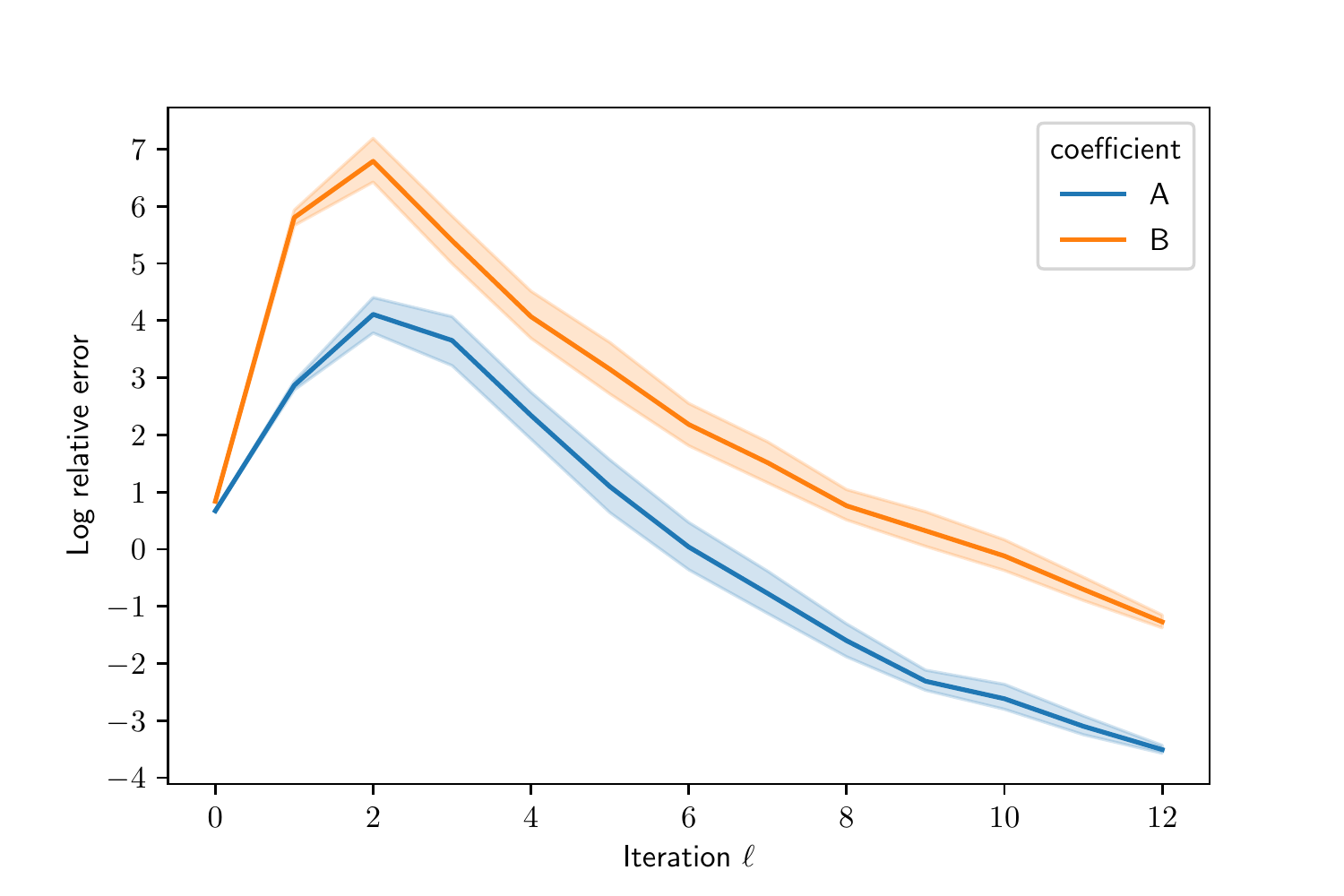}
    \caption{Parameter estimation errors.}
    \label{fig:est_err_3d_m0_1}
\end{subfigure}
%\hfill
\begin{subfigure}{0.32\textwidth}
    \includegraphics[trim=18 5 30 25, clip, width=\textwidth]{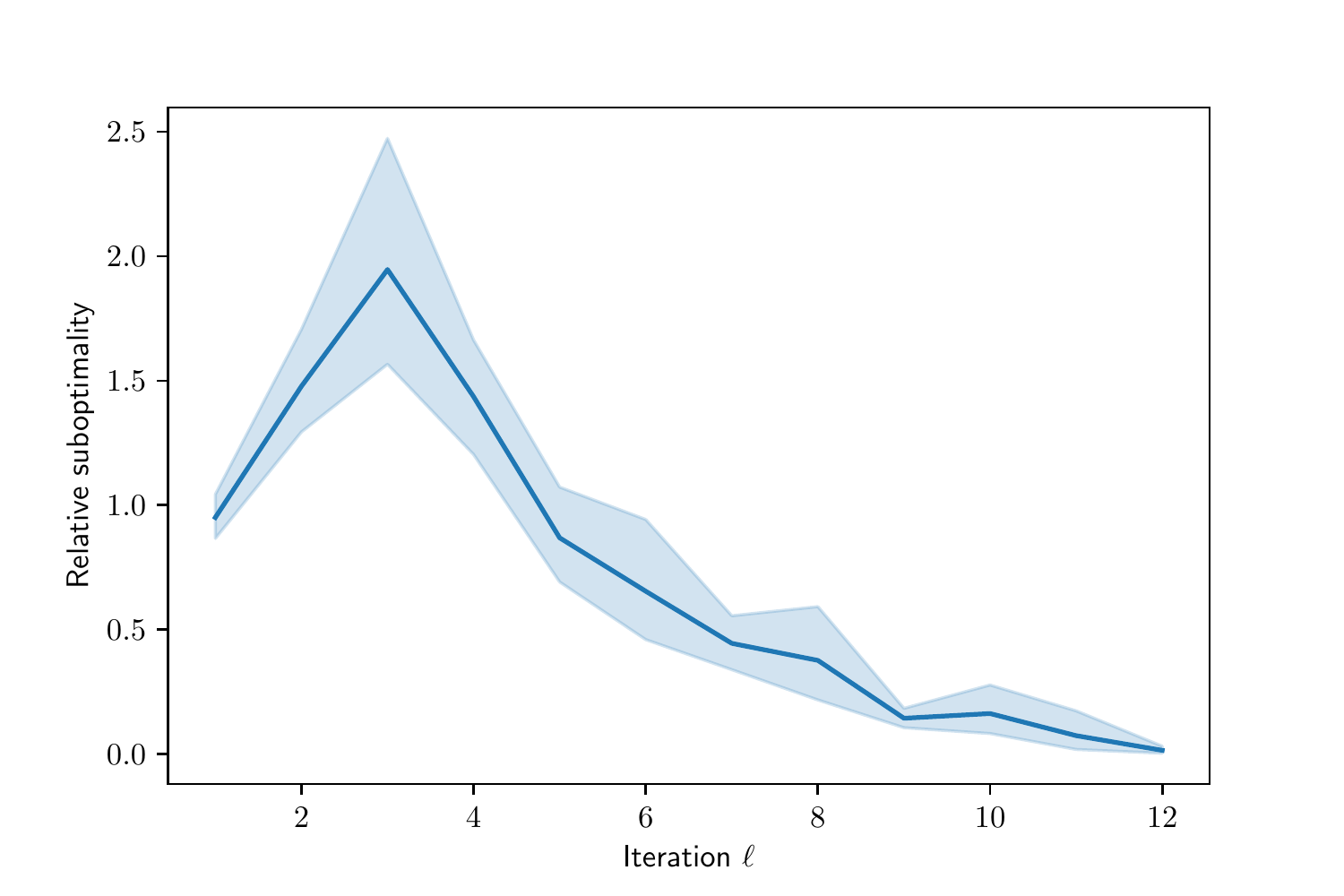}
    \caption{Cost suboptimality gap.}
    \label{fig:subopt_3d_m0_1}
\end{subfigure}
%\hfill
\begin{subfigure}{0.32\textwidth}
    \includegraphics[trim=18 5 30 25, clip,  width=\textwidth]{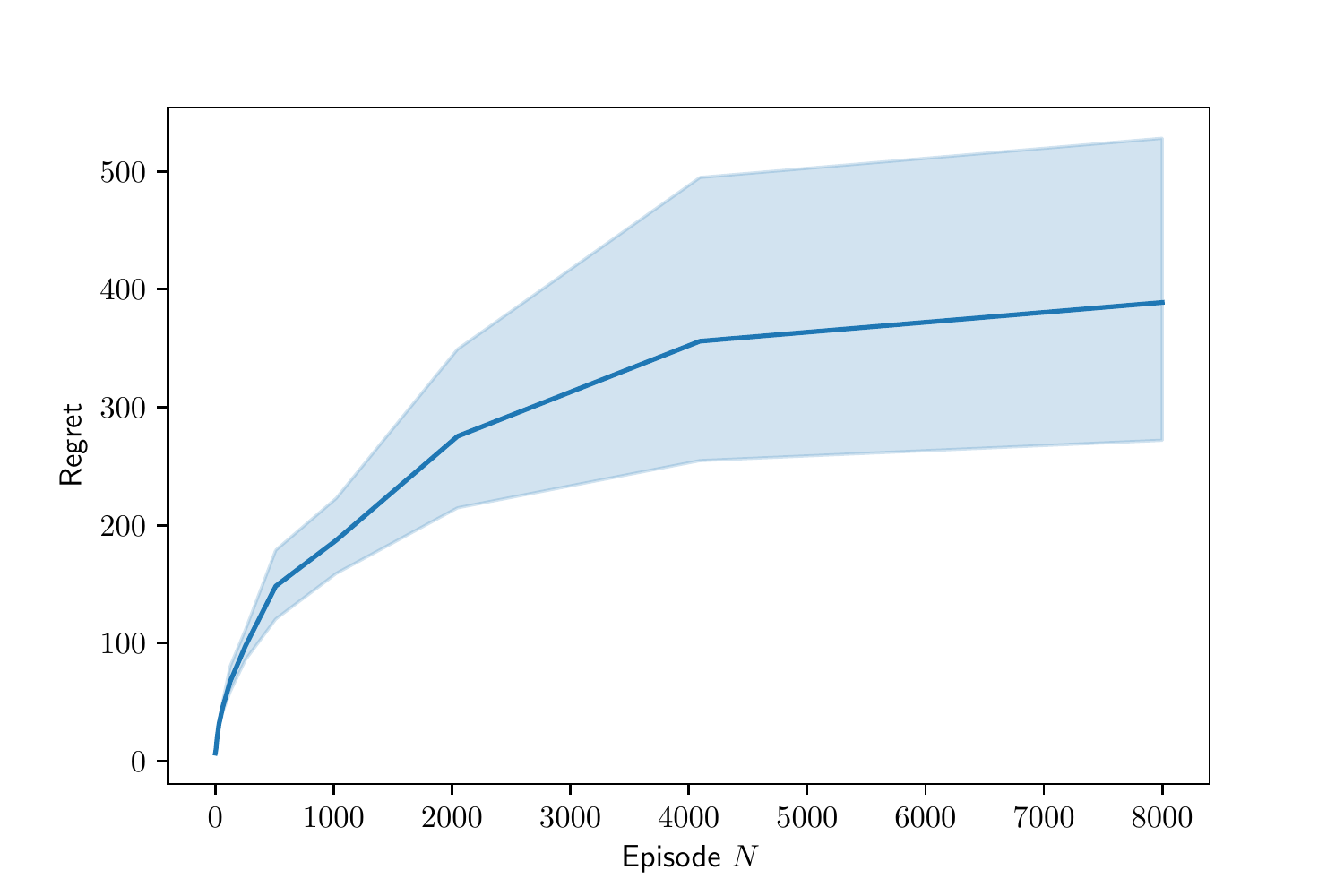}
    \caption{Algorithm regret.}
    \label{fig:regret_3d_m0_1}
\end{subfigure}
        
\caption{Performance of Algorithm \ref{alg:greedy} for the LQ-RL problem ($m_0=1$).}
\label{fig:3d_m0_1}
\end{figure}

}
\begin{appendices}

\section{Preliminaries}

Here, we collect some fundamental  results which are used for our analysis.

We start with  a stability result for coupled FBSDEs 
under a generalized monotonicity condition, 
which is crucial for our stability analysis of feedback controls. 
For any given 
  $t\in [0,T]$
 and $\lambda\in [0,1]$,
 we consider the following FBSDE 
 defined on $[t,T]$:
for $s\in [t,T]$,
\begin{subequations}\l{eq:moc}
\begin{alignat}{2}
\d X_s&=(\lambda \bar{b}(s,X_s,Y_s)+\cI^b_s)\,\d s +
\bar{\sigma}(s)\, \d W_s
+\int_{\sR^p_0}
\bar{\gamma}(s,u)\, \tilde{N}(\d s,\d u),
&&
\q X_t=\xi,
\l{eq:moc_sde}
\\
\d Y_s&=-(\lambda \bar{f}(s,X_s,Y_s)+\cI^f_s)\,\d t+Z_s\,\d W_s
+\int_{\sR^p_0}
M_s\, \tilde{N}(\d s,\d u),
&&
\q Y_T=\lambda \bar{g}(X_T)+\cI^g,
\l{eq:moc_bsde}
\end{alignat}
\end{subequations}
with given 
$\xi\in L^2(\cF_t;\sR^n)$,
$(\cI^b,\cI^f)\in \cH^2(\sR^n\t\sR^{n})$, $\cI^g\in L^2(\cF_T;\sR^m)$
and measurable functions
$\bar{\sigma}:[0,T]\to \sR^{n\t d}$,
$\bar{\gamma}:[0,T]\t \sR^p_0\to \sR^{n}$, 
  $\bar{b},\bar{f}:[0,T]\t \sR^n\t \sR^n\to\sR^n$
  and  $\bar{g}:\sR^n\to \sR^n$.

\begin{Lemma}\l{lemma:mono_stab}
Let $K\ge 0$, for each $i\in \{1,2\}$, let 
$\bar{b}_i,\bar{f}_i:[0,T]\t \sR^n\t \sR^n\to\sR^n$, $\bar{g}_i:\sR^n\to \sR^n$
satisfy
$\int_0^T(|\bar{b}_i(t,0,0)|^2+|\bar{f}_i(t,0,0)|^2) \,\d t<\infty$
and 
  for all $t\in [0,T]$ that 
$\bar{f}_i(t,\cdot), \bar{g}_i$ are $K$-Lipschitz continuous,
let 
$\bar{\sigma}_i:[0,T]\to \sR^{n\t d}$ satisfy $\int_0^T|\bar{\sigma}_i(t)|^2\,\d t<\infty$
and let 
$\bar{\gamma}_i:[0,T]\t \sR^p_0\to \sR^{n}$ satisfy $\int_0^T\int_{\sR^p_0} |\bar{\gamma}_i(t,u)|^2\,\nu(\d u)\d t<\infty$.
Assume further that  there exists $\tau>0$ 
and a measurable function $\eta:[0,T]\t \sR^n\t \sR^n\t \sR^n\t \sR^n\to [0,\infty)$
such that for all 
$t\in [0,T], (x,y),(x',y')\in  \sR^n\t  \sR^n$,
\begin{align}
\la \bar{b}_1(t,x,y)- \bar{b}_1(t,x',y'),y-y'\ra
&  +\la -\bar{f}_1(t,x,y)+\bar{f}_1(t,x',y'),x-x'\ra 
\le -\tau \eta(t,x,y,x',y'),
\l{eq:moc_mono}
\\
|\bar{b}_1(t,x,y)- \bar{b}_1(t,x',y')|
&\le K(|x-x'|+ \eta(t,x,y,x',y')),
\l{eq:lipschitz_eta}
\\
\la  \bar{g}(x)-\bar{g}(x'),x-x'\ra
\l{eq:moc_g}
&\ge 0.
 \end{align}
Then there exists $C>0$, depending only on $T, K, \lambda$ and the dimensions, such that 
 for all
 $t\in [0,T]$,
  $\lambda_0\in [0,1]$,
 $i\in \{1,2\}$,
for every   
$(X_i,Y_i, Z_i,M_i)\in  \cS^2(t,T;\sR^n) \t \cS^2(t,T;\sR^n) \t \cH^2(t,T;\sR^{n\t d})\t \cH^2_\nu(t,T;\sR^{n})$
satisfying \eqref{eq:moc}
with 
$\lambda=\lambda_0$,
   $(\bar{b},\bar{\sigma},\bar{\gamma},\bar{f},\bar{g})
   =(\bar{b}_i,\bar{\sigma}_i,\bar{\gamma}_i,\bar{f}_i,\bar{g}_i)$,
$\xi=\xi_i \in L^2(\cF_t;\sR^n)$,
$\cI^g=\cI^g_i \in L^2(\cF_T;\sR^n)$
and
$(\cI^b,\cI^f)=(\cI^b_i,\cI^f_i)\in \cH^2(\sR^n\t\sR^n)$, we have that 
 \begin{align}\l{eq:mono_stab}
 \begin{split}
 &\|X_1-{X}_2\|_{\cS^2}^2+ \|Y_1-{Y}_2\|_{\cS^2}^2+ \|Z_1-{Z}_2\|_{\cH^2}^2
 + \|M_1-{M}_2\|_{\cH^2_\nu}^2
 \\
& \le C\Big\{\|\xi_1-{\xi_2}\|_{L^2}^2+
\|\lambda_0 (\bar{g}_1({X}_{2,T})-\bar{g}_2({X}_{2,T}))+\cI^g_1-{\cI}^g_2\|_{L^2}^2
+\|\bar{\sigma}_1-\bar{\sigma}_2\|_{\cH^2}^2+\|\bar{\gamma}_1-\bar{\gamma}_2\|_{\cH^2_\nu}^2
\\
&\quad
+\|\lambda_0 (\bar{b}_1(\cdot,{X}_2,{Y}_2)-\bar{b}_2(\cdot,{X}_2,{Y}_2))
+\cI^b_1-{\cI}^b_2\|_{\cH^2}^2
\\
&\quad
+\|\lambda_0 (\bar{f}_1(\cdot,{X}_2,{Y}_2)-\bar{f}_2(\cdot,{X}_2,{Y}_2))
+\cI^f_1-{\cI}^f_2\|_{\cH^2}^2
\Big\}.
\end{split}
  \end{align}

\end{Lemma}

\begin{proof}
Throughout this proof, let $C$ be a generic constant depending only on $T$,  $K$, $\lambda$ and the dimensions,
 let $t\in[0,T]$,  $\lambda_0\in [0,1]$,
let
$(\delta X,\delta Y,\delta  Z,\delta  M)
=(X_1-X_2,Y_1-Y_2, Z_1-Z_2,M_1-M_2)$,
$\delta \xi=\xi_1-{\xi}_2$,
$\delta \sigma=\bar{\sigma}_1-\bar{\sigma}_2$,
$\delta \gamma=\bar{\gamma}_1-\bar{\gamma}_2$,
$\delta \cI^g=\cI^g_1-{\cI}^g_2$,
and
for each $s\in [t,T]$ 
let  $\delta\cI^b_s=\cI^b_{1,s}-{\cI}^b_{2,s}$,
 $\delta\cI^f_s=\cI^f_{1,s}-{\cI}^f_{2,s}$,
  $\bar{b}_1(\Theta_{1,s})=\bar{b}_1(t,X_{1,s},Y_{1,s})$,
    $\bar{b}_1(\Theta_{2,s})=\bar{b}_1(t,X_{2,s},Y_{2,s})$
    and
    $\bar{b}_2(\Theta_{2,s})=\bar{b}_2(t,X_{2,s},Y_{2,s})$.
Similarly, we introduce the notation
$\bar{f}_1(\Theta_{1,s}), \bar{f}_1(\Theta_{2,s}), \bar{f}_2(\Theta_{2,s})$ for  $s\in [t,T]$.

By applying It\^{o}'s formula to $(\la Y_{1,s}-Y_{2,s},X_{1,s}-X_{2,s}\ra)_{s\in [t,T]} $, we can obtain 
from \eqref{eq:moc}
that  
\begin{align*}
&\sE[\la \lambda_0(\bar{g}_1(X_{1,T})-\bar{g}_2(X_{2,T}))+\delta \cI^g,\delta X_{T}\ra-\la \delta Y_{t}, \delta \xi\ra]
\\
&=\sE\bigg[\int_t^T
\bigg(\la \lambda_0(\bar{b}_1(\Theta_{1,s})-\bar{b}_2(\Theta_{2,s}))+\delta \cI^b_s, \delta Y_{s}\ra 
%\la \lambda_0(\sigma(\Theta_t)-\bar{\sigma}(\bar{\Theta}_t))+\delta \cI^\sigma_t, G^*( Z_t-\bar{Z}_t)\ra
%\\
%&\quad 
-
\la \lambda_0(\bar{f}_1(\Theta_{1,s})-\bar{f}_2(\Theta_{2,s}))+\delta \cI^f_s, \delta X_{s}\ra
\\
&\quad 
+\la \delta \sigma(s), \delta Z_{s}\ra
+\int_{\sR^p_0}\la \delta \gamma(s,u), \delta M_{s}\ra\, \nu(\d u)\bigg)\, \d s
\bigg]
\\
&\le \sE\bigg[\int_t^T
\bigg(
-\lambda_0 \tau\eta(s, X_{1,s}, Y_{1,s},X_{2,s},Y_{2,s})
+
\la \lambda_0(\bar{b}_1(\Theta_{2,s})-\bar{b}_2(\Theta_{2,s}))+\delta \cI^b_s, \delta Y_{s}\ra 
\\
&\quad 
-
\la \lambda_0(\bar{f}_1(\Theta_{2,s})-\bar{f}_2(\Theta_{2,s}))+\delta \cI^f_s, \delta X_{s}\ra
+\la \delta \sigma(s), \delta Z_{s}\ra
+\int_{\sR^p_0}\la \delta \gamma(s,u), \delta M_{s}\ra\, \nu(\d u)\bigg)\, \d s
\bigg],
\end{align*}
where for the last inequality, we have added and subtracted the terms $\lambda_0\bar{b}_1(\Theta_{2,s})$
and $-\lambda_0\bar{f}_1(\Theta_{2,s})$,
and applied \eqref{eq:moc_mono}.
Then, we can further deduce from \eqref{eq:moc_g} that
\begin{align*}
&
\lambda_0 \tau\sE\bigg[\int_t^T
\eta(s, X_{1,s}, Y_{1,s},X_{2,s},Y_{2,s})\, \d s\bigg]
\\
&\le
-\sE[\la \lambda_0(\bar{g}_1(X_{2,T})-\bar{g}_2(X_{2,T}))+\delta \cI^g,\delta X_{T}\ra-\la \delta Y_{t}, \delta \xi\ra]
\\
&\q 
+ \sE\bigg[\int_t^T
\bigg(
\la \lambda_0(\bar{b}_1(\Theta_{2,s})-\bar{b}_2(\Theta_{2,s}))+\delta \cI^b_s, \delta Y_{s}\ra 
-
\la \lambda_0(\bar{f}_1(\Theta_{2,s})-\bar{f}_2(\Theta_{2,s}))+\delta \cI^f_s, \delta X_{s}\ra
\\
&\quad 
+\la \delta \sigma(s), \delta Z_{s}\ra
+\int_{\sR^p_0}\la \delta \gamma(s,u), \delta M_{s}\ra\, \nu(\d u)\bigg)\, \d s
\bigg],
\end{align*}
from which we can apply Young's inequality and obtain for all $\eps>0$ that
\begin{align*}
&
\lambda_0 \sE\bigg[\int_t^T
\eta(s, X_{1,s}, Y_{1,s},X_{2,s},Y_{2,s})\, \d s\bigg]
\\
&\le
\eps\Big(\|\delta X_{T}\|_{L^2}^2+\|\delta Y_{t}\|_{L^2}^2
+\|\delta X\|_{\cH^2}^2+\|\delta Y\|_{\cH^2}^2
+
\|\delta Z\|_{\cH^2}^2+\|\delta M\|_{\cH^2_\nu}^2
\Big)+C\textrm{RHS}/\eps,
\end{align*}
where \textrm{RHS} denotes the terms at the right-hand side of \eqref{eq:mono_stab}.

By 
\eqref{eq:lipschitz_eta} and a
standard stability estimate of \eqref{eq:moc_sde}, we can deduce that 
\begin{align}\l{eq:moc_sde_estimate}
% \begin{split}
% \|\delta X\|_{\cS^2}^2
% &\le C\bigg(
% \|\delta \xi\|_{L^2}^2
% +\lambda_0 \sE\bigg[\int_t^T
% \eta(s, X_{1,s}, Y_{1,s},X_{2,s},Y_{2,s})\, \d s\bigg]
% +\|\delta{\sigma}\|_{\cH^2}^2+\|\delta{\gamma}\|_{\cH^2_\nu}^2
% \bigg)
% \\
% &\le 
% \eps\Big(\|\delta Y_{t}\|_{L^2}^2
% +\|\delta Y\|_{\cH^2}^2
% +
% \|\delta Z\|_{\cH^2}^2+\|\delta M\|_{\cH^2_\nu}^2
% \Big)+C\textrm{RHS}/\eps
% \end{split}
% \end{align}
\begin{split}
\|\delta X\|_{\cS^2}^2
&\le C\bigg(
\lambda_0 \sE\bigg[\int_t^T
\eta(s, X_{1,s}, Y_{1,s},X_{2,s},Y_{2,s})\, \d s\bigg]
+\textrm{RHS}
\bigg)
\\
&\le 
\eps C\Big(\|\delta Y_{t}\|_{L^2}^2
+\|\delta Y\|_{\cH^2}^2
+
\|\delta Z\|_{\cH^2}^2+\|\delta M\|_{\cH^2_\nu}^2
\Big)+C\textrm{RHS}/\eps
\end{split}
\end{align}
for all small enough $\eps>0$. Moreover, by  the Lipschitz continuity of $\bar{f}_1$, $\bar{g}_1$ and  the stability estimate of \eqref{eq:moc_bsde} 
(see e.g.~\cite[Proposition A4]{quenez2013bsdes}), we have that 
\begin{align*}
&\|\delta Y\|_{\cS^2}^2
+\|\delta Z\|_{\cH^2}^2+\|\delta M\|_{\cH^2_\nu}^2
\\
&\le C\Big(
\|\lambda_0 (\bar{g}_1(X_{1,T})-\bar{g}_2(X_{2,T}))+\delta \cI^g\|_{L^2}^2
+ 
\|\lambda_0 (\bar{f}_1(\cdot,X_{1},Y_{2})-\bar{f}_2(\cdot,X_{2},Y_{2}))+\delta \cI^f\|^2_{\cH^2}
\Big)
\\
&\le C\Big(
\|\delta X\|^2_{\cS^2}+
\|\lambda_0 (\bar{g}_1(X_{2,T})-\bar{g}_2(X_{2,T}))+\delta \cI^g\|_{L^2}^2
+ 
\|\lambda_0 (\bar{f}_1(\cdot,X_{2},Y_{2})-\bar{f}_2(\cdot,X_{2},Y_{2}))+\delta \cI^f\|^2_{\cH^2}
\Big)
\\
&\le C\textrm{RHS},
\end{align*}
where  we have applied \eqref{eq:moc_sde_estimate} with a sufficiently small $\eps$ for the last inequality.
This completes the desired stability estimate.
\end{proof}

We then present
 a version of Burkholder's inequality 
for the $\|\cdot\|_{L^q}$-norm of stochastic integrals,
which not only extends %\cite[Corollary 2.1]{breton2019integrability}
 \cite[Corollary 2.2]{breton2019integrability}
to 
stochastic integrals with respect to
general Poisson random
measures  on $[0,T]\t \sR^p_0$,
but also improves the bounding constants
there
%in \cite[Corollary 2.1]{breton2019integrability}
with a sharper dependence on the index $q$.
 
\begin{Lemma}\l{lemma:bdg}
% Let $v:\Om\t [0,T]\to \sR^d$ and $w:\Om\t [0,T]\t \sR^p_0\to \sR$ be adapted processes such that
% $\sE[\int_0^T |v|^2\,\d t]<\infty$ and $\sE[\int_0^T\int_{\sR^p_0} |w(t,u)|^2\,\nu(\d u)\d t]<\infty$. Then 
%  for any $q\ge 2$,
For all 
$v\in \cH^2(0,T;\sR^d)$,
$w\in \cH^2_\nu(0,T;\sR)$
and $q\ge 2$,
we have
\begin{align}
\sE\bigg[\bigg|\int_0^T v^\trans_t\,\d W_t\bigg|^q\bigg]
&\le C_q\sE\bigg[
\bigg(
\int_0^T |v_t|^2\,\d t
\bigg)^{q/2}\bigg],
\l{eq:bdg_brownian}
\\
\sE\bigg[\bigg|\int_0^T\int_{\sR^p_0} w(t,u)\,\tilde{N}(\d t, \d u)\bigg|^q\bigg]
&\le \tilde{C}_q
\bigg(
\sE\bigg[\int_0^T\int_{\sR^p_0} |w(t,u)|^{q}\,\nu(\d u)\d t 
\bigg]
\nb
\\
&\q +
\sE\bigg[\bigg(\int_0^T\int_{\sR^p_0} |w(t,u)|^{2}\,\nu(\d u)\d t 
\bigg)^{q/2}\bigg]
\bigg),
\l{eq:bdg_poisson}
\end{align}
where $C_q=(\sqrt{e/2}q)^q$ 
and $\tilde{C}_q=
21{e}^{q}
q^{ 2q}$.
\end{Lemma}
\begin{proof}
Recall that Burkholder's inequality in 
\cite[Theorem 4.2.12]{bichteler2002stochastic} shows 
for all $1\le q<\infty$ and for every local martingale $(M_t)_{t\in [0,T]}$ that 
$\sE[|M^*_T|^q]\le C_q\sE[[M,M]_T^{q/2}]$,
where 
$M^*_T=\sup_{t\in [0,T]}|M_t|$,
$[M,M]$ is the quadratic variation of $M$,
and $C_q=(\sqrt{10q})^q$ for $q\in [1,2)$ and $C_q=(\sqrt{e/2}q)^q$ for $q\in [2,\infty)$.
Hence we can  obtain
\eqref{eq:bdg_brownian}  by setting $M_t=\int_0^t v^\trans_s\,\d W_s$ for all $t\in [0,T]$,
whose  quadratic variation process is given by  
$[M,M]_T=\int_0^T |v|^2\,\d t$.

We  proceed  to establish \eqref{eq:bdg_poisson}
by following the arguments of %\cite[Lemma 2]{breton2019integrability}
\cite[Lemma 2.1]{breton2019integrability}.
For all $r\ge 1$ and $t\in [0,T]$, let $K^{(r)}_t=\int_0^t\int_{\sR^p_0} w(s,u)^r\,\tilde{N}(\d s, \d u)$.
For any given $r\ge 1$ and $q\ge 2$, we can obtain from Burkholder's inequality
and $\tilde{N}(\d t,\d u)=N(\d t,\d u)-\nu(\d u)\d t$
  that 
\begin{align*}
\sE[|(K^{(r)})^*_T|^q]
&\le {C_q} \sE\bigg[\bigg(\int_0^T\int_{\sR^p_0} |w(t,u)|^{2r}\,{N}(\d t, \d u)\bigg)^{q/2}\bigg]
\\
&= {C_q} \sE\bigg[\bigg(\int_0^T\int_{\sR^p_0} |w(t,u)|^{2r}\,\tilde{N}(\d t, \d u)
+\int_0^T\int_{\sR^p_0} |w(t,u)|^{2r}\,\nu(\d u)\d t 
\bigg)^{q/2}\bigg]
\\
&\le 2^{\frac{q}{2}-1}{C_q} 
\sE[|(K^{(2r)})^*_T|^{q/2}]
+2^{\frac{q}{2}-1}{C_q} 
\sE\bigg[\bigg|\int_0^T\int_{\sR^p_0} |w(t,u)|^{2r}\,\nu(\d u)\d t 
\bigg|^{q/2}\bigg].
\end{align*}
%where 
%for the last inequality 
%we have used the fact that 
%$|(x+y)/2|^{l}\le (|x|^{l}+|y|^{l})/2$ for all $l\ge 1$ and $x,y\in \sR$.
Hence,  recursively applying the above estimate yields for all $q\ge 2$ and 
$n\in \sN$ with $q/2^{n-1}\ge 2$ that 
\begin{align}\l{eq:bdg_poisson_induction}
\begin{split}
\sE\bigg[\bigg|\int_0^T\int_{\sR^p_0} w(t,u)\,\tilde{N}(\d t, \d u)\bigg|^q\bigg]
&\le 
\bigg(
\prod_{j=1}^n 2^{\frac{q}{2^j}-1}{C_{\frac{q}{2^{j-1}}}} 
\bigg)
\sE[|(K^{(2^n)})^*_T|^{q/2^n}]
\\
&
\q
+
\sum_{k=1}^n
\bigg(
\prod_{j=1}^k 2^{\frac{q}{2^j}-1}{C_{\frac{q}{2^{j-1}}}} 
\bigg)
\sE\bigg[\bigg|\int_0^T\int_{\sR^p_0} |w(t,u)|^{2^k}\,\nu(\d u)\d t 
\bigg|^{q/2^k}\bigg].
\end{split}
\end{align}
Now let $q\ge 2$ be fixed and set $n=\lfloor \log_2 q\rfloor$ such that 
$q\in [2^n,2^{n+1})$.
Since $q/2^n\in [1,2)$, 
the constant $C_{q/2^n}$ in Burkholder's inequality
satisfies $C_{q/2^n}\le 20$,
from which
%, by  proceeding along the lines of 
we can show that (see \cite[Lemma 2.1]{breton2019integrability}):
\begin{align*}
\sE[|(K^{(2^n)})^*_T|^{q/2^n}]
&
\le 20
\sE\bigg[\int_0^T\int_{\sR^p_0} |w(t,u)|^q\,\nu(\d u)\d t 
\bigg].
\end{align*}
Moreover, by  proceeding along the lines of  \cite[Corollary 2.2]{breton2019integrability},
we   
obtain for all $k=1,\ldots, n$ that 
\begin{align*}
&\sE\bigg[\bigg|\int_0^T\int_{\sR^p_0} |w(t,u)|^{2^k}\,\nu(\d u)\d t 
\bigg|^{q/2^k}\bigg]
\\
&
\le 
\sE\bigg[\int_0^T\int_{\sR^p_0} |w(t,u)|^{q}\,\nu(\d u)\d t 
\bigg]
+
\sE\bigg[\bigg|\int_0^T\int_{\sR^p_0} |w(t,u)|^{2}\,\nu(\d u)\d t 
\bigg|^{q/2}\bigg].
\end{align*}
Hence, we can deduce  from \eqref{eq:bdg_poisson_induction} that 
\begin{align*}
\sE\bigg[\bigg|\int_0^T\int_{\sR^p_0} w(t,u)\,\tilde{N}(\d t, \d u)\bigg|^q\bigg]
&\le 
21
\sum_{k=1}^{\lceil \log_2 q\rceil}
\bigg(
\prod_{j=1}^k 2^{\frac{q}{2^j}-1}{C_{\frac{q}{2^{j-1}}}} 
\bigg)
\bigg(
\sE\bigg[\int_0^T\int_{\sR^p_0} |w(t,u)|^{q}\,\nu(\d u)\d t 
\bigg]
\\
&\q +
\sE\bigg[\bigg|\int_0^T\int_{\sR^p_0} |w(t,u)|^{2}\,\nu(\d u)\d t 
\bigg|^{q/2}\bigg]
\bigg).
\end{align*}
We now obtain an upper bound of the constant 
$21\sum_{k=1}^{\lfloor \log_2 q\rfloor}
\left(
\prod_{j=1}^k 2^{\frac{q}{2^j}-1}{C_{\frac{q}{2^{j-1}}}} 
\right)$ as follows:
\begin{align*}
&21\sum_{k=1}^{\lfloor \log_2 q\rfloor}
\bigg(
\prod_{j=1}^k 2^{\frac{q}{2^j}-1}{C_{\frac{q}{2^{j-1}}}} 
\bigg)
= 
21\sum_{k=1}^{\lfloor \log_2 q\rfloor}\prod_{j=1}^k
2^{\frac{q}{2^j}-1}
\bigg(\sqrt{\frac{e}{2}}
\frac{q}{2^{j-1}}
\bigg)^\frac{q}{2^{j-1}}
\\
&
\le 
21
\bigg(\sum_{k=1}^{\lfloor \log_2 q\rfloor} 2^{-k}\bigg)
{e}^{\sum_{j=1}^{\lfloor\log_2 q\rfloor}\frac{q}{2^{j}}}\prod_{j=1}^{\lfloor \log_2 q\rfloor}
\bigg(
\frac{q}{2^{j-1}}
\bigg)^\frac{q}{2^{j-1}}
\le
21
{e}^{q}
2^{ \sum_{j=1}^{\lfloor \log_2 q\rfloor}
\frac{q}{2^{j-1}}\log_2
(
\frac{q}{2^{j-1}}
)
}
\\
&
\le
21
{e}^{q}
2^{ \sum_{j=1}^{\lfloor \log_2 q\rfloor}
\frac{q}{2^{j-1}}\log_2 q
}
\le
21
{e}^{q}
2^{ 2q\log_2 q
}
=
21
{e}^{q}
q^{ 2q}\coloneqq \tilde{C}_q,
\end{align*}
which leads  us to  the desired conclusion.
\end{proof}

The following lemma  
estimates the tail behaviors  of solutions to  SDEs 
with jumps.
The result 
has been established 
in 
Lemma 2.1 and Theorem 2.8 
of \cite{ma2010transportation}
for SDEs with time homogenous coefficients
and 
bounded Lipschitz continuous functions $\mathfrak{f}$
via Malliavin Calculus,
which can be extended 
to SDEs with time inhomogeneous coefficients
and unbounded $\mathfrak{f}$
(via Fatou's lemma) 
in a straightforward manner.

\begin{Lemma}\l{lemma:sde_transportation}
Let $T\ge 0$ and $b:[0,T]\t \sR^n\to \sR^n$, $\sigma :[0,T]\to \sR^{n\t d}$, 
$\gamma:[0,T]\t \sR^p_0\to \sR^{n}$ 
be measurable functions 
such that 
there exist $K,\sigma_{\max}\ge 0$ and a measurable function 
$\bar{\gamma}:\sR^p_0\to \sR$
satisfying for all $(t,u)\in [0,T]\t \sR^p_0$, $x,x'\in \sR^n$ that 
$|b(t,0)|\le K$, $|b(t,x)-b(t,x')|\le K|x-x'|$,
$|\sigma(t)|\le \sigma_{\max}$
and
$ |\gamma(t,u)|\le \bar{\gamma}(u)$, $\nu$-a.e..~Let $\b:[0,\infty)\to [0,\infty]$ be   defined by 
$\b(\lambda)\coloneqq \int_{\sR^p_0} \left(e^{\lambda \bar{\gamma}(u)}-\lambda \bar{\gamma}(u)-1\right)\,\nu(\d u)$
 for any $\lambda\ge 0$.
 Assume  that 
$\b(\lambda)<\infty$ for some $\lambda> 0$.

Then there exists a constant $C>0$, depending only on $K$ and $T$, such that 
for all $x\in \sR^n$, the  unique solution $X^x\in \cS^2(\sR^n)$  to the following SDE
$$%\bb\l{eq:sde_subexp}
\d X_t =b(t,X_t)\,\d t+\sigma(t)\, \d W_t
+\int_{\sR^p_0}
\gamma(t,u)\, \tilde{N}(\d t,\d u)
, \q t\in [0,T],
\q X_0=x
$$%\ee
satisfies for every Lipschitz continuous function
$\mathfrak{f}:(\sD([0,T];\sR^n),d_\infty)\to \sR$ 
that 
\bb\l{eq:f_subexp}
\sE\left[e^{\lambda(\mathfrak{f}(X^x)-\sE[\mathfrak{f}(X^x)])}\right]
\le 
e^{C\eta\big(C\lambda \|\mathfrak{f}\|_{\textnormal{Lip}}\big)}
\q \fa \lambda>0,
\ee
where 
$\sD([0,T];\sR^n)$ is the space of $\sR^n$-valued
 c\`adl\`{a}g functions on $[0,T]$,
$d_\infty$ is   the uniform metric defined by
 $d_\infty(\rho_1,\rho_2)\coloneqq\sup_{t\in [0,T]}|\rho_1(t)-\rho_2(t)|$
 for any $\rho_1,\rho_2\in \sD([0,T];\sR^n)$,
 $\|\mathfrak{f}\|_{\textnormal{Lip}}$ 
 is the constant defined by
 $\|\mathfrak{f}\|_{\textnormal{Lip}}\coloneqq
\sup_{\rho_1\not=\rho_2}\f{|\mathfrak{f}(\rho_1)-\mathfrak{f}(\rho_2)|}{d_\infty(\rho_1,\rho_2)}$,
and $\eta:[0,\infty)\to [0,\infty]$ is the function defined by
$\eta(\lambda) \coloneqq \b(\lambda)+\sigma_{\max}^2\lambda^2/2$
for any $\lambda\ge 0$.
\end{Lemma}

The next lemma presents a concentration inequality for 
the sum of 
independent sub-Weibull 
random variables, which follows directly from
Theorem 3.1 and Proposition A3 in
\cite{kuchibhotla2018moving}.

\begin{Lemma}\l{lemma:concentration_alpha_subWeibull}
Let 
$\a\in (0,1]$, 
$N\in \sN$ and $X_1, \ldots, X_N\in \subW(\a)$ be independent    random variables
satisfying $\sE[X_i]=0$ for all $i=1,\ldots, N$.
Then there exists a constant $C\ge 0$, depending only on $\a$, such that 
%for all $\eps'\ge 0$,
$$
\sP\bigg(\bigg|\sum_{i=1}^N X_i\bigg|\ge \eps' \bigg)\le 
2\exp\bigg(-C\min
\left\{
\frac{(\eps')^2}{\sum_{i=1}^N \|X_i\|^2_{\Psi_\a}},
\bigg(
\frac{\eps'}{\max_{i} \|X_i\|_{\Psi_\a}}
\bigg)^{\a}
\right\}\bigg),
\q
\fa \eps'\ge 0.
$$

\end{Lemma}
%\begin{proof}
%By using 
%Theorem 3.1 and Proposition A3 in 
%\cite{kuchibhotla2018moving},
%we see there exists a constant $C\ge 0$, depending only on $\a$,
%such that for all $t\ge 0$,  
%$$
%\sP\bigg(\bigg|\sum_{i=1}^N X_i\bigg|\ge C(D_1 \sqrt{t}+D_2 t^{1/\a})\bigg)\le 2e^{-t}
%$$
%with $D_1\coloneqq (\sum_{i=1}^N \|X_i\|^2)^{1/2}$
% and $D_2\coloneqq \max_{i}\|X_i\|_{{\Psi_\a}}$.
%
%\end{proof}

\section{Proofs of Lemmas \ref{lemma:f_psi},
\ref{lemma:deterministic_integral},
\ref{lemma:stochastic_integral_jump},
\ref{lemma:stochatic_integral_diffusion},
\ref{lemma:state_subexp}
} 
\l{appendix:technical results}

 \begin{proof}[Proof of Lemma \ref{lemma:f_psi}]
 We start by establishing the regularity of 
 $\sR^n\t \sR^k\ni (x,z)\mapsto f(t,x,\p_z f^*(t,x,z))\in \sR\cup\{\infty\}$ for a given $t\in [0,T]$.
 Observe that  for all $(t,x)\in [0,T]\t \sR^n$, 
   $f(t,x,\cdot)$ is proper, convex, and lower semicontinuous,
   which  
along with the   Fenchel-Young identity
implies
\bb\l{eq:conv_max}
f(t,x,\p_z f^*(t,x,z))=\la z, \p_z f^*(t,x,z)\ra- f^*(t,x,z)
\in \sR,
\q
\fa (t,x,z)\in [0,T]\t \sR^n\t \sR^k.
\ee
 Given $t\in [0,T]$ and $(x_1,z_1),(x_2,z_2)\in \sR^n\t \sR^k$, 
by \eqref{eq:conv_max}, 
 \begin{align*}%\l{eq:ff*}
&| f(t,x_1,\p_z f^*(t,x_1,z_1))-f(t,x_2,\p_z f^*(t,x_2,z_2))|
\\
 &\le| \la z_1, \p_z f^*(t,x_1,z_1)\ra -\la z_2, \p_z f^*(t,x_2,z_2) \ra |
 +| f^*(t,x_1,z_1)- f^*(t,x_2,z_1)|
 \\
& \q +| f^*(t,x_2,z_1)- f^*(t,x_2,z_2)|.
 \end{align*}
We now estimate all the terms on the right hand side of the above inequality. 
By  the Lipchitz continuity and local boundedness of  $\p_zf^*(t,\cdot)$
(see the proof of Lemma  \ref{lemma:lipschitz_phi}), we can 
obtain the following upper bound for the first and third terms: 
 \begin{align}\l{eq:ff*_step1}
 \begin{split}
&| \la z_1, \p_z f^*(t,x_1,z_1)\ra -\la z_2, \p_z f^*(t,x_2,z_2) \ra |
  +| f^*(t,x_2,z_1)- f^*(t,x_2,z_2)|
\\
&\le | \la z_1-z_2, \p_z f^*(t,x_1,z_1)\ra| +|\la z_2, \p_z f^*(t,x_1,z_1)-\p_z f^*(t,x_2,z_2) \ra |
\\
  &\q +| f^*(t,x_2,z_1)- f^*(t,x_2,z_2)|
  \\
  &\le C(1+|x_1|+|x_2|+|z_1|+|z_2|)(|x_1-x_2|+|z_1-z_2|),
 \end{split}
 \end{align}
 where  the last inequality is by  the mean value theorem. 
 Moreover, by applying \eqref{eq:conv_max}
to $ f^*(t,x_1,z_1)$
and by
the definition of $f^*(t,x_2,z_1)$  in \eqref{eq:conjugate},
 (H.\ref{assum:lc_ns}\ref{item:f0R}), the linear growth of $\p_z f^*(t,\cdot)$,
% we have that 
 \begin{align}\l{eq:ff*_step2}
 \begin{split}
 f^*(t,x_1,z_1)- f^*(t,x_2,z_1)
&\le \la z_1, \p_z f^*(t,x_1,z_1)\ra-f(t,x_1,\p_z f^*(t,x_1,z_1))
\\
&\q -( \la z_1, \p_z f^*(t,x_1,z_1)\ra-f(t,x_2,\p_z f^*(t,x_1,z_1)))
\\
&= -f_0(t,x_1,\p_z f^*(t,x_1,z_1))+f_0(t,x_2,\p_z f^*(t,x_1,z_1))
\\
&\le C(1+|x_1|+|x_2|+|z_1|)|x_1-x_2|.
 \end{split}
 \end{align}
Then, by interchanging the roles of $x_1,x_2$ in \eqref{eq:ff*_step2} and  taking account of \eqref{eq:ff*_step1},
we can obtain the following estimate  for all $t\in [0,T]$, $(x_1,z_1),(x_2,z_2)\in \sR^n\t \sR^k$:
\begin{align*}
&| f(t,x_1,\p_z f^*(t,x_1,z_1))-f(t,x_2,\p_z f^*(t,x_2,z_2))|
\\
&\le C(1+|x_1|+|x_2|+|z_1|+|z_2|)(|x_1-x_2|+|z_1-z_2|).
\end{align*}

Therefore, by %using 
 \eqref{eq:phi_f*}, \eqref{eq:feedback}
and 
\eqref{eq:feedback_per}, 
%we have 
for all  $t\in [0,T]$ and $x,x'\in \sR^n$,
\begin{align*}
&|f(t,x,\psi(t,x))-f(t,x',\tilde{\psi}(t,x'))|
=|f(t,x,\phi(t,x,Y^{t.x}_t))-f(t,x',\tilde{\phi}(t,x',\tilde{Y}^{t.x'}_t))|
\\
&=
|f(t,x,\p_z f^* (t,x,-b_2(t)^\trans Y^{t.x}_t))
-f(t,x',\p_z f^*(t,x',-\tilde{b}_2(t)^\trans \tilde{Y}^{t.x'}_t))|
\\
&\le C(1+|x|+|x'|+|b_2(t)^\trans Y^{t.x}_t|+|\tilde{b}_2(t)^\trans \tilde{Y}^{t.x'}_t|)(|x-x'|+|b_2(t)^\trans Y^{t.x}_t-\tilde{b}_2(t)^\trans \tilde{Y}^{t.x'}_t|)
\\
&\le C(1+|x|+|x'|+\| Y^{t.x}\|_{\cS^2}+\| \tilde{Y}^{t.x'}\|_{\cS^2})
\\
 &\q \times
(|x-x'|+\|b_2-b_2\|_{L^\infty} \|Y^{t.x}\|_{\cS^2}+\| Y^{t.x}-\tilde{Y}^{t.x'}\|_{\cS^2})
\\
&\le
C(1+|x|+|x'|)
(|x-x'|+\cE_{\textrm{per}} (1+|x|)),
\end{align*}
which along with Young's inequality leads to the desired conclusion.
 \end{proof}
 
 \begin{proof}[Proof of Lemma \ref{lemma:deterministic_integral}]
It suffices to show the statement for 
processes $X,Y$ such that 
$\|X\|_{L^2(0,T)},
\|Y\|_{L^2(0,T)}\in \subW(\a)$
with $\|X\|_{L^2(0,T)}\coloneqq (\int_0^T |X|^2\, \d t)^{\frac{1}{2}}$
and $\|Y\|_{L^2(0,T)}\coloneqq (\int_0^T |Y|^2\, \d t)^{\frac{1}{2}}$,
as otherwise the right-hand side of the inequality would be infinity. 
Since $\|\cdot\|_{\Psi_{\a}}$ is a quasi-norm for any $\a>0$, we shall assume without loss of generality that 
$\|\|X\|_{L^2(0,T)}\|_{\Psi_{\a}}=\|\|Y\|_{L^2(0,T)}\|_{\Psi_{\a}}=1$.
Then, we can deduce from H\"{o}lder's inequality and Young's inequality that 
\begin{align*}
&\sE\bigg[\exp\bigg(\bigg|\int_0^T XY\, \d t\bigg|^\frac{\a}{2}\bigg)\bigg]
\le 
\sE\bigg[\exp\bigg(
\bigg|
\|X\|_{L^2(0,T)}\|Y\|_{L^2(0,T)}\bigg|^\frac{\a}{2}\bigg)\bigg]
\\
&\le \sE\bigg[\exp\bigg(\frac{1}{2}\|X\|^\a_{L^2(0,T)}+\frac{1}{2}\|Y\|^\a_{L^2(0,T)}\bigg)\bigg]
=\sE\bigg[\exp\bigg(\frac{1}{2}\|X\|^\a_{L^2(0,T)}\bigg)\exp\bigg(\frac{1}{2}\|Y\|^\a_{L^2(0,T)}\bigg)\bigg]
\\
&\le
\bigg(\sE\bigg[\exp\bigg(\|X\|^\a_{L^2(0,T)}\bigg)\bigg]\bigg)^\frac{1}{2}
\bigg(\sE\bigg[\exp\bigg(
\|Y\|^\a_{L^2(0,T)}\bigg)\bigg]\bigg)^\frac{1}{2}
\le 2,
\end{align*}
which implies that $\|\int_0^T XY\, \d t\|_{\Psi_{\a/2}}\le 1$ and finishes the proof.
\end{proof}

\begin{proof}[Proof of Lemma \ref{lemma:stochastic_integral_jump}]

Note that 
\eqref{eq:sub_weibull_lp}
and 
H\"{o}lder's inequality suggest that 
it suffices to estimate the growth of $\|\cdot\|_{L^q}$-norms 
of the stochastic integrals for $q\ge 2$.
Hence, by  
\eqref{eq:bdg_brownian},  there exists a constant $C$ such that for all $q\ge 2$,
\begin{align*}
q^{-2}\bigg\|\int_0^T X_t\sigma^\trans\, \d W_t\bigg\|_{L^q}
&\le q^{-2}Cq \bigg\|
\bigg(
\int_0^T |X_t\sigma|^2\,\d t
\bigg)^{\frac{1}{2}}\bigg\|_{L^q}
%\le CKq^{-1} \bigg\|
%\bigg(
%\int_0^T |X_t\sigma(t)|^2\,\d t
%\bigg)^{\frac{1}{2}}\bigg\|_{L^q}
\le 
 C|\sigma|\sup_{q\ge 1}
 \bigg(q^{-1} \bigg\|
\bigg(
\int_0^T |X_t|^2\,\d t
\bigg)^{\frac{1}{2}}\bigg\|_{L^q}
\bigg)
\\
&\le 
 C|\sigma| \bigg\|
\bigg(
\int_0^T |X_t|^2\,\d t
\bigg)^{\frac{1}{2}}\bigg\|_{\Psi_1},
\end{align*}
which along with  \eqref{eq:sub_weibull_lp} 
leads to the  desired estimate for 
$\|\int_0^T X_t\sigma^\trans\, \d W_t\|_{\Psi_{1/2}}$.

Similarly, by  \eqref{eq:bdg_poisson},  there exists a constant $C$ satisfying for all $q\ge 2$ that
\begin{align*}
&\bigg\|\int_0^T\int_{\sR^p_0} X_t\gamma(u)\,\tilde{N}(\d t, \d u)\bigg\|_{L^q}
\\
&\le 
Cq^{2}
\bigg\{
\bigg(
\sE\bigg[\int_0^T\int_{\sR^p_0} |X_t\gamma(u)|^{q}\,\nu(\d u)\d t 
\bigg]
\bigg)^{\frac{1}{q}}
+
\bigg(
\sE\bigg[\bigg(\int_0^T\int_{\sR^p_0} |X_t\gamma(u)|^{2}\,\nu(\d u)\d t 
\bigg)^{\frac{q}{2}}\bigg]
\bigg)^{\frac{1}{q}}
\bigg\}
\\
&\le 
Cq^{2}
\bigg\{
\bigg(
\int_{\sR^p_0} |\gamma(u)|^{q}\,\nu(\d u)
\sE\bigg[\int_0^T|X_t|^q\d t 
\bigg]
\bigg)^{\frac{1}{q}}
+
\bigg(\int_{\sR^p_0} |\gamma(u)|^{2}\,\nu(\d u)
\bigg)^{\frac{1}{2}}
\bigg(
\sE\bigg[\bigg(\int_0^T|X_t|^2\, \d t
\bigg)^{\frac{q}{2}}\bigg]
\bigg)^{\frac{1}{q}}
\bigg\}
\\
&\le 
Cq^{2}
\bigg\{
\bigg(
\int_{\sR^p_0} |\gamma(u)|^{q}\,\nu(\d u)
\bigg)^{\frac{1}{q}}
\bigg\|
\bigg(\int_0^T|X_t|^q\d t 
\bigg)^{\frac{1}{q}}
\bigg\|_{L^q}
+
\bigg(\int_{\sR^p_0} |\gamma(u)|^{2}\,\nu(\d u)
\bigg)^{\frac{1}{2}}
\bigg\|\bigg(\int_0^T|X_t|^2\, \d t 
\bigg)^{\frac{1}{2}}\bigg\|_{L^q}
\bigg\}.
\end{align*}
Hence by   (H.\ref{assum:lc_rl}\ref{item:jump}) 
and
\eqref{eq:sub_weibull_lp},
 for all $q\ge 2$,
\begin{align*}
&q^{-(3+\vartheta)}\bigg\|\int_0^T\int_{\sR^p_0} X_t\gamma(u)\,\tilde{N}(\d t, \d u)\bigg\|_{L^q}
\\
&\le 
C
\bigg(
\sup_{q\ge 2}
q^{-\vartheta}
\bigg(
\int_{\sR^p_0} |\gamma(u)|^{q}\,\nu(\d u)
\bigg)^{\frac{1}{q}}
\bigg)
\bigg\{
q^{-1}\bigg\|
\bigg(\int_0^T|X_t|^q\d t 
\bigg)^{\frac{1}{q}}
\bigg\|_{L^q}
+
q^{-1}
\bigg\|\bigg(\int_0^T|X_t|^2\, \d t 
\bigg)^{\frac{1}{2}}\bigg\|_{L^q}
\bigg\}
\\
&\le 
C
\gamma_{\max}
\bigg(
\bigg\|
\bigg(\int_0^T|X_t|^q\d t 
\bigg)^{\frac{1}{q}}
\bigg\|_{\Psi_1}
+
\bigg\|\bigg(\int_0^T|X_t|^2\, \d t 
\bigg)^{\frac{1}{2}}\bigg\|_{\Psi_1}
\bigg).
\end{align*}
Therefore,
 taking the
supremum over $q\ge 2$ in the above inequality
leads to 
the desired estimate 
of
$\|\int_0^T\int_{\sR^p_0} X_t\gamma(u)\,\tilde{N}(\d t, \d u)\|_{\Psi_{1/(3+\vartheta)}}$
from \eqref{eq:sub_weibull_lp}.
\end{proof}

\begin{proof}[Proof of Lemma \ref{lemma:stochatic_integral_diffusion}]

Let us assume without loss of generality that
$|\sigma|>0$ and 
  $\tau\coloneqq \|(\int_0^T |X_t|^2\, \d t)^{1/2}\|_{\Psi_2}<\infty$,
which 
implies that 
$\|(\int_0^T 2|X_t\sigma|^2\, \d t)^{1/2}\|_{\Psi_2}\le \sqrt{2}|\sigma|\tau$.
Then, by  the characterization of sub-Gaussian random variable in \cite[Proposition 2.5.2(iii)]{vershynin2018high},
there exists  $C\ge 0$ such that 
$$
\sE\left[\exp\bigg(2\lambda^2 \int_0^T |X_t\sigma|^2\, \d t\bigg)\right]
\leq \exp(2C^2\lambda^2|\sigma|^2\tau^2)<\infty
\q \fa 
|\lambda|\le \frac{1}{{\sqrt{2}C|\sigma|\tau}}.
$$
Hence, 
it holds for all $|\lambda|\le 1/({\sqrt{2}C|\sigma|\tau})$ that 
the  process
 $(M_{\lambda,t})_{t\in [0,T]}$ 
 defined by:
 $$
M_{\lambda,t} \coloneqq \exp\left(\int_0^t 2\lambda X_s \sigma^\trans \d W_s-\frac{1}{2} \int_0^t 4\lambda^2 |X_s \sigma|^2 \d s\right)
\q \fa t\in [0,T]
$$
  is a martingale,
   since 
 Novikov's condition is satisfied, 
which 
implies that $\sE[M_{\lambda,T}]=1$.
Thus, 
for any given $|\lambda|\le 1/(\sqrt{2}C|\sigma|\tau)$,
by the Cauchy-Schwarz inequality,
\begin{align*}
&\sE\bigg[\exp\bigg(\lambda\int_0^T  X_t \sigma^\trans \d W_t\bigg)\bigg]
\\
&=
\sE\bigg[\exp\bigg(\int_0^T  \lambda X_t \sigma^\trans \d W_t
-\frac{(2\lambda)^2}{4} \int_0^T |X_t \sigma|^2\,\d t\bigg)
\exp\bigg(\frac{(2\lambda)^2}{4} \int_0^T  |X_t \sigma|^2\,\d t\bigg)\bigg]
\\
&\le 
\sE[M_{\lambda,T}]^{1/2}\sE\bigg[\exp\bigg(2{\lambda^2}\int_0^T  |X_t \sigma|^2\,\d t\bigg)\bigg]^{1/2}
\le   \exp(C^2\lambda^2|\sigma|^2\tau^2),
\end{align*}
which along with the fact that 
$\sE[\int_0^T  X_t \sigma^\trans \d W_t]=0$ and 
 the characterization of sub-exponential random variable
\cite[Proposition 2.7.1(v)]{vershynin2018high} yields that 
$\|\int_0^T  X_t \sigma^\trans \d W_t\|_{\Psi_1}\leq C|\sigma|\tau$.
\end{proof}

\begin{proof}[Proof of Lemma \ref{lemma:state_subexp}]
Throughout this proof, let 
$x_0\in \sR^n$ and $\theta\in \sR^{n\t (n+k)}$ 
be given constants satisfying 
$|\theta|\le K$,
and let $C$ be a generic constant
depend on $K$, $T$ and the constants 
in 
(H.\ref{assum:lc_rl}),
%(H.\ref{assum:lc_rl}\ref{item:fg}\ref{item:jump}),
but independent of $x_0$ and $\theta$.
 
By  \eqref{eq:lc_sde_theta},
we see that   the process $X^{x_0,\theta}$ satisfies the SDE:
$$
\d X_t =b^\theta(t,X_t)\,\d t+\sigma\, \d W_t
+\int_{\sR^p_0}
\gamma(u)\, \tilde{N}(\d t,\d u),
 \q t\in [0,T],
\q X_0=x_0,
$$
where 
$b^\theta(t,x)=A^\star x+B^\star{\psi}^\theta(t,x)$ for all $(t, x)\in [0,T]\t \sR^n$.
The definition of the feedback control $\psi^\theta$,
(H.\ref{assum:lc_rl}\ref{item:fg}) 
and 
Theorem  \ref{thm:lc_fb}
show  that 
there exists  $C\ge 0$
such that 
$|\psi^\theta(t,0)|\le C$ and $|\psi^\theta(t,x)-\psi^\theta(t,x')|\le C|x-x'|$
  for all  $t\in [0,T]$, $x,x'\in \sR^n$,
  which implies the same properties for the function $b^\theta$.
  Then, by  Lemma \ref{lemma:sde_transportation}, 
  for 
 every 
  Lipschitz continuous function
$\mathfrak{f}:(\sD([0,T];\sR^n),d_\infty)\to \sR$, 
$\sE\big[\exp\big(\lambda(\mathfrak{f}(X^{x_0,\theta})-\sE[\mathfrak{f}(X^{x_0,\theta})])\big)\big]
\le 
\exp\big({C\eta\big(C\lambda \|\mathfrak{f}\|_{\textnormal{Lip}}\big)}\big)
$
 for all $\lambda>0$,
with the function 
$\eta:[0,\infty]\to [0,\infty]$ defined by:
\bb\l{eq:eta_subexp}
\eta(\lambda) \coloneqq
 \int_{\sR^p_0} \left(e^{\lambda {\gamma}(u)}-\lambda {\gamma}(u)-1\right)\,\nu(\d u)+\frac{\sigma^2}{2}\lambda^2
\q \fa \lambda> 0.
\ee
By   (H.\ref{assum:lc_rl}\ref{item:jump})
 and 
 Stirling's approximation
 $q !\ge (q/e)^q$ for all $q\ge 2$,
we have  for each $\lambda\in [0,1/(2\gamma_{\max}e))$,
 \begin{align*}
 &\int_{\sR^p_0} \left(e^{\lambda {\gamma}(u)}-\lambda {\gamma}(u)-1\right)\,\nu(\d u)
 \\
&=\int_{\sR^p_0}\sum_{q= 2}^\infty\frac{|\lambda {\gamma}(u)|^q}{q !}\, \nu (\d u)
=\sum_{q= 2}^\infty\frac{\lambda ^q}{q !}\int_{\sR^p_0}|{\gamma}(u)|^q \, \nu (\d u)
\le \sum_{q= 2}^\infty\frac{\lambda ^q}{q !}\gamma_{\max}^qq^{\vartheta q}
\\
&\le \sum_{q= 2}^\infty\frac{(\lambda \gamma_{\max} e) ^q}{q^{(1-\vartheta)q}}
\le \frac{(\lambda \gamma_{\max} e)^2}{1-\lambda \gamma_{\max} e}
% \le (\lambda \gamma_{\max} e)^2(1+\lambda \gamma_{\max} e)
\le 2(\lambda \gamma_{\max} e)^2,
 \end{align*}
 which implies 
 for all $0\le \lambda \le 1/C$
 and $\mathfrak{f}:(\sD([0,T];\sR^n),d_\infty)\to \sR$
satisfying $\|\mathfrak{f}\|_{\textnormal{Lip}}\le 1$
  that 
 $\sE\big[\exp\big(\lambda(\mathfrak{f}(X^{x_0,\theta})-\sE[\mathfrak{f}(X^{x_0,\theta})])\big)\big]
\le 
\exp(C^2\lambda^2).
$
Replacing $\mathfrak{f}$  with $-\mathfrak{f}$ shows that 
the same estimate holds for all 
for all $|\lambda| \le 1/C$,
which,
along with the characterization of sub-exponential random variable in 
 \cite[Proposition 2.7.1(v)]{vershynin2018high}, 
leads to
$\|\mathfrak{f}(X^{x_0,\theta})-\sE[\mathfrak{f}(X^{x_0,\theta})]\|_{\Psi_1}\le C$
for some constant $C$,
uniformly with respect to $x_0\in \sR^n$, $|\theta|\le K$ and 
$\mathfrak{f}:(\sD([0,T];\sR^n),d_\infty)\to \sR$
satisfying $\|\mathfrak{f}\|_{\textnormal{Lip}}\le 1$.

Since $\|\cdot\|_{\Psi_1}$ is a norm
and $\|\sE[\mathfrak{f}(X^x)]\|_{\Psi_1}\le |\sE[\mathfrak{f}(X^x)]|/\ln 2$,
%we have that 
$\|\mathfrak{f}(X^{x_0,\theta})\|_{\Psi_1}\le C(1+|\sE[\mathfrak{f}(X^{x_0,\theta})]|)$
for all $\mathfrak{f}$ with $\|\mathfrak{f}\|_{\textnormal{Lip}}\le 1$.
The estimate for a general Lipschitz continuous 
function $\mathfrak{f}$ 
follows by considering 
$\mathfrak{f}/\|\mathfrak{f}\|_{\textnormal{Lip}}$
and by using the fact that 
$\|\cdot\|_{\Psi_1}$ is a norm.
\end{proof}

\end{appendices}

\bibliographystyle{siam}

\bibliography{lc_rl.bib}

\end{document}